\documentclass{article}
\usepackage[utf8]{inputenc}
\usepackage{graphicx}
\usepackage{amsthm,amssymb,amsmath}
\usepackage[colorlinks=true]{hyperref}
\usepackage{asymptote}
\usepackage{tabularx}
\usepackage{xcolor}
\newcommand\setrow[1]{\gdef\rowmac{#1}#1\ignorespaces}
\newcommand\clearrow{\global\let\rowmac\relax}
\clearrow

\renewcommand{\SS}{\mathcal{S}}

\newcommand{\PP}{\mathcal{P}}
\newcommand{\QQ}{\mathcal{Q}}
\newcommand{\id}{\mathrm{id}}
\newcommand{\dom}{\mathrm{dom}}
\newcommand{\roots}[2]{#1_{\downarrow #2}}
\newcommand{\rootsrem}[3]{\roots{(#1-#3)}{#2\setminus #3}}

\newtheorem{theorem}{Theorem}
\newtheorem{lemma}[theorem]{Lemma}
\newtheorem{corollary}[theorem]{Corollary}

\newtheorem{observation}[theorem]{Observation}

\newtheorem{claim}{Claim}
\newenvironment{subproof}{%
  \begin{proof}[Proof of claim]%
}{%
  \end{proof}%
}

\begin{asydef}
usepackage("amsmath");
usepackage("amssymb");
usepackage("xcolor");
unitsize(9mm);

DefaultHead.size=new real(pen p=currentpen) {return 2mm;};

void vertex(pair a, pen barva=white, real pol = 0.1)
{
  filldraw(circle (a, pol), fillpen=barva);
}

pair v[];
int i,j;
\end{asydef}

\title{Extremal function for rooted $K_5$ minors}
\author{Zdeněk Dvořák\thanks{Charles University, Prague, Czech Republic,
	\url{rakdver@iuuk.mff.cuni.cz}.
	Supported by ERC-CZ project LL2328 (Beyond the Four Color Theorem) of the Ministry of Education of Czech Republic.}
}

\begin{document}

\maketitle

\begin{abstract}
We show that if an $n$-vertex 5-connected graph has at least $4n-10$ edges, then for any choice
of five of its vertices, we can contract disjoint connected subgraphs containing these vertices
to obtain $K_5$ as a minor.  The bound on the number of edges is the best possible.
\end{abstract}

\section{Introduction}

A famous conjecture of Hadwiger~\cite{hadwiger} states that for every positive integer $k$,
if a graph does not contain $K_{k+1}$ as a minor\footnote{For graphs $H$ and $G$, a \emph{model} of $H$ in $G$ is a function $\mu$ that
\begin{itemize}
\item maps vertices of $H$ to pairwise disjoint non-empty subsets of vertices of $G$ inducing connected subgraphs of $G$, and
\item maps each edge $e=uv$ of $H$ to an edge $\mu(e)$ of $G$ with one end in the set $\mu(u)$ and the other end in the set $\mu(v)$.
\end{itemize}
If there exists a model of $H$ in $G$, we say that $H$ is a \emph{minor} of $G$.
Equivalently, $H$ is a minor of $G$ if a graph isomorphic to $H$ can be obtained from a subgraph of $G$ by contracting edges.},
then it is $k$-colorable.
Despite a lot of attention, this conjecture is wide open, with the only resolved
cases being $k\le 3$ (simple, proved already by Hadwiger~\cite{hadwiger}),
$k=4$ (equivalent to the Four Color Theorem by Wagner's characterization of $K_5$-minor-free graphs~\cite{wagner}),
and $k=5$ (equivalent to the Four Color Theorem by a tour de force argument of Robertson, Seymour, and Thomas~\cite{robertsonseymourthomas}).
For the next open case $k=6$, even the question whether all $K_7$-minor-free graphs are $7$-colorable (rather than $6$-colorable,
as postulated by Hadwiger's conjecture) is open!

On the other hand, it is at least known that Hadwiger's conjecture is approximately true,
in the sense that every $K_k$-minor-free graph has chromatic number bounded by a function of $k$.
Indeed, more is true: Every $K_k$-minor-free graph has average degree $O(k\sqrt{\log k})$, as shown by Kostochka~\cite{kostochka1984lower}
and Thomason~\cite{thomason1984lower}, and thus it also has chromatic number $O(k\sqrt{\log k})$.  For a long time, this
density-based bound was the best general partial result towards Hadwiger's conjecture. This barrier was eventually broken
by Norin, Postle, and Song~\cite{norin2023breaking}.  Subsequently, Delcourt and Postle~\cite{delcourt2025reducing}
refined the approach to improve the bound on the chromatic number of $K_k$-minor-free graphs to $O(k\log\log k)$,
and very recently Liu and Luo~\cite{lllhadwiger} further improved the bound to $O(k\log \log\log k)$,

Moreover, the density bounds form the natural starting point for approaching specific cases of Hadwiger's conjecture as well as its variations
with other forbidden minors.  The exact density bounds on $K_k$-minor-free graphs for small values of $k$ were first
systematically investigated by Mader~\cite{maderdens}, who proved that for $k\in\{2,\ldots,7\}$, every $n$-vertex $K_k$-minor-free
graph (where $n\ge k-1$) has at most
\begin{equation}\label{eq-mader}
(k-2)n-\binom{k-1}{2}
\end{equation}
edges.  This bound does not hold for $k=8$, but J{\o}rgensen~\cite{jorgensen1994contractions} exactly characterized the counterexamples
(an infinite family of $K_8$-minor-free graphs with exactly one more edge, starting with $K_{2,2,2,2,2}$).
Song and Thomas~\cite{SONG2006240} further extended the result to $K_9$-minor-free graphs by exactly characterizing
those with more edges than predicted by (\ref{eq-mader}).

In this paper, we consider an analogous question for rooted minors.  Let us start with the definitions.
A \emph{rooted graph} is an undirected simple graph $G$ with a specified
set $X_G\subseteq V(G)$ of \emph{roots}.  If $|X_G|=k$, then we say that $G$
is a \emph{$k$-rooted graph}.  
For a graph $F$ without roots and a set $Z\subseteq V(F)$,
let $\roots{F}{Z}$ denote the rooted graph with the underlying graph $F$ and with the root set $Z$.

Consider a graph $H$ (without roots), a set $Y\subseteq V(H)$ of size $|X_G|$, and a bijection $\pi:Y\to X_G$.
A \emph{$\pi$-rooted model} of $H$ in $G$ is a model $\mu$ of $H$ in $G$ such that $\pi(v)\in \mu(v)$ holds for every $v\in \dom(\pi)=Y$.
If such a model exists, then we say that $H$ is a \emph{$\pi$-rooted minor} of $G$.
Equivalently, a graph isomorphic to $H$ via an isomorphism extending $\pi$ is obtained from $G$ by deleting edges,
deleting non-root vertices, and contracting edges with at least one non-root end.
We also use these definitions in the case that $H$ is a rooted graph, in which case we require that $\dom(\pi)=X_H$
(and in particular $|X_H|=|X_G|$).
When $H=K_{|X_G|}$ is a clique of size $|X_G|$, then we say that $H$ is a \emph{rooted minor} of $G$ if $H$ is a $\pi$-rooted minor
of $G$ for any bijection $\pi:V(H)\to X_G$ (here, the exact bijection of course does not matter).
In case that $X_G\subseteq V(H)$, we usually map the roots to themselves; more precisely, we let $\id$ denote the partial
function $V(H)\to X_G$ such that $\dom(\id)=X_G$ and $\id(x)=x$ for every $x\in X_G$, and we refer to $H$ being an \emph{$\id$-rooted minor} of $G$.
Let us remark that in the case we apply this definition for a rooted graph $H$, we require that $X_H=X_G$.

Rooted minors naturally arise in the context of Hadwiger's conjecture when we consider separations\footnote{A \emph{separation} of a (rooted) graph
$G$ is a pair $(A,B)$ of subsets of $V(G)$ such that $V(G)=A\cup B$ and
$G=G[A]\cup G[B]$ (and thus $G$ does not have any edges between $A\setminus B$ and $B\setminus A$).
The \emph{order} of the separation $(A,B)$ is $|A\cap B|$, and
the separation is \emph{proper} if $A\not\subseteq B$ and $B\not\subseteq A$.}
in $K_k$-minor-free graphs.  For instance, one of the steps in the proof of Hadwiger's conjecture for $K_6$-minor-free graphs~\cite{robertsonseymourthomas}
was to restrict separations $(A,B)$ of order six in a (hypothetical) smallest counterexample $G$.
To do so, the authors needed several statements of form ``under certain conditions, $\roots{(G[A])}{(A\cap B)}$ contains $H$ as an $\id$-rooted minor''
for suitable choices of graphs $H$ with vertex set $A\cap B$.  This then leads to a contradiction
by showing that the rooted models in $\roots{(G[A])}{(A\cap B)}$ and in $\roots{(G[B])}{(A\cap B)}$ combine to a model of $K_6$ in $G$.

The theory of rooted minors is substantially less developed (and even more challenging) than for the unrooted case.
It is fairly easy to characterize $k$-rooted graphs avoiding $K_k$ as a rooted minor for $k\le 3$, but beyond that
the complexity quickly rises.  As part of their graph minors series, Robertson and Seymour~\cite{rs9}
characterized 4-rooted graphs $G$ avoiding $2K_2$ as a $\pi$-rooted minor for a fixed bijection $\pi:V(2K_2)\to X_G$.
Building on this result, Robertson, Seymour, and Thomas~\cite{robertsonseymourthomas} then developed structural descriptions
for all other 4-vertex graphs (for an exact characterization in the case of rooted $K_4$ minor,
see also Fabila-Monroy and Wood~\cite{fabila2013rooted}).  However, obtaining exact characterizations for
larger forbidden rooted minors (and in particular, for 5-rooted graphs avoiding $K_5$ as a rooted minor)
seems far beyond our means.

However, let us note that in the unrooted setting, we are similarly far from being able to give an exact characterization for $K_6$-minor-free
graphs, yet we know the extremal density all the way up to $K_9$-minor-free graphs.  By analogy, could the density
problems for graphs avoiding rooted minors be more feasible?  Before we discuss the previous results on this topic,
let us mention one caveat: For every $k\ge 2$, we can easily construct $k$-rooted graphs of arbitrarily large density which
avoid $K_k$ as a rooted minor.  Indeed, we can increase the density arbitrarily by adding a large dense component not containing
any of the roots.  To avoid this technicality, we need to restrict ourselves to graphs of sufficiently large connectivity.

In general, Wollan~\cite{wollan2008extremal} proved that the extremal functions for the rooted and unrooted case differ
at most by a constant factor.
\begin{theorem}[Wollan~\cite{wollan2008extremal}]
Let $H$ be a graph and let $c\ge 1$ be a real number such that every $n$-vertex graph with at least $cn$ edges contains $H$
as a minor, and let $G$ be a rooted graph with $X_G=V(H)$.  If $G$ is $|V(H)|$-connected and has at least
$(9c+26833|V(H)|)\cdot|V(G)|$ edges, then $H$ is an $\id$-rooted minor of $G$.
\end{theorem}

Of course, we can determine the extremal density exactly in the cases when an exact characterization is known for a given rooted
minor.  For instance, the result of Fabila-Monroy and Wood~\cite{fabila2013rooted} easily implies that if an $n$-vertex 4-connected
4-rooted graph avoids $K_4$ as a rooted minor, then it has at most $3n-7$ edges; see also~\cite{norin2025every} for a
derivation of this bound from the weaker description by Robertson, Seymour, and Thomas~\cite{robertsonseymourthomas}.

Another setting where the density function is well understood is for complete bipartite graphs with one side of size two,
where the two vertices in this side are not rooted.  More precisely, Wollan~\cite{wollan2008extremal} proved
the following: Let $G$ be an $n$-vertex $t$-connected $t$-rooted graph, and let $\pi$ be a bijection
between the vertices of $K_{t,2}$ in the part of size $t$ and $X_G$.  If $G$ does not contain $K_{t,2}$ as
a $\pi$-rooted minor, then $|E(G)|\le tn-\binom{t+1}{2}$; and this bound is the best possible.  Norin and Totschnig~\cite{norin2025every} proved
that for $t=4$, the same bound holds for the graph obtained from $K_{4,2}$ by adding an edge between the two vertices of degree four.
J{\o}rgensen and Kawarabayashi~\cite{jorgensen2007extremal} studied the analogous problem for $K_{t,3}$
and obtained bounds for the cases $t\in\{2,3,4\}$ (tight for $t=2$ and likely not tight for $t\in\{3,4\}$).

Finally, let us remark on a connection to another well studied subject.
A graph $G$ with at least $2k$ vertices is \emph{$k$-linked} if for every choice of distinct vertices $s_1,\ldots,s_k,t_1,\ldots,t_k\in V(G)$,
there exist pairwise disjoint paths in $G$ joining $s_1$ to $t_1$, $s_2$ to $t_2$, \ldots, and $s_k$ to $t_k$.
In other words, this is the case if for every set $X\subseteq V(G)$ of size $2k$ and for every bijection $\pi:V(kK_2)\to X$,
the $2k$-rooted graph $\roots{G}{X}$ contains $kK_2$ as a $\pi$-rooted minor.  Thus, the result of Thomas and Wollan~\cite{tomlinear}
on linkedness of graphs can be restated as saying that every $2k$-connected $2k$-rooted $n$-vertex graph with at least
$5kn$ edges contains $kK_2$ as a $\pi$-rooted minor for every bijection $\pi$ between $V(kK_2)$ and the roots.

In this paper, we push our understanding of the extremal functions for rooted clique minors one step further,
by determining it exactly for $K_5$.  Let us first remark that at least $4n-10$ edges are needed to force $K_5$
as a rooted minor in a 6-connected 5-rooted $n$-vertex graph.  Indeed, let $G_0$ be any 5-connected plane graph
with $n-1$ vertices and with all faces of length three except for a single face bounded by a 4-cycle $x_1x_2x_3x_4$.
Thus, $|E(G_0)|=3|V(G_0)|-7=3n-10$.  Moreover, note that $\roots{(G_0)}{\{x_1,\ldots,x_4\}}$ does not contain $K_4$
as a rooted minor.  Let $G$ be the 6-connected $n$-vertex graph with $4n-11$ edges obtained from $G_0$ by adding a universal vertex $x_5$,
and let $X=\{x_1,\ldots,x_5\}$.  Then $\roots{G}{X}$ does not contain $K_5$ as a rooted minor.
As our main result, we show that one extra edge forces $K_5$ as a rooted minor, even in $5$-connected graphs.

\begin{theorem}\label{thm-allK5}
Let $G$ be a 5-rooted graph.  If $G$ is 5-connected and has at least $4|V(G)|-10$ edges, then $K_5$ is a rooted minor of $G$.
\end{theorem}
Let us remark that Theorem~\ref{thm-allK5} strengthens a recent result of Du, Li, and Yu~\cite{du2025rooted}, who gave a worse bound for
rooted $C_5$ minors.  As an illustration of the usefulness of Theorem~\ref{thm-allK5}, let us note that using it together with further tools
inspired by this paper, we were able to prove that for $n\ge 6$, every $n$-vertex $5$-connected graph with at least $4n-7$ edges contains $K_7^{=}$ as a minor~\cite{k7mm},
where $K_7^{=}$ is the graph obtained from $K_7$ by removing two edges not incident with the same vertex.  This then easily
implies that $K_7^{=}$-minor-free graphs are $6$-colorable, resolving an open problem by Norin and Totschnig~\cite{norin2025every}.

\subsection{AI usage declaration}

AI was only used for proofreading.

\section{Generalization and proof ideas}

Let us for contradiction consider a smallest counterexample $G$ to Theorem~\ref{thm-allK5} (say one with fewest vertices).
We aim to use the following standard idea: Suppose we contract an edge $e$ between vertices $u,v\in V(G)\setminus X_G$.
The resulting 5-rooted graph $G'$ does not contain $K_5$ as a rooted minor.  Since $G'$ is smaller than $G$,
it is not a counterexample, and thus if it is 5-connected, it must have at most $4|V(G')|-11$ edges.  This means
that the edge $e$ is contained in at least 4 triangles in $G$, or equivalently, the vertex $u$ has degree at least four
in the subgraph of $G$ induced by the neighbors of $v$.  If this happened to hold for all edges $e$ of $G$,
it is fairly straightforward to check that we can obtain $K_5$ as a rooted minor of $G$ by contracting paths between $X_G$
and some five neighbors of $v$ together with a suitably chosen subgraph of the closed neighborhood of $v$
(some complications arise when the neighborhood of $v$ intersects $X_G$, but they are manageable).

Thus, the main issue with this argument is that contracting an edge of $G$ might result in a graph which is not 5-connected.
To deal with this problem, we are going to relax the assumption that $G$ is 5-connected: We will allow $(\le\!4)$-cuts,
but only if the subgraph that they cut off is sparse (this idea was inspired by a similar approach of 
Wollan~\cite{wollan2008extremal}).  Let us give a few additional definitions to be able to state this more
general statement exactly.

For a rooted graph $G$, let $\widetilde{G}$ denote the rooted graph obtained from $G$ by removing
all edges between the roots.  We say that a separation $(A,B)$ of the underlying graph of $G$ is
a \emph{root separation} of $G$ if $X_G\subseteq A$.
A rooted graph is \emph{internally $m$-connected} if it has no proper root separation of order less than $m$.
The \emph{left-hand side} $L_{A,B}$ of a root separation $(A,B)$ of $G$ is the rooted graph $\roots{G[A]}{X_G}$,
and the \emph{right-hand side} $R_{A,B}$ is the rooted graph $\roots{G[B]}{(A\cap B)}$
(we also refer to these rooted graphs as $L^G_{A,B}$ and $R^G_{A,B}$
when the rooted graph $G$ is not clear from the context).
We also often consider just the \emph{strictly right-hand side} of the separation, defined as $\widetilde{R}_{A,B}$.
For a rooted graph $G$,
\begin{itemize}
\item let $\rho(G)=|E(\widetilde{G})|$ be the number of edges incident with at least one non-root vertex, and
\item let $n(G)=|V(G)|-|X_G|$ be the number of non-root vertices of $G$.
\end{itemize}
For a positive real number $t$, we define the \emph{$t$-density} of $G$ as
$$\rho_t(G)=\rho(G)-t\cdot n(G);$$
note that $\rho_t(\widetilde{G})=\rho_t(G)$.
We say that a root separation $(A,B)$ is \emph{$t$-light} if $\rho_t(R_{A,B})\le 0$.
A $k$-rooted graph $G$ is \emph{$t$-light} if every root separation of $G$ of order
less than $k$ is $t$-light.  Our goal will now be to prove the following stronger version of Theorem~\ref{thm-allK5} where
we relax the 5-connectedness assumption.
\begin{theorem}\label{thm-allK5-better}
Let $G$ be a 5-rooted graph.  If $G$ is 4-light and has at least $4|V(G)|-10$ edges, then $K_5$ is a rooted minor of $G$.
\end{theorem}
Let us now come back to the proof idea.  The relaxed assumption of 4-lightness is easier to maintain than 5-connectivity,
but we still need to overcome significant issues.  Contracting an edge can now violate the 4-lightness in two ways:
We could increase the 4-density of the right side of a previously existing root separation of order at most four,
or we could turn a root separation $(A,B)$ of order five with too dense right side into a separation of order four.
The basic idea for dealing with the latter issue is that we apply Theorem~\ref{thm-allK5-better}
inductively to $R_{A,B}$, showing that we can contract it to a clique and eliminate the offending root separation.
However, it turns out that Theorem~\ref{thm-allK5-better} is not detailed enough for this purpose;
we are going to need to understand rooted minors even when $\rho_4(R_{A,B})$ is not large enough to force the appearance of $K_5$
as a rooted minor.  Hence, we are going to prove a further technical refinement of Theorem~\ref{thm-allK5-better}.

Let $k$ be a positive integer and suppose that $\SS$ is a class of graphs on $k$ vertices.
We say that the $k$-rooted graph $G$ is \emph{$\SS$-universal}
if for every graph $S\in \SS$ and for every bijection $\pi:V(S)\to X_G$, the graph $S$ is a $\pi$-rooted minor of $G$.
When $\SS$ consists just of a single graph $S$, then we say that $G$ is $S$-universal instead of $\{S\}$-universal.
The following classes of graphs will be important in formulating the main result: For integers $k$ and $t$, let $\SS_{k,t}$
denote the class of all graphs on $k$ vertices with at most $t$ edges.  Moreover, let $\SS_{5,4}^-\subset\SS_{5,4}$ denote the class of
all graphs on five vertices with at most four edges except for those isomorphic to $K_2+K_3$ (the disjoint union of $K_2$ and $K_3$).
For an integer $m$, the \emph{$m$-target} is the class
\begin{itemize}
\item $\SS_{5,10}$ if $m\ge 7$,
\item $\SS_{5,9}$ if $m=6$,
\item $\SS_{5,8}$ if $m=5$,
\item $\SS_{5,6}$ if $m=4$,
\item $\SS_{5,4}^-$ if $m=3$,
\item $\SS_{5,3}$ if $m=2$,
\item $\SS_{5,1}$ if $m=1$, and
\item $\SS_{5,0}$ if $m\le 0$.
\end{itemize}
For a $5$-rooted graph $G$, the \emph{target} of $G$ is defined as the $\rho_4(G)$-target.
We say that a $5$-rooted graph $G$ is \emph{universal} if it is $\SS$-universal for the target $\SS$ of $G$.
Our main result can now be stated as follows.
\begin{theorem}\label{thm-mainplus}
Every $4$-light $5$-rooted graph is universal.
\end{theorem}

\begin{figure}
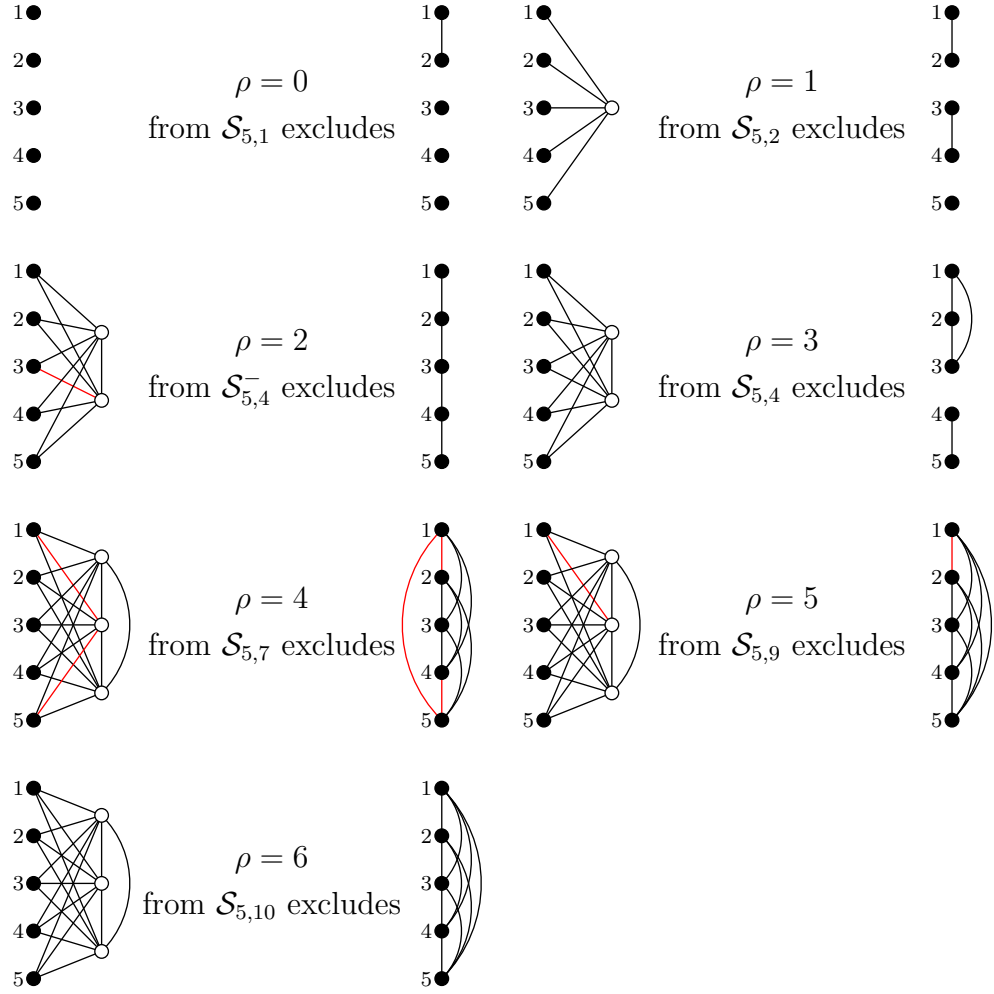

\begin{asy}
real vsz=0.7, rt = 6;

struct loc
{
  pair at;
  bool left;
};

void k5(pair reds[])
{
  for (int a = 0; a < 4; ++a)
    for (int b = a + 1; b < 5; ++b)
      {
        pair d1, d2;
	pen col;
	
	if (find(reds == (a,b)) >= 0)
	  {
	    d1 = dir(-135);
	    d2 = dir(-45);
	    col = red;
	  }
	else
	  {
	    d1 = dir(-45);
	    d2 = dir(-135);
	    col = black;
	  }

	if (b == a + 1)
	  d1 = d2 = dir (-90);

	draw (v[5+a]{d1} .. {d2}v[5+b], col);
      }
}

void mkvs ()
{
  for (int k = 0; k < 10; ++k)
    vertex (v[k], black);
}

void ram(loc where, int rho, string ex)
{
  for (int k = 0; k < 5; ++k)
    {
      v[k] = where.at + (0,-k * vsz);
      v[5 + k] = where.at + (rt,-k * vsz);
      label ("\scriptsize$" + string(k+1) + "$", v[k], W);
      label ("\scriptsize$" + string(k+1) + "$", v[5+k], W);
    }

  label ("$\rho_4=" + string(rho) + "$", where.at + (rt - 2.5,-1.5vsz));
  label ("from $" + ex + "$ excludes", where.at + (rt - 2.5,-2.5vsz));

  if (where.left)
    {
      where.at += (rt+1.5,0);
      where.left = false;
    }
  else
    {
      where.at += (-(rt+1.5),-(4*vsz + 1));
      where.left = true;
    }
}

loc w;
w.at = (0,0);
w.left = true;

ram (w, 0, "\mathcal{S}_{5,1}");
draw (v[5] -- v[6]);
mkvs ();

ram (w, 1, "\mathcal{S}_{5,2}");
v[10] = v[2] + (1,0);
for (i = 0; i < 5; ++i)
  draw (v[i] -- v[10]);
vertex(v[10]);
draw(v[5]--v[6]);
draw(v[7]--v[8]);
mkvs ();

ram (w, 2, "\mathcal{S}_{5,4}^-");
v[10] = v[2] + (1,-0.5);
v[11] = v[2] + (1,0.5);
for (i = 0; i < 5; ++i)
  {
    draw (v[i] -- v[11]);
    draw (v[i] -- v[10], i == 2 ? red : black);
  }
draw(v[10]--v[11]);
vertex(v[10]);
vertex(v[11]);
draw(v[5]--v[6]--v[7]--v[8]--v[9]);
mkvs ();

ram (w, 3, "\mathcal{S}_{5,4}");
v[10] = v[2] + (1,-0.5);
v[11] = v[2] + (1,0.5);
for (i = 0; i < 5; ++i)
  draw (v[11] -- v[i] -- v[10]);
draw (v[10] -- v[11]);
vertex(v[10]);
vertex(v[11]);
draw(v[5]--v[6]--v[7]);
draw(v[5]{dir(-45)}..v[7]{dir(-135)});
draw(v[8]--v[9]);
mkvs ();

ram (w, 4, "\mathcal{S}_{5,7}");
v[10] = v[2] + (1,-1);
v[11] = v[2] + (1,0);
v[12] = v[2] + (1,1);
for (i = 0; i < 5; ++i)
  {
    draw (v[i] -- v[10]);
    draw (v[i] -- v[11], (i == 0 || i == 4) ? red : black);
    draw (v[i] -- v[12]);
  }
draw (v[10] -- v[11]--v[12]{dir(-45)}..{dir(-135)}cycle);
vertex(v[10]);
vertex(v[11]);
vertex(v[12]);
k5(new pair[]{(0,1),(0,4),(3,4)});
mkvs ();

ram (w, 5, "\mathcal{S}_{5,9}");
v[10] = v[2] + (1,-1);
v[11] = v[2] + (1,0);
v[12] = v[2] + (1,1);
for (i = 0; i < 5; ++i)
  {
    draw (v[i] -- v[10]);
    draw (v[i] -- v[11], (i == 0) ? red : black);
    draw (v[i] -- v[12]);
  }
draw (v[10] -- v[11]--v[12]{dir(-45)}..{dir(-135)}cycle);
vertex(v[10]);
vertex(v[11]);
vertex(v[12]);
k5(new pair[]{(0,1)});
mkvs ();

ram (w, 6, "\mathcal{S}_{5,10}");
v[10] = v[2] + (1,-1);
v[11] = v[2] + (1,0);
v[12] = v[2] + (1,1);
for (i = 0; i < 5; ++i)
  {
    draw (v[i] -- v[10]);
    draw (v[i] -- v[11]);
    draw (v[i] -- v[12]);
  }
draw (v[10] -- v[11]--v[12]{dir(-45)}..{dir(-135)}cycle);
vertex(v[10]);
vertex(v[11]);
vertex(v[12]);
k5(new pair[]{});
mkvs ();

\end{asy}
\caption{Tight graphs for Theorem~\ref{thm-mainplus}; red lines are used to emphasise important non-edges.}\label{fig-addittight}
\end{figure}

Let us remark that Theorem~\ref{thm-mainplus} is tight in the sense that we cannot increase the number of the edges of the rooted minors
in any of the cases in the definition of universality, as shown by the graphs depicted in Figure~\ref{fig-addittight}.
To see that Theorem~\ref{thm-mainplus} implies Theorem~\ref{thm-allK5-better}, let us note the following fact
shown by an inspection of the definition of the $m$-target.
\begin{observation}\label{obs-canaddm}
For every non-negative integer $m\le 10$, the $m$-target is a superset of $\SS_{5,m}$.
\end{observation}
Hence, Theorem~\ref{thm-mainplus} has the following weakening.
\begin{corollary}\label{cor-univ}
For every non-negative integer $t\le 10$, every $4$-light $5$-rooted graph $G$ such that $\rho_4(G)\ge t$ is $\SS_{5,t}$-universal.
\end{corollary}
With this, it is easy to obtain the desired conclusion about the number of edges forcing $K_5$ as a rooted minor.
\begin{proof}[Proof of Theorem~\ref{thm-allK5}]
Note that since $G$ is 5-connected, it is also 4-light.
Let $F=\overline{G[X_G]}$ be the graph formed by the non-edges between the roots of $G$.
Since $|E(G)|\ge 4|V(G)|-10=4n(G)+10$, we have $\rho_4(G)\ge 10-|E(G[X_G])|=|E(F)|$.  Corollary~\ref{cor-univ} then implies that $\widetilde{G}$ is $\SS_{5,|E(F)|}$-universal,
and in particular $\widetilde{G}$ contains $F$ as an $\id$-rooted minor.  This implies that $G$ contains $K_5$ as a rooted minor.
\end{proof}

\section{Preliminaries}

Let us now give a few more definitions and simple results that we need in the argument.
We often use the following standard fact.
\begin{observation}\label{obs-submod}
Let $(A,B)$ and $(C,D)$ be separations of a graph $G$ of orders $n_1$ and $n_2$, respectively.
Then $(A\cup C,B\cap D)$ and $(A\cap C,B\cup D)$ are also separations of $G$, and their orders $n'$ and $n''$ satisfy
$$n'+n''=n_1+n_2.$$
\end{observation}

For a rooted graph $G$ and a vertex $v\in V(G)\setminus X_G$, let $\deg^X_G v$ denote the number of neighbors of $v$ in $X_G$
and let $\deg^+_G v=\deg_G v+\deg^X_G v$. We drop the subscripts when the graph $G$ is clear from the context.
Observe that
\begin{equation}\label{eq-rhot}
\rho_t(G)=\sum_{v\in V(G)\setminus X_G} (\tfrac{1}{2}\deg^+ v-t).
\end{equation}

We say that a \emph{counterexample} is a $5$-rooted $4$-light graph $G$
that is not universal. Thus, there exists a graph $H$ with vertex set $X_G$ contained in the target of $G$ such that $H$ is not an $\id$-rooted minor of $G$;
we say that such a graph is a \emph{flaw} of $G$ if it is maximal, i.e., no proper supergraph of $H$ belongs to the target of $G$.
We aim to prove Theorem~\ref{thm-mainplus} by contradiction, and thus we suppose that a counterexample exists.
Note that if $G$ is a counterexample, then $\widetilde{G}$ is a counterexample as well.
A \emph{minimal counterexample} is a counterexample $G$ with $n(G)$ minimum and subject to that with $|E(G)|$ minimum (and in particular, $X_G$ is an independent set in $G$).
We now study the properties of minimal counterexamples, eventually concluding that they cannot exist and obtaining the desired contradiction.
Let us start with several simple observations.
\begin{observation}\label{obs-univsep}
If $G$ is a minimal counterexample and $(A,B)$ is a root separation of order five such that $A\neq X_G$, then the 5-rooted
graphs $R_{A,B}$ and $\widetilde{R}_{A,B}$ are universal.
\end{observation}
\begin{proof}
Note that if $(C,D)$ is a root separation of $F=\widetilde{R}_{A,B}$, then $(A\cup C,D)$ is a root separation of $G$ and
$\widetilde{R}^F_{C,D}=\widetilde{R}^G_{A\cup C, D}$.
Since $G$ is 4-light, it follows that $F$ is 4-light.  Since $X_G\subsetneq A$, we have $n(F)=|B\setminus A|<|V(G)\setminus X_G|=n(G)$,
and thus $F$ is not a counterexample.  Consequently, $F$ is universal.  Since $F$ is an $\id$-rooted minor of $R_{A,B}$,
the 5-rooted graph $R_{A,B}$ is universal as well.
\end{proof}

\begin{observation}\label{obs-lightminor}
Let $G$ and $G'$ be 5-rooted graphs such that $X_{G'}=X_G$.
Suppose that $G$ is a minimal counterexample, $G'$ is an $\id$-rooted minor of $G$, and $G'$ is $4$-light.
If either $n(G')<n(G)$, or $n(G')=n(G)$ and $|E(G')|<|E(G)|$, then $G'$ is universal and $\rho_4(G')<\rho_4(G)$.
\end{observation}
\begin{proof}
By the minimality of $G$, the rooted graph $G'$ cannot be a counterexample, and thus it is universal.
Let $F$ be a flaw of $G$. Recall that $F$ belongs to the target $\SS$ of $G$ and $F$ is not an $\id$-rooted minor of $G$.
Since $G'$ is an $\id$-rooted minor of $G$, it follows that the graph $F$ is not an $\id$-rooted minor of $G'$.
Since $G'$ is universal, this implies that $F$ does not belong to the target $\SS'$ of $G'$.
Therefore, $\SS\not\subseteq \SS'$, and thus $\rho_4(G)>\rho_4(G')$.
\end{proof}

\begin{observation}\label{obs-nomg}
Let $G$ and $G'$ be 5-rooted graphs such that $X_{G'}=X_G$.
If $G$ is a counterexample, $F$ is a flaw of $G$, and $G'$ is an $\id$-rooted minor of $G$, then $F-E(G'[X_G])$ is not an $\id$-rooted minor of $G'$.
In particular, if $G$ is a minimal counterexample,
$G'$ is $4$-light, and either $n(G')<n(G)$, or $n(G')=n(G)$ and $|E(G')|<|E(G)|$, then $F-E(G'[X_G])$ does not belong to the target of $G'$.
\end{observation}

\section{The $(\le\!4)$-rooted case}\label{sec-fewroots}

For $k\le 4$, an exact characterization of $k$-rooted graphs not containing $K_k$ as a rooted minor is known.
To state this characterization, the following concept will be useful.
For a $k$-rooted graph $G$, we say that a $k$-rooted graph $M$ is a \emph{skeleton} of $G$
if $V(M)\subseteq V(G)$, $X_M=X_G$, $M$ is a supergraph of $G[V(M)]$, and for each component $C$ of $G-V(M)$, the vertices of $M$
with a neighbor in $C$ induce a clique of size at most $k-1$ in $M$.  The following result for $k\in\{1,2,3\}$ is folklore.
\begin{lemma}\label{lemma-123}
Let $G$ be a $k$-rooted graph with $k\in \{1,2,3\}$.  The graph $G$ does not contain $K_k$ as a rooted minor if and only if
\begin{itemize}
\item $k\in \{2,3\}$ and $G$ has a skeleton $M\neq K_k$ with vertex set $X_G$, or
\item $k=3$ and $G$ has a skeleton $M$ such that $|V(M)\setminus X_G|=1$ and $X_G$ is an independent set in $M$.
\end{itemize}
\end{lemma}

The characterization for $k=4$ is more involved and was found by Fabila-Monroy and Wood~\cite{fabila2013rooted}.
For a vertex set $X$, an \emph{$X$-web} is a plane graph with the outer face bounded by a cycle with vertex set $X$,
all internal faces having length three, and no non-facial triangles.  See Figure~\ref{fig-nok4} illustrating the
five cases from this theorem.
\begin{figure}
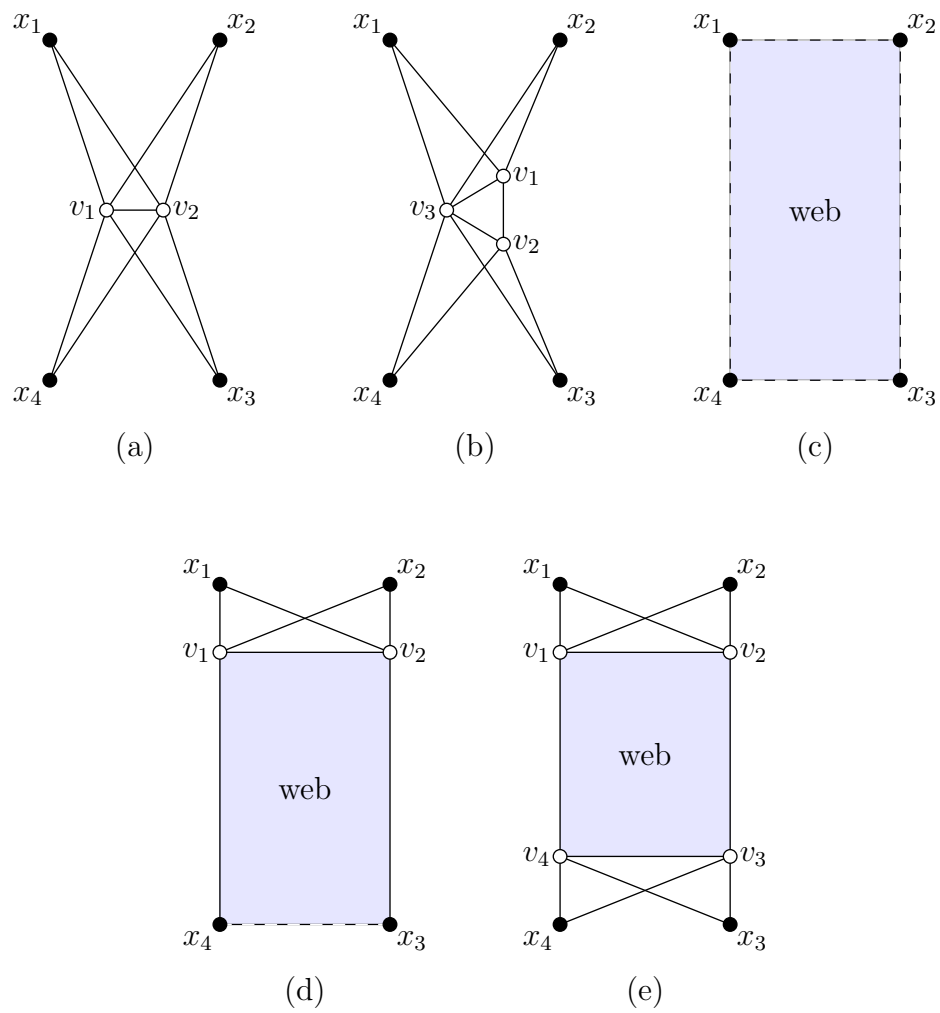


\begin{asy}
void mkvs (bool web, bool webu, bool webl)
{
  if (web && !webu)
    draw (v[0] -- v[1], white + dashed);
  if (web && !webl)
    draw (v[2] -- v[3], white + dashed);
  if (web && !webu && ! webl)
    {
      draw (v[1] -- v[2], white + dashed);
      draw (v[0] -- v[3], white + dashed);
    }

  for (int k = 0; k < 4; ++k)
    vertex (v[k], black);
  label ("$x_1$", v[0], NW);
  label ("$x_2$", v[1], NE);
  label ("$x_3$", v[2], SE);
  label ("$x_4$", v[3], SW);

  if (web && webu)
    {
      vertex (v[4]);
      vertex (v[5]);
      label ("$v_1$", v[4], W);
      label ("$v_2$", v[5], E);
    }

  if (web && webl)
    {
      vertex (v[6]);
      vertex (v[7]);
      label ("$v_3$", v[6], E);
      label ("$v_4$", v[7], W);
    }
}

real wd = 2.5, sk = 2.5, ht = 5, wo = 1;

void ram(pair at, string cse, bool web, bool webu, bool webl)
{
  v[0] = at;
  v[1] = v[0] + (wd,0);
  v[2] = v[0] + (wd,-ht);
  v[3] = v[0] + (0,-ht);

  v[4] = v[0] + (0, webu ? -wo : 0);
  v[5] = v[1] + (0, webu ? -wo : 0);
  v[6] = v[2] + (0, webl ? wo : 0);
  v[7] = v[3] + (0, webl ? wo : 0);

  if (web)
    {
      filldraw (v[4] -- v[5] -- v[6] -- v[7] -- cycle, fillpen=interp(blue, white, 0.9));
      label ("web", interp (v[4], v[6], 0.5));
    }

  label (cse, at + (0.5wd, -ht - 1));
}

ram ((0,0), "(a)", false, false, false);

v[4] = interp (v[0],v[3], 0.5) + (wd/3,0);
v[5] = interp (v[0],v[3], 0.5) + (2wd/3,0);

for (i = 0; i < 4; ++i)
  draw (v[4] -- v[i] -- v[5]);
draw (v[4]--v[5]);
vertex (v[4]);
vertex (v[5]);
label ("$v_1$", v[4], W);
label ("$v_2$", v[5], E);

mkvs (false, false, false);

ram ((wd + sk,0), "(b)", false, false, false);

v[4] = interp (v[0],v[3], 0.5) + (wd/3,0);
v[5] = interp (v[0],v[3], 0.5) + (2wd/3,0.5);
v[6] = interp (v[0],v[3], 0.5) + (2wd/3,-0.5);

for (i = 0; i < 4; ++i)
  draw (v[4] -- v[i]);
draw (v[0] -- v[5] -- v[1]);
draw (v[2] -- v[6] -- v[3]);
draw (v[4]--v[5]--v[6]--cycle);
vertex (v[4]);
vertex (v[5]);
vertex (v[6]);
label ("$v_1$", v[5], E);
label ("$v_2$", v[6], E);
label ("$v_3$", v[4], W);

mkvs (false, false, false);

ram ((2*(wd + sk),0), "(c)", true, false, false);
mkvs (true, false, false);

ram (((wd+sk)/2, -ht - 3), "(d)", true, true, false);
for (i = 0; i < 2; ++i)
  for (j = 0; j < 2; ++j)
    draw (v[i] -- v[4+j]);
mkvs (true, true, false);

ram ((3*(wd+sk)/2, -ht - 3), "(e)", true, true, true);
for (i = 0; i < 2; ++i)
  for (j = 0; j < 2; ++j)
    {
      draw (v[i] -- v[4+j]);
      draw (v[2 + i] -- v[6+j]);
    }
mkvs (true, true, true);
\end{asy}

\caption{Skeletons from Theorem~\ref{thm-k4}.}\label{fig-nok4}
\end{figure}

\begin{theorem}\label{thm-k4}
Let $G$ be a rooted graph with four roots.  If $G$ does not contain $K_4$ as a rooted minor, then $G$ has a skeleton $M$
such that
\begin{itemize}
\item[(a)] $|V(M)\setminus X_G|\le 2$; or
\item[(b)] $V(M)\setminus X_G=\{v_1,v_2,v_3\}$, $v_1v_2v_3$ is a triangle in $M$, $\deg^X_M v_3=4$, $\deg^X_M v_1,\deg^X_M v_2=2$, and $v_1$ and $v_2$ do not have a common neighbor in $X_G$; or
\item[(c)] $M$ is an $X_G$-web; or
\item[(d)] there exist distinct roots $x_1$ and $x_2$ and vertices $v_1,v_2\in M\setminus X_G$ such that the neighbors of each of $x_1$ and $x_2$ are exactly $v_1$ and $v_2$, and $M-\{x_1,x_2\}$ is
an $((X_G\setminus\{x_1,x_2\})\cup\{v_1,v_2\})$-web with $v_1$ and $v_2$ adjacent through an edge incident with the outer face; or
\item[(e)] letting $X_G=\{x_1,x_2,x_3,x_4\}$, the graph $M-X_G$ is a $\{v_1,v_2,v_3,v_4\}$-web with the outer face bounded by a 4-cycle $v_1v_2v_3v_4$,
the neighbors of each of the roots $x_1$ and $x_2$ are exactly $v_1$ and $v_2$
and neighbors of each of the roots $x_3$ and $x_4$ are exactly $v_3$ and $v_4$.
\end{itemize}
\end{theorem}

\begin{corollary}\label{cor-k4}
Let $G$ be a 4-light $k$-rooted graph with $k\le 4$ such that $\rho_4(G)>0$.
If $K_k$ is not a rooted minor of $G$, then
$k=4$, $X_G$ is an independent set in $G$, and $\rho_4(G)=1$; and in that case $G$ is $\SS_{4,5}$-universal.
\end{corollary}
\begin{proof}
The claims are true if $G$ contains $K_k$ as a rooted minor.  Hence, we can assume that this is not the case.
Let $M$ be a skeleton of $G$ as in Lemma~\ref{lemma-123} or Theorem~\ref{thm-k4}.
Let $K_1$, \ldots, $K_m$ be the components of $G-V(M)$, and for $i\in\{1,\ldots,m\}$,
let $(A_i,B_i)$ be the root separation of $G$ such that $A_i=V(G)\setminus V(K_i)$ and $B_i$ consists of $V(K_i)$
and the vertices of $M$ with a neighbor in $K_i$.  Since $M$ is a skeleton of $G$, note that the
separation $(A_i,B_i)$ has order less than $k$, and since $G$ is $4$-light, we have $\rho_4(R_{A_i,B_i})\le 0$.
Therefore,
$$1\le \rho_4(G)=\rho_4(G[V(M)])+\sum_{i=1}^m \rho_4(R_{A_i,B_i})\le \rho_4(G[V(M)])\le \rho_4(M).$$
Since $k\le 4$, this is only possible if $n(M)\ge 2$.  In particular, we must have $k=4$,
since no such skeleton arises in Lemma~\ref{lemma-123}.

If $n(M)=2$, then $\rho_4(M)\ge 1$ implies that both vertices $v_1,v_2\in V(M)\setminus X_G$ are adjacent to all roots.
Moreover, we necessarily have $\rho_4(G)=\rho_4(M)=1$, and thus $\rho_4(G[V(M)])=1$.
Hence, all edges of $M$ incident with $v_1$ and $v_2$ are also present in $G$.
This clearly implies that $G$ is $\SS_{4,5}$-universal, and that $K_4$ is a rooted minor of $G$ if $X_G$ is not an independent set in $G$.
Therefore, we can assume that $n(M)\ge 3$ and the outcome (a) of Theorem~\ref{thm-k4} does not hold.

In the outcome (b), we would have $\rho_4(M)\le -1$, which is not possible.
In the three remaining outcomes, we apply Euler's formula to the web forming a part of $M$,
obtaining
\begin{itemize}
\item $\rho(M)=3n(M)+1$ and $\rho_4(M)=1-n(M)$ in the outcome (c),
\item $\rho(M)=3n(M)+2$ and $\rho_4(M)=2-n(M)$ in the outcome (d), and
\item $\rho(M)=3n(M)+1$ and $\rho_4(M)=1-n(M)$ in the outcome (e).
\end{itemize}
Since $n(M)\ge 3$, these outcomes are also excluded.
\end{proof}

Theorem~\ref{thm-k4} can also be used to find rooted minors in 5-rooted graphs.
In particular, we are going to apply it to derive the following properties.
Let $K_{2,3}^+$ denote the graph obtained from $K_{2,3}$ by adding an edge between
the vertices of degree three.  Let $\{V_2,V_4\}$ be the partition of the vertices of $K_{2,3}^+$
to those of degree two and degree four, respectively.  We say that a $5$-rooted graph $G$
is \emph{nearly $K_{2,3}^+$-universal} if there exists a root $y\in X_G$ (called an \emph{exceptional root})
such that for every bijection $\pi:V(K_{2,3}^+)\to X_G$,
\begin{itemize}
\item if $\pi^{-1}(y)\in V_4$, then $K_{2,3}^+$ is a $\pi$-rooted minor of $G$, and
\item if $\pi^{-1}(y)\in V_2$, then for every edge $e\in E(K_{2,3}^+)$ incident with
$\pi^{-1}(y)$, the graph $K_{2,3}^+-e$ is a $\pi$-rooted minor of $G$.
\end{itemize}

\begin{observation}\label{obs-23univ-to-sizeuniv}
If a $5$-rooted graph is nearly $K_{2,3}^+$-universal, then it is $\SS_{5,3}$-universal,
and if it is $K_{2,3}^+$-universal, then it is $\SS_{5,4}^-$-universal.
\end{observation}

To show that a 5-rooted graph is (nearly) $K_{2,3}^+$-universal, we generally use the following observation.
Here, $K_{2,5}^-$ denotes the graph obtained from $K_{2,5}$ by deleting an edge.
\begin{observation}\label{obs-near-bistar-univ}
Let $G$ be a 5-rooted graph, let $H$ be $K_{2,5}$ or $K^-_{2,5}$ and let $\pi$ be a bijection
between the vertices of $H$ of degree at most two and the roots of $G$.  Suppose that $H$ is a $\pi$-rooted
minor of $G$.
\begin{itemize}
\item If $H=K_{2,5}$, then $G$ is $K_{2,3}^+$-universal.
\item If $H=K^-_{2,5}$ and $v$ is the vertex of $H$ of degree one, then $G$ is nearly $K_{2,3}^+$-universal with the exceptional root $\pi(v)$.
\end{itemize}
\end{observation}

A \emph{vampire} is a 5-rooted graph $W$ such that $V(W)\setminus X_W=\{v_1,v_2\}$, $v_1v_2\in E(W)$,
and there exist distinct vertices $x_1,x_2\in X_W$ such that $v_1$ is adjacent to all vertices in $X_W\setminus\{x_2\}$
and $v_2$ is adjacent to all vertices in $X_W\setminus\{x_1\}$.  The edges $v_1x_1$ and $v_2x_2$ are the \emph{fangs} of the vampire.

\begin{corollary}\label{cor-nok4}
Let $G$ be a 4-light 5-rooted graph.  If $\rho_4(G)>0$, then $G$ is $K_{1,4}$-universal and
\begin{itemize}
\item[(x)] $G$ is $(K_4+K_1)$-universal, or
\item[(i)] $\rho_4(G)=1$ and
\begin{itemize}
\item $G$ contains a non-root vertex adjacent to all roots, or
\item $G$ is nearly $K_{2,3}^+$-universal, or
\item there exists a vampire $W$ with $X_W=X_G$ such that $G$ contains $W$ as an $\id$-rooted minor.
\end{itemize}
\item[(ii)] $\rho_4(G)=2$ and $G$ is nearly $K_{2,3}^+$-universal; or
\item[(iii)] $\rho_4(G)=3$ and $G$ contains vertices $v_1,v_2\in V(G)\setminus X_G$ such that $v_1v_2\in E(G)$
and both $v_1$ and $v_2$ are adjacent to all roots (and in particular $G$ is $K_{2,3}^+$-universal).
\end{itemize}
In particular, either $\rho_4(G)=1$ and a non-root vertex of $G$ is adjacent to all roots, or
$G$ contains a graph with vertex set $X_G$ and with seven edges as an $\id$-rooted minor.
\end{corollary}
\begin{proof}
Let $c=\rho_4(G)$.  Without loss of generality, we can assume that $X_G$ is an independent set in $G$,
since otherwise we can delete all edges between the roots.

Since $G$ is $4$-light and $\rho_4(G)>0$, there exists a component $K_0$ of $G-X_G$
such that each root of $G$ has a neighbor in $K_0$.  Thus, for any $x\in X_G$, we can contract all vertices of
$K_0$ to $x$ and show that $G$ contains $K_{1,4}$ as a $\pi$-rooted minor for each bijection $\pi:V(K_{1,4})\to X_G$
mapping the vertex of $K_{1,4}$ of degree four to $x$.  This shows that $G$ is $K_{1,4}$-universal.

Suppose now that the conclusion (x) is false, and thus there exists a root $x\in X_G$ such that the 4-rooted graph $G-x$ does not
contain $K_4$ as a rooted minor.  Then $G-x$ has a skeleton $M$ satisfying one of the conditions
from the statement of Theorem~\ref{thm-k4}.  Let $G'=G[V(M)\cup \{x\}]$ and let $c'=\rho_4(G')$.
Let $K_1$, \ldots, $K_m$ be the components of $G-(V(M)\cup\{x\})$, and for $i\in\{1,\ldots,m\}$,
let $(A_i,B_i)$ be the root separation of $G$ such that $A_i=V(G)\setminus V(K_i)$ and $B_i$ consists of $V(K_i)$
and the vertices of $V(M)\cup\{x\}$ with a neighbor in $K_i$.  Since $M$ is a skeleton of $G-x$, note that the
separation $(A_i,B_i)$ has order less than $5$, and since $G$ is $4$-light, we have $\rho_4(R_{A_i,B_i})\le 0$.
Therefore,
$$1\le c=\rho_4(G)=\rho_4(G')+\sum_{i=1}^m \rho_4(R_{A_i,B_i})\le \rho_4(G')=c'.$$
Moreover, let $q_x$ be the number of vertices of $V(M)\setminus X_M$ that are not adjacent to $x$ in $G$,
and let $q=q_x+\rho(M)-|E(G[V(M)])|$.
We have
\begin{equation}
c'=\rho_4(G')=\rho_4(M)+n(M)-q.\label{eq-missing}
\end{equation}
Let us now discuss the possible outcomes of Theorem~\ref{thm-k4} one by one.
In each of the cases, we use the labelling of the vertices of $M$ as in the statement of Theorem~\ref{thm-k4} (and Figure~\ref{fig-nok4}).
\begin{itemize}
\item Let us first consider the case that $M$ satisfies (a).  There are the following possibilities:
\begin{itemize}
\item $G'$ has exactly one non-root vertex $v$, the vertex $v$ is adjacent to all roots,
$c'=c=1$, and the outcome (i) holds.
\item $G'$ has exactly two non-root vertices $v_1$ and $v_2$, and $v_1v_2\not\in E(G)$.
Then $c'\in\{1,2\}$, $\deg^X_G v_1,\deg^X_G v_2\ge 4$, and at least one of $v_1$ and $v_2$
is adjacent to all roots.  It follows that $G$ is nearly $K_{2,3}^+$-universal and
one of the outcomes (i) and (ii) holds.
\item $G'$ has exactly two non-root vertices $v_1$ and $v_2$, and $v_1v_2\in E(G)$.
Without loss of generality, we have $\deg^X v_1\ge \deg^X v_2$.
\begin{itemize}
\item Suppose first that $\deg^X v_1=5$.  If $c=1$, then the outcome (i) holds,
and thus suppose that $2\le c\le c'$.  Thus, $\deg^X v_2=2+c'\ge 4$.
If $c'=2$, then $G$ is nearly $K_{2,3}^+$-universal and the outcome (ii) holds,
and if $c'=3$, then $G$ is $K_{2,3}^+$-universal and one of the outcomes (ii) and (iii) holds.
\item Next, suppose that $\deg^X v_1=4$, and thus $\deg^X v_2=4$ and $c'=c=1$.
If $v_1$ and $v_2$ have different neighborhoods in $X_G$, then there exists a vampire $W$ with $X_W=X_G$ such that $G$ contains $W$ as an $\id$-rooted minor,
and the outcome (i) holds.
Otherwise, let $y\in X_G$ be the common non-neighbor of $v_1$ and $v_2$.  As we have observed at the beginning of
the proof, there exists a component $K_0$ of $G-X_G$ such that each root has a neighbor in $K_0$.  Since $M$ is a skeleton of $G-x$,
we necessarily have $\{v_1,v_2\}\cap V(K_0)\neq\emptyset$.
Thus, there exists a path $P$ from $y$ to $v_1$ or $v_2$ through $K_0$.  In this case, we contract the path $P$ to a single edge and
apply Observation~\ref{obs-near-bistar-univ} to show that $G$ is nearly $K_{2,3}^+$-universal.
Therefore, the outcome (i) again holds.
\end{itemize}
\end{itemize}
\item Next, let us consider the case that $M$ satisfies (b).  Then $\rho_4(M)=-1$ and
by (\ref{eq-missing}), $1\le c'=2-q$.
\begin{itemize}
\item Suppose first that $q=1$, and thus $c=c'=1$.  If $\deg^X_G v_3=5$, then the outcome (i) holds.
Otherwise, since $q=1$, we have $\deg^X_G v_3=4$, $\deg^X_G v_1=\deg^X_G v_2=3$ and $v_1v_2\in E(G)$;
we can contract the edge $v_1v_2$ and apply Observation~\ref{obs-near-bistar-univ} to show that $G$ is nearly $K_{2,3}^+$-universal.
Hence, the outcome (i) again holds.
\item Next, suppose that $q=0$.  Then $\deg^X_G v_3=5$, and we can contract the edge $v_1v_2$ and apply Observation~\ref{obs-near-bistar-univ} to show
that $G$ is $K_{2,3}^+$-universal.  Therefore, one of the outcomes (i) and (ii) holds.
\end{itemize}
\item Next, let us consider the case that $M$ satisfies (c).  Euler's formula implies that
$\rho_4(M)=1-n(M)$, and by (\ref{eq-missing}), we have $c'=1-q$.
Consequently $q=0$ and $c=c'=1$.  
Let $x_1x_2x_3x_4$ be the cycle bounding the outer face of the $(X_G\setminus \{x\})$-web $M$.
We can assume that we are not in the case (a), and thus $n(M)\ge 3$.
Since $M$ does not contain non-facial triangles, it follows that $x_1x_3,x_2x_4\not\in E(M)$.
Note that since $X_G$ is an independent set in $G$, the cycle $x_1x_2x_3x_4$ does not appear in $G$,
and since $q=0$, we have $M-E(x_1x_2x_3x_4)\subseteq G$.  Moreover, $q=0$ also implies that all vertices
of $V(M)\setminus \{x_1,\ldots,x_4\}$ are adjacent to $x$ in $G$.

If the vertices $x_1$ and $x_3$ have a common neighbor in $M-\{x_2,x_4\}$ and the vertices $x_2$ and $x_4$ have
a common neighbor in $M-\{x_1,x_3\}$, then by planarity, there exists a vertex $v\in V(M)\setminus\{x_1,\ldots, x_4\}$
adjacent to all vertices of $X_G$ in $G$.  Consequently, the outcome (i) holds.

Otherwise, we can by symmetry assume that $x_1$ and $x_3$ do not have any common neighbors other than $x_2$ and $x_4$.
Since all internal faces of $M$ are triangles and $M$ does not contain non-facial triangles,
for $i\in\{1,3\}$, the neighbors of $x_i$ in $M$ form a path $P_i$ from $x_2$ to $x_4$. Since $x_2x_4\not\in E(M)$,
each of the paths $P_1$ and $P_3$ has at least one internal vertex, and $M-\{x_2,x_4\}$ contains a path $P$ starting in
$V(P_1)\setminus\{x_2,x_4\}$, ending in $V(P_3)\setminus\{x_2,x_4\}$, and otherwise disjoint from these sets.
Let us contract the path $P_1-\{x_2,x_4\}$ to a single vertex $v_1$, the path $P_3-\{x_2,x_4\}$ to a single vertex $v_3$,
and the path $P$ to the edge $v_1v_3$.  The resulting rooted minor of $G$ is a vampire with fangs $v_1x_1$ and $v_3x_3$,
and thus the outcome (i) holds.

\item Next, let us consider the case that $M$ satisfies (d).  Euler's formula implies that
$\rho_4(M)=2-n(M)$, and by (\ref{eq-missing}), we have $c'=2-q$.
Moreover, we can assume that we are not in the case (a), and thus $n(M)\ge 3$.  Since the $\{v_1,v_2,x_3,x_4\}$-web $M-\{x_1,x_2\}$
does not contain non-facial triangles, it follows that $v_1v_2x_3x_4$ is an induced cycle in $M$.
\begin{itemize}
\item Suppose first that $q=0$.  Note that any two distinct roots except for $x_3$ and $x_4$ have $v_1$ or $v_2$ as a common neighbor,
and $x_3$ and $x_4$ have a common neighbor in the $\{v_1,v_2,x_3,x_4\}$-web $M-\{x_1,x_2\}$.
Since all internal faces of $M-\{x_1,x_2\}$ are triangles, the graph $M-\{x_1,x_2\}$ contains either a path from $v_1$ to $x_3$
disjoint from $\{v_2,x_4\}$, or a path from $v_2$ to $x_4$ disjoint from $\{v_1,x_3\}$.
We can contract such a path to a single edge ($v_1x_3$ or $v_2x_4$) and apply
Observation~\ref{obs-near-bistar-univ} to show that $G$ is nearly $K_{2,3}^+$-universal.
Therefore, one of the outcomes (i) and (ii) holds.

\item Next, suppose that $q=1$, and thus $c=c'=1$.
For $i\in\{1,2\}$, let $S_i=\{v_ix_1,v_ix_2,v_ix_{5-i},v_ix\}$.
If $S_1\cup S_2\cup\{v_1v_2\}\subseteq E(G)$, then $G$ contains a vampire with fangs
$v_1x_4$ and $v_2x_3$ as a subgraph, and the outcome (i) holds.

Otherwise, there exists an edge $e_0\in (S_1\cup S_2\cup \{v_1v_2\})\setminus E(G)$;
this edge is unique, since $q=1$.  By symmetry between $S_1$ and $S_2$ and between $x_1$ and $x_2$, we can assume
that $e_0\not\in S_2\cup \{v_1x_2\}$.  Let us first consider the case that $e_0\in \{v_1x_4,v_1v_2\}$.
Since the cycle $v_1v_2x_3x_4$ is induced, the ends of $e_0$ have a common neighbor $v$ in $M-\{x_1,x_2\}$
not belonging to this cycle.  By contracting the edge $vv_1$, we obtain a vampire with fangs
$v_1x_4$ and $v_2x_3$ as an $\id$-rooted minor of $G$, and the outcome (i) holds.

Similarly, if $e_0=v_1x$, we can contract an edge from $v_1$ to a neighbor in $M-\{x_1,x_2,v_2,x_4\}$ (which is adjacent to $x$, since $q=1$),
and obtain a vampire with fangs $v_1x_4$ and $v_2x_3$ as an $\id$-rooted minor of $G$, again resulting in the outcome (i).

It remains to consider the case that $e_0=v_1x_1$.  Since all internal faces of $M-\{x_1,x_2\}$ are triangles and $v_2x_4\not\in E(G)$,
there exists a path $P$ from $v_1$ to $x_3$ in $M-\{x_1,x_2,v_2,x_4\}$.  By contracting $P$ to a single edge,
we obtain a vampire with fangs $v_1x_4$ and $v_2x_1$ as an $\id$-rooted minor of $G$, and the outcome (i) holds.
\end{itemize}

\item Finally, let us consider the case that $M$ satisfies (e).  Euler's formula implies that $\rho_4(M)\le 1-n(M)$,
and by (\ref{eq-missing}), we have $c'=1-q$.
Consequently $q=0$ and $c=c'=1$.  Since $q=0$, we have $M-E(M[X_M])\subseteq G$ and all vertices of $M-X_G$ are adjacent to $x$.
Thus, we can contract the edges $v_1v_4$ and $v_2v_3$ of $G$ and apply Observation~\ref{obs-near-bistar-univ}
to see that $G$ is $K_{2,3}^+$-universal.  Hence, the outcome (i) holds.
\end{itemize}
Therefore, one of the conclusions of the lemma holds.
\begin{itemize}
\item If $G$ is nearly $K_{2,3}^+$-universal, or if $G$ contains a vampire as an $\id$-rooted minor,
then we can contract edges of $G$ to make the subgraph induced by the roots
a supergraph of $K_{2,3}^+$, showing that $G$ contains an $\id$-rooted minor with seven edges between roots.
\item If $G$ is $(K_4+K_1)$-universal, then let $H=K_4+K_1$ and let $\pi:V(H)\to X_G$ be any bijection.
Then $G$ contains $H$ as a $\pi$-rooted minor.  Let $\mu$ be a $\pi$-rooted model of $H$ in $G$
and let $z$ be the isolated vertex of $H$.  Recall that $G-X_G$ has a component $K_0$
such that each root has a neighbor in $K_0$, and thus $G$ contains a path $P$ from
$\pi(z)$ to another root.  Let $P_0$ be the longest initial segment vertex-disjoint from $\mu(V(H)\setminus \{z\})$,
and let $y$ be the vertex of $H$ such that the vertex of $P$ following $P_0$ belongs to $\mu(y)$.
Let $H'=H+yz$; then $G$ clearly contains $H'$ as a $\pi$-rooted minor, and thus
$G$ contains an $\id$-rooted minor with seven edges between roots.
\end{itemize}
It follows that if $\rho_4(G)\ge 2$, or if $\rho_4(G)=1$ and no non-root vertex is adjacent to all the roots,
then $G$ contains an $\id$-rooted minor with seven edges between roots.
\end{proof}

For a rooted graph $G$ and a root $z$, the \emph{$z$-star} in $G$ is the graph with vertex set $X_G$ and with edges $zx$ for all $x\in X_G\setminus\{z\}$.
Thus, the rooted graph $G$ is $K_{1,|X_G|-1}$-universal if and only if for every $z\in X_G$ it contains the $z$-star as an $\id$-rooted minor.
For a set $Z\subseteq X_G$, the \emph{$Z$-star} in $G$ is the union of the $z$-stars for $z\in Z$.
Let us note the following simple consequence of the first part of Corollary~\ref{cor-nok4}.

\begin{observation}\label{obs-42dense}
Every counterexample $G$ satisfies $\rho_4(G)\ge 2$, and each flaw of $G$ contains a matching of size two or a triangle.
\end{observation}
\begin{proof}
Consider any flaw $H$ of $G$.  Since $H$ is not an $\id$-rooted minor of $G$, $H$ has at least one edge, and since $H$ belongs to the target of $G$,
we have $\rho_4(G)\ge 1$.  By Corollary~\ref{cor-nok4}, $G$ is $K_{1,4}$-universal.  Since $H$ is not an $\id$-rooted minor of $G$,
not all edges of $H$ are incident with the same vertex, and thus $H$ contains a matching of size two or a triangle.
In particular $|E(H)|\ge 2$, and since $H$ belongs to the target of $G$, we have $\rho_4(G)\ge 2$.
\end{proof}

The second part of Corollary~\ref{cor-nok4} has the following easy consequence.

\begin{corollary}\label{cor-foden}
Let $G$ be a 4-light 5-rooted graph, let $x$ be a root of $G$ and let $F$ be a graph with vertex set $X_G$
such that $x$ is an isolated vertex of $F$.
If $G$ does not contain $F$ as an $\id$-rooted minor, then
\begin{itemize}
\item $\rho_4(G)\le 1$, or
\item $\rho_4(G)=2$ and $|E(F)|\ge 5$, or
\item $\rho_4(G)\le 3$ and $F-x$ is isomorphic to $K_4$.
\end{itemize}
\end{corollary}
\begin{proof}
We can assume $\rho_4(G)\ge 2$, as otherwise the first conclusion holds.
We apply Corollary~\ref{cor-nok4} to the graph $\widetilde{G}$.
Since $G$ does not contain $F$ as an $\id$-rooted minor,
the rooted graph $\widetilde{G}$ is not $(K_4+K_1)$-universal, and thus the outcome (x) does not hold.
Therefore, (ii) or (iii) holds, and thus $\rho_4(G)\le 3$.

Suppose now that $F-x$ is not isomorphic to $K_4$.  Let $X_G=\{x,y_1,y_2,y_3,y_4\}$, where the labels are chosen so that $y_1y_2\not\in E(F)$.
Let $H$ be the $\{y_3,y_4\}$-star in $G$.
Since $F$ is not an $\id$-rooted minor of $G$, $H$ is not an $\id$-rooted minor of $\widetilde{G}$.
Since $H$ is isomorphic to $K_{2,3}^+$, it follows that $\widetilde{G}$ is not $K_{2,3}^+$-universal, and in particular (iii) cannot hold.

Therefore (ii) holds, $\rho_4(G)=2$, and $\widetilde{G}$ is nearly $K_{2,3}^+$-universal.
The exceptional root of $\widetilde{G}$ is one of the vertices of $H$ of degree two.
Since $F$ is not an $\id$-rooted minor of $G$ and $x$ is an isolated vertex of $F$, the graph $H-xy_3$ is not an $\id$-rooted minor of $\widetilde{G}$, and thus $x$ is not the exceptional root.
Hence, we can by symmetry assume that $y_1$ is the exceptional root of $\widetilde{G}$.

It follows that $\widetilde{G}$ contains as $\id$-rooted minors the graphs $H-y_1y_3$ and $H-y_1y_4$, and
since $F$ is not an $\id$-rooted minor of $G$, we have $y_1y_3,y_1y_4\in E(F)$.  Moreover,
$\widetilde{G}$ contains as an $\id$-rooted minor the $\{y_1,y_i\}$-star for each $i\in \{2,3,4\}$,
and since $F$ is not an $\id$-rooted minor of $G$, we conclude that $y_3y_4,y_2y_4,y_2y_3\in E(F)$.
Therefore, $|E(F)|=5$.
\end{proof}

We can now apply this claim to a minimal counterexample.

\begin{corollary}\label{cor-nearfour}
If $F$ is a flaw of a minimal counterexample $G$, then $F$ does not have isolated vertices, and in particular $|E(F)|\ge 3$.
\end{corollary}
\begin{proof}
By Observation~\ref{obs-42dense}, we have $\rho_4(G)\ge 2$.  Moreover, if $\rho_4(G)\le 3$, then the flaw $F\in \SS_{5,4}^-$ has at most four edges.
Consequently, none of the conclusions of Corollary~\ref{cor-foden} holds, and thus no vertex of $F$ is isolated.
\end{proof}

\section{$K_\star$ and $K^-_\star$-universal separations}

For a root separation $(A,B)$ of a rooted graph which is clear from the context and for a class $\SS$ of graphs (or a single graph $S$),
we are going to say that $(A,B)$ is \emph{$\SS$-universal} (or \emph{$S$-universal}, or \emph{nearly $K_{2,3}^+$-universal}) if the rooted graph $R_{A,B}$ has this property.
To further simplify the notation, we say that a root separation $(A,B)$ of order $k$ is
\begin{itemize}
\item \emph{$K_{1,\star}$-universal} to mean that it is $K_{1,k-1}$-universal,
\item \emph{$K_\star$-universal} to mean that it is $K_k$-universal, and
\item \emph{$K^-_\star$-universal} to mean that it is $K^-_k$-universal, where $K^-_k$ denotes the graph obtained from the clique $K_k$ by removing a single edge.
\end{itemize}

First, let us note several technical observations.
The \emph{census} of a rooted graph $F$ is the set of all its $\id$-rooted minors with vertex set $X_F$.
If $F$ is a part of a larger rooted graph $G$ attaching to the rest only through the roots of $F$,
it should be clear that replacing $F$ by any other rooted graph with the same census does not affect
the census of $G$.  More precisely:

\begin{observation}\label{obs-ceneq}
Let $G_1$ be a rooted graph and let $(A,B)$ be a root separation of $G_1$.  Let $H$ be a rooted graph with $X_H=A\cap B$ and otherwise vertex-disjoint from $G_1$.
Let $G_2=(L_{A,B}-E(G_1[A\cap B]))\cup H$ be the graph obtained from $G_1$ by replacing the right part of the separation $(A,B)$ by $H$.
If $H$ and $R_{A,B}$ have the same census, then $G_1$ and $G_2$ have the same census.
\end{observation}
\begin{proof}
Consider any graph $S$ in the census of $G_1$ and let $\mu$ be an $\id$-rooted model of $S$ in $G_1$.
Let $\PP$ be a partition of $A\cap B$ such that distinct vertices $u,v\in A\cap B$ belong to the same part if and only if
there exists a root $x\in X_{G_1}$ such that $u,v\in \mu(x)$ and the graph $R_{A,B}[\mu(x)\cap B]$ contains a path from $u$ to $v$
(this path may be needed to keep the subgraph induced by $\mu(x)$ connected).

Let $D$ be the graph with vertex set $\PP$ and with two parts $P_1$ and $P_2$ adjacent if and only if there exist distinct vertices
$x,y\in X_{G_1}$ such that $P_1\subseteq \mu(x)$, $P_2\subseteq\mu(y)$, and $R_{A,B}$ has an edge between the component
of $R_{A,B}[\mu(x)\cap B]$ containing $P_1$ and the component of $R_{A,B}[\mu(y)\cap B]$ containing $P_2$ (such an edge may be
needed to represent the edge of $S$ between $x$ and $y$, if any).

The restriction of $\mu$ to $R_{A,B}$ shows that there exists a graph $Z$ with $V(Z)=A\cap B$ such that $Z$ is an $\id$-rooted minor of $R_{A,B}$,
each part of $\PP$ induces a connected subgraph of $Z$, and for each edge $P_1P_2\in E(D)$, there exists an edge of $Z$ with one end in $P_1$ and the other end in $P_2$.
Since $H$ and $R_{A,B}$ have the same census, there also exists an $\id$-rooted model $\eta$ of $Z$ in $H$.  The natural combination of the restriction of $\mu$ to $L_{A,B}$ with $\eta$
gives an $\id$-rooted model of $S$ in $G_2$, and thus $S$ belongs to the census of $G_2$.

Since this holds for every graph in the census of $G_1$, it follows that the census of $G_1$ is a subset of the census of $G_2$.  The opposite inclusion follows by symmetry.
\end{proof}

Let $(A,B)$ and $(C,D)$ be root separations of a rooted graph $G$.
We say that $(C,D)$ is an \emph{isolator} of $(A,B)$ if $C\subseteq A$, $B\subseteq D$,
and the order of $(C,D)$ is minimal among the root separations with these properties.
The \emph{root connectivity} of $(A,B)$ is defined as the order of its isolators.
We say that $(A,B)$ is \emph{linked to the roots} if its root connectivity is equal to $|A\cap B|$.
Moreover, the root separation $(A,B)$ is \emph{strongly linked to the roots}
if $(A,B)$ has no isolators other than itself and the separation $(X_G,V(G))$ in the case that $|A\cap B|=|X_G|$.
Menger's theorem has the following consequence.
\begin{observation}\label{obs-isol}
Let $(A,B)$ be a root separation of a rooted graph $G$ and let $k$ be the root connectivity of $(A,B)$.
Then there exists a system $\QQ$ of $k$ pairwise vertex-disjoint paths in $G[A]$ starting in $X_G$,
ending in $A\cap B$, and otherwise disjoint from $X_G\cup B$.  Moreover, for each vertex $x\in X_G\cap B$,
the system $\QQ$ contains a single-vertex path consisting of $x$.
\end{observation}
We say such a system of paths $\QQ$ is a \emph{root linkage} for $(A,B)$.
For an isolator $(C,D)$ of $(A,B)$, note that $(A\cap D, B)$ is a root separation of $R_{C,D}$ of root connectivity $|C\cap D|$.
We say that a root linkage $\PP$ for $(A\cap D, B)$ in $R_{C,D}$ is an \emph{isolator linkage} for $(C,D)$ and $(A,B)$;
we usually view $\PP$ as a system of paths in $G$ from $C\cap D$ to $A\cap B$.
For a path $P$ of a root or isolator linkage, the \emph{terminator} of $P$ is its end in $A\cap B$ and the \emph{origin} is its other end.

\begin{observation}\label{obs-cliques}
Let $G$ be a rooted graph and let $(A,B)$ be a $K_\star$-universal root separation of $G$ of order $k$.
Then each isolator $(C,D)$ of $(A,B)$ is $K_\star$-universal.  In particular, if $G$ is a counterexample
and $(A,B)$ is linked to the roots, then $k<5$.
\end{observation}
\begin{proof}
Consider any isolator $(C,D)$ of $(A,B)$, and let $\PP$ be an isolator linkage for $(C,D)$ and $(A,B)$.
Since $(A,B)$ is $K_\star$-universal, we can contract $R_{A,B}$ to form a clique on the terminators of $\PP$
and then contract the paths of $\PP$ to form a clique on $C\cap D$, showing that $(C,D)$ is $K_\star$-universal.

Suppose now that $(A,B)$ is linked to the roots.  If $G$ is a counterexample, then $|X_G|=5$ and $K_5$ is not a rooted minor of $G$,
and thus the root separation $(X_G,V(G))$ is not $K_\star$-universal.  It follows that
$(X_G,V(G))$ is not an isolator of $(A,B)$, and thus $k<5$.
\end{proof}

\begin{lemma}\label{lemma-diamonds}
Let $G$ be a rooted graph and let $(A,B)$ be a $K^-_\star$-universal root separation of $G$ of order $k\le |X_G|$.
Suppose that every root separation $(M,N)$ of $G$ of order at most two such that $B\subseteq M$ is $4$-light.
\begin{itemize}
\item For each isolator $(C,D)$ of $(A,B)$, either
\begin{itemize}
\item $(C,D)$ is $K_\star$-universal, or
\item $(C,D)$ is $K^-_\star$-universal, $A\cap B$ and $C\cap D$ are independent sets of the same size $k$, and $\rho_4(R_{A,B})\ge\rho_4(R_{C,D})$.
\end{itemize}
\item If $G$ is a minimal counterexample, $k=5$, and $(A,B)$ is linked to the roots, then $(A,B)=(X_G,V(G))$.
\end{itemize}
\end{lemma}
\begin{proof}
Consider any isolator $(C,D)$ of $(A,B)$, and let $\PP$ be an isolator linkage for $(C,D)$ and $(A,B)$.
If $|C\cap D|<k$, then since $(A,B)$ is $K^-_\star$-universal, we can contract $R_{A,B}$ to form a clique on the terminators of $\PP$
and then contract the paths of $\PP$ to form a clique on $C\cap D$, showing that $(C,D)$ is $K_\star$-universal.
Hence, suppose that $|C\cap D|=k$.  An analogous argument shows that the root separation $(C,D)$ is $K^-_\star$-universal.

Next, let us consider the case that the graph $L_{A,B}\cap R_{C,D}$ has a component containing more than one vertex of $C\cap D$.
Then there exist distinct paths $P_1,P_2\in \PP$ and a path $Q$ in $L_{A,B}\cap R_{C,D}$
with one end in $V(P_1)$, the other end in $V(P_2)$, and otherwise disjoint from the paths of $\PP$.
Let $u$ and $v$ be the terminators of the paths $P_1$ and $P_2$.
Since the root separation $(A,B)$ is $K^-_\star$-universal, we can contract $R_{A,B}$ to form the clique on $A\cap B$ except for the edge $uv$,
contract the paths of $\PP$, and contract all but one edge of $Q$ to obtain $K_k$ as a rooted minor of $R_{C,D}$, again showing that
the root separation $(C,D)$ is $K_\star$-universal.

Finally, suppose that this is not the case, and thus we can partition $A\cap D$ into $k$ sets $D_v$ for $v\in C\cap D$ so that
\begin{itemize}
\item the path $P_v$ of $\PP$ with origin $v$ is contained in $G[D_v]$, and
\item for all distinct $u,v\in C\cap D$, the graph $G$ does not contain any edge between $D_u$ and $D_v$.
\end{itemize}
In particular, this implies that $A\cap B$ and $C\cap D$ are independent sets in $G$.
For each vertex $v\in C\cap D$, let $A_v=(V(G)\setminus D_v)\cup \{v,x\}$, where $x$ is the terminator of $P_v$; then $(A_v,D_v)$ is a root separation of $G$ of
order at most two and $B\subseteq A_v$, and by assumptions, we have $\rho_4(R_{A_v,D_v})\le 0$.
Let $\delta_v=0$ if $v=x$ and $\delta_v=|E(G[\{v,x\}])|-4$ if $v\neq x$.
Then
$$\rho_4(R_{C,D})=\rho_4(R_{A,B})+\sum_{v\in C\cap D} (\rho_4(R_{A_v,D_v})+\delta_v)\le \rho_4(R_{A,B}).$$
This finishes the proof of the first part of the claim.

Suppose now that $G$ is a minimal counterexample, $k=5=|X_G|$, and the root separation $(A,B)$ is linked to the roots.
Then we can consider the isolator $(C,D)=(X_G,V(G))$.  By the first part, the root separation $(X_G,V(G))$ is $K^-_\star$-universal,
and thus the clique $K_5$ is the only flaw of $G$.  Since $K_5$ belongs to the target of $G$, we have $\rho_4(G)\ge 7$.
Since $K_5$ is not a rooted minor of $G$, the root separation $(X_G,V(G))$ is not $K_\star$-universal, and thus
the first part implies that $\rho_4(R_{A,B})\ge \rho_4(G)\ge 7$.
Moreover, the root separation $(A,B)$ is not $K_\star$-universal by Observation~\ref{obs-cliques}.
This implies that $R_{A,B}$ is a counterexample with a flaw $K_5$.
By the minimality of $G$, we conclude that $(A,B)=(X_G,V(G))$.
\end{proof}

For a root separation $(A,B)$ of a rooted graph $G$, the graph obtained from $L_{A,B}$ by adding all edges of a clique on $A\cap B$
is the \emph{torso} of $(A,B)$.  The graph obtained from $L_{A,B}$ by adding $t$ vertices adjacent only to $A\cap B$ is the \emph{$t$-twin reduction} of $(A,B)$.
Observation~\ref{obs-ceneq} has the following consequence.
\begin{observation}\label{obs-samecensus}
Let $(A,B)$ be a root separation of a rooted graph $G$.
\begin{itemize}
\item If $(A,B)$ is $K_\star$-universal, then the torso of $(A,B)$ is an $\id$-rooted minor of $G$ with the same census as $G$.
\item If $(A,B)$ is $K^-_\star$-universal but not $K_\star$-universal (and in particular $A\cap B$ is an independent set),
then the $(|A\cap B|-2)$-twin reduction of $(A,B)$ has the same census as $G$.
\end{itemize}
\end{observation}
A useful property of the torso is that in each root separation, the whole clique replacing $A\cap B$ must be contained
in one of the sides of the separation, and thus root separations of the torso correspond to root separations of the same order
in the original graph.  This is not necessarily true in twin reductions; however, the following weaker claim will suffice for us.
\begin{lemma}\label{lemma-2twin}
Let $H$ be a $5$-rooted graph, let $S\subseteq V(H)$ be a set of four of its vertices, and let $v_1,v_2\in V(H)\setminus (S\cup X_H)$
be distinct vertices adjacent to all vertices of $S$ and no other vertices of $H$.  Let $S'=S\cup\{v_1,v_2\}$ and $M=V(H)\setminus\{v_1,v_2\}$ and
suppose that the root separation $(M,S')$ is strongly linked to the roots.
If $H$ has a root separation $(C,D)$ of order at most four that is not 4-light, then it also has a non-4-light root separation $(C',D')$ of order at most $|C\cap D|$
such that $S'\subseteq C'$ and $v_1,v_2\not\in D'$.
\end{lemma}
\begin{proof}
Since $(C,D)$ and $(M,S')$ are root separations of $H$, Observation~\ref{obs-submod} implies
that $(C_1,D_1)=(C\cap M,D\cup S')=(C\setminus\{v_1,v_2\},D\cup S')$ is also a root separation of $H$.  We claim that
$$|C_1\cap D_1|\ge 5.$$
Indeed, suppose for a contradiction that $|C_1\cap D_1|\le 4$.
Since the root separation $(M,S')$ is linked to the roots and its order is $|S|=4$,
it follows that $|C_1\cap D_1|=4$ and $(C_1,D_1)$ is an isolator of $(M,S')$.
Since $(M,S')$ is actually strongly linked to the roots and $|X_H|=5>|M\cap S'|$,
it follows that $(C_1,D_1)=(M,S')$. Consequently $M\subseteq C$ and $D\subseteq S'$, and thus
$D\setminus C\subseteq S'\setminus M=\{v_1,v_2\}$.
However, since $v_1v_2\not\in E(H)$, we have
$$\rho_4(R_{C,D})=\sum_{v\in D\setminus C} (\deg v-4)\le \sum_{v\in D\setminus C} (|C\cap D|-4)\le 0,$$
contradicting the assumption that the root separation $(C,D)$ is not $4$-light.

Now, note that
\begin{align*}
|C_1\cap D_1|&=|(D\cup S')\cap C\setminus\{v_1,v_2\}|\\
&=|D\cap C\setminus \{v_1,v_2\}|+|(S'\setminus D)\cap C\setminus\{v_1,v_2\}|\\
&=|C\cap D|-|C\cap D\cap \{v_1,v_2\}|+|S\setminus D|.
\end{align*}
Consequently,
\begin{equation}\label{eq-c1d1}
|S\setminus D|=|C_1\cap D_1|-|C\cap D|+|C\cap D\cap \{v_1,v_2\}|\ge |C\cap D\cap \{v_1,v_2\}|+1.
\end{equation}
In particular, $S\setminus D\neq\emptyset$, and since $v_1$ and $v_2$ are adjacent to all vertices of $S$ and $(C,D)$ is a separation,
it follows that $v_1,v_2\in C$.

Next, we claim that
\begin{equation}\label{eq-scsmall}
|S\setminus C|\le \min(1,|C\cap D\cap \{v_1,v_2\}|).
\end{equation}
This is trivial if $S\setminus C=\emptyset$.  On the other hand, if $S\setminus C\neq\emptyset$,
then since $v_1$ and $v_2$ are adjacent to all vertices of $S$ and $(C,D)$ is a separation, we have
$v_1,v_2\in D$ and $|C\cap D\cap \{v_1,v_2\}|=2$, and (\ref{eq-c1d1}) implies
$|S\setminus C|\le 4-|S\setminus D|\le 1<|C\cap D\cap \{v_1,v_2\}|$.

Finally, let us consider
$$(C',D')=(C\cup S',D\cap M)=(C\cup S,D\setminus \{v_1,v_2\}).$$  Note that $X_H\subseteq C\subseteq C'$.
By Observation~\ref{obs-submod}, since $(C,D)$ and $(S',M)$ are separations of $H$,
$(C',D')$ is a root separation of $H$.  Using (\ref{eq-scsmall}), we see that
\begin{align*}
|C'\cap D'|&=|(C\cup S)\cap D\setminus \{v_1,v_2\}|\\
&=|C\cap D\setminus \{v_1,v_2\}|+|(S\setminus C)\cap D\setminus \{v_1,v_2\}|\\
&=|C\cap D|-|C\cap D\cap \{v_1,v_2\}|+|S\setminus C|\\
&\le |C\cap D|.
\end{align*}
Recall that $|S\setminus C|\le 1$ by (\ref{eq-scsmall}).
Let $C_2=C\cup S$; then $(C_2,D)$ is a root separation of $H$ and the rooted graph $R_{C_2,D}$
is obtained from $R_{C,D}$ by making the vertex of $S\setminus C$ (if any) into a root.
We have
$$\rho_4(R_{C_2,D})\ge \rho_4(R_{C,D})+(4-|C\cap D|)\cdot |S\setminus C|\ge \rho_4(R_{C,D})>0.$$
Moreover, note that $R_{C',D'}=R_{C_2,D}-(\{v_1,v_2\}\cap D)$.  Since $v_1,v_2\in C$,
the vertices of $\{v_1,v_2\}\cap D$ are roots of $R_{C_2,D}$.  Moreover, since $S\subseteq C_2$,
all neighbors of the vertices of $\{v_1,v_2\}\cap D$ in $R_{C_2,D}$ are also roots.  Therefore,
$\rho_4(R_{C',D'})=\rho_4(R_{C_2,D})>0$,
and thus the separation $(C',D')$ is not $4$-light.  Since $S'\subseteq C'$ and $v_1,v_2\not\in D'$,
the conclusion of this lemma holds.
\end{proof}

We can now exclude two special kinds of root separations of a minimal counterexample.
We say that a root separation $(A,B)$ of a rooted graph $G$ is \emph{reducible} if 
\begin{itemize}
\item[(R1)] $B\setminus A\neq \emptyset$ and the separation $(A,B)$ is $K_\star$-universal, or
\item[(R2)] $|A\cap B|=4$, $|B\setminus A|\ge 3$, and the separation $(A,B)$ is $K^-_\star$-universal.
\end{itemize}

\begin{lemma}\label{lemma-noclicut}
A minimal counterexample $G$ does not contain any reducible root separation.
\end{lemma}
\begin{proof}
Suppose for a contradiction that $G$ contains a reducible root separation $(A,B)$;
we choose such a root separation with $B$ maximal and subject to that with $A$ minimal.

Consider any isolator $(C,D)$ of $(A,B)$.  If (R1) holds, then $(C,D)$ is $K_\star$-universal by Observation~\ref{obs-cliques}.
If (R2) holds, then Lemma~\ref{lemma-diamonds} implies that either $(C,D)$ is $K_\star$-universal, or
$|C\cap D|=4$ and $(C,D)$ is $K^-_\star$-universal.  Since $B\setminus A\subseteq D\setminus C$, in both cases
we conclude that the root separation $(C,D)$ is reducible.  By the choice of $(A,B)$, it follows that
$(C,D)=(A,B)$.  Therefore, $(A,B)$ is strongly linked to the roots and $|A\cap B|\le |X_G|=5$.

Moreover, if $|A\cap B|=5$, then $(A,B)$ is $K_\star$-universal, and since $(A,B)$ does not have any isolator other than itself,
we have $(A,B)=(X_G,V(G))$.  This implies that $K_5$ is a rooted minor of $G$, which is a contradiction.
Therefore, we have $|A\cap B|\le 4$.

Let $S=A\cap B$ and let us define a $5$-rooted graph $L$ and a set $Y$ as follows.
\begin{itemize}
\item If (R1) holds, then let $L$ be the torso of $(A,B)$ and let $Y=\emptyset$.
\item Otherwise (i.e., (R2) holds and $(A,B)$ is not $K_\star$-universal), then let $L$ be the $2$-twin reduction of $(A,B)$ and let $Y$ consist of the two newly added vertices.
\end{itemize}
Let $F$ be a flaw of the counterexample $G$.
By Observation~\ref{obs-samecensus}, the rooted graph $L$ has the same census as $G$, and thus $L$ does not contain $F$ as an $\id$-rooted minor.
Since the graph $G$ is $4$-light and $|A\cap B|\le 4$, we have
$$\rho_4(L)\ge\rho_4(L_{A,B})=\rho_4(G)-\rho_4(R_{A,B})\ge \rho_4(G).$$
Consequently, $F$ belongs to the target of $L$.  Moreover, since $(A,B)$ is reducible, we have $n(L)<n(G)$, and thus $L$ is not a counterexample.
This implies that $L$ is not $4$-light.

Let $(C,D)$ be a root separation of $L$ of smallest possible order (at most four) that is not $4$-light (and in particular $|D\setminus C|\ge 2$).
If $L$ is the $2$-twin reduction of $(A,B)$, Lemma~\ref{lemma-2twin} shows that we can without loss of generality additionally assume that
$S\subseteq C$ and $Y\subseteq C\setminus D$ (the root separation $(A,S\cup Y)$ of $L$ is strongly linked to the roots, since the root separation $(A,B)$ of $G$ is strongly linked to the roots).

If $S\subseteq C$, then let $C'=(C\setminus Y)\cup B$ and note that $(C',D)$ is a root separation
of $G$ of order $|C\cap D|\le 4$ and $\widetilde{R}^G_{C',D}=\widetilde{R}^L_{C,D}$.  Since $G$ is 4-light, it follows that the root
separation $(C,D)$ is 4-light, which is a contradiction.

Therefore, we have $S\not\subseteq C$, and in particular $L$ is the torso of $(A,B)$.
Since $S$ is a clique in $L$, it follows that $S\subseteq D$. 
Since $(C,D)$ is a non-4-light root separation of $L$
of the smallest possible order, the rooted graph $R_{C,D}$ is $4$-light.
By Corollary~\ref{cor-k4}, we conclude that either $(C,D)$ is $K_\star$-universal, or $|C\cap D|=4$ and $(C,D)$ is $K^-_\star$-universal.
Consequently, the root separation $(C, D\cup B)$ of $G$ either is $K_\star$-universal, or has order four and is $K^-_\star$-universal.
Moreover, $|(D\cup B)\setminus C|=|B\setminus A|+|D\setminus C|\ge |B\setminus A|+2\ge 3$,
and thus the root separation $(C,D\cup B)$ is reducible.
This contradicts the choice of $(A,B)$.
\end{proof}

In particular, Lemma~\ref{lemma-noclicut} implies that minimal counterexamples do not contain
root separations of order at most two, and moreover, we can constrain root separations of order three.
A rooted graph is \emph{essentially $4$-connected} if it is internally $3$-connected
and for every proper root separation $(A,B)$ of $G$ of order $3$, we have $|B\setminus A|=1$ and $A\cap B$ is an independent set in $G$.
In particular, this has the following consequences.
\begin{corollary}\label{cor-3conn}
Every minimal counterexample $G$ is essentially $4$-connected, and the underlying graph of $G$ is connected.
\end{corollary}
\begin{proof}
Let $G$ be a minimal counterexample, and let $(A,B)$ be a root separation of $G$ such that $B\not\subseteq A$,
chosen so that the order $t$ of $(A,B)$ is the smallest possible.
By Lemma~\ref{lemma-noclicut}, the root separation $(A,B)$ is not $K_\star$-universal, and in particular $t\ge 2$;
hence, $G$ is internally $2$-connected.

If $t=2$, then since $G$ is internally $2$-connected, the graph $R_{A,B}$ is connected, and in particular contains
a path between the two vertices of $A\cap B$.
However, this implies that the root separation $(A,B)$ is $K_\star$-universal, contradicting Lemma~\ref{lemma-noclicut}.
Therefore, $t\ge 3$ and $G$ is internally $3$-connected.

Suppose now that $t=3$.  If $R_{A,B}$ contained a cycle, then by internal $3$-connectivity, we could connect it by three vertex-disjoint paths
to the vertices of $A\cap B$, obtaining $K_3$ as a rooted minor.  Since the root separation $(A,B)$ is not
$K_\star$-universal, the graph $R_{A,B}$ is a tree. More precisely, since $R_{A,B}$ is internally 3-connected,
$R_{A,B}$ can only consist of a single vertex not in $A\cap B$ adjacent to all vertices of $A\cap B$ (and $A\cap B$ is an independent set).
Since this holds for every proper root separation of $G$ of order three, it follows that $G$
is essentially $4$-connected.

Finally, suppose for a contradiction that the underlying graph of $G$ is not connected, and thus $G$ is the disjoint union of two rooted
graphs $G_1$ and $G_2$.  Since $2\le \rho_4(G)=\rho_4(G_1)+\rho_4(G_2)$, we can assume that $\rho_4(G_1)>0$.
Since $G$ is 4-light, this implies $|X_{G_1}|>4$, and thus $|X_{G_1}|=5$ and $|X_{G_2}|=0$.
However, then $(V(G_1),V(G_2))$ is a proper root separation of $G$ of order $0$, which is a contradiction.
\end{proof}

Lemma~\ref{lemma-noclicut} also implies that in minimal counterexamples, all ``sufficiently nontrivial'' root separations of order five
are linked to the roots.
\begin{corollary}\label{cor-litor}
Let $G$ be a minimal counterexample and let $(A,B)$ be a root separation of $G$ of order five.
If $\rho_4(R_{A,B})\ge 3$, then the root separation $(A, B)$ is linked to the roots.
\end{corollary}
\begin{proof}
Let $(C,D)$ be an isolator for $(A,B)$, and suppose for a contradiction that $|C\cap D|\le 4$.
Since $G$ is essentially $4$-connected and $|D|\ge |A\cap B|=5$, we have $|C\cap D|=4$.
Let $\PP$ be an isolator linkage for $(C,D)$ and $(A,B)$ and let $Y$ be the set of its terminators.

We claim that the root separation $(C,D)$ is $K^-_\star$-universal, or equivalently, that for
all distinct $u,v\in C\cap D$, the 4-rooted graph $R_{C,D}$ contains the $\{u,v\}$-star as an $\id$-rooted
minor.  Indeed, let $u',v'\in A\cap B$ be the terminators of the paths of $\PP$ with origins $u$ and $v$, respectively.
By Corollary~\ref{cor-nok4}, the 5-rooted graph $R_{A,B}$ is $(K_4+K_1)$-universal or $K_{2,3}^+$-universal.
In the former case, we can contract $R_{A,B}$ to the clique on $Y$, and in the latter case to the $\{u',v'\}$-star.
By further contracting the paths of $\PP$, we obtain (a supergraph of) the $\{u,v\}$-star as an $\id$-rooted minor of $R_{C,D}$.

Moreover, $\rho_4(R_{A,B})\ge 3$ implies $|B\setminus A|\ge 2$, and thus $|D\setminus C|\ge 3$.
This contradicts Lemma~\ref{lemma-noclicut} applied to the root separation $(C,D)$.
\end{proof}

With this, we can show that the right-hand sides of proper root separations of order five in a minimal counterexample
are close to being 4-light.

\begin{lemma}\label{lemma-ubdel}
Let $G$ be a minimal counterexample.  If $(A,B)$ is a proper root separation of $G$ of order five, then
$\rho_4(R_{A,B})\le 4$.
\end{lemma}
\begin{proof}
If $(A,B)$ is not linked to the roots, then $\rho_4(R_{A,B})\le 2$ by Corollary~\ref{cor-litor}.
Hence, suppose that $(A,B)$ is linked to the roots, and let $\PP$ be a root linkage for $(A,B)$.
Since the underlying graph of $G$ is connected, each component of $L_{A,B}$ contains at least one of the paths of $\PP$.
Since the root separation $(A,B)$ is proper, $L_{A,B}$ has a component $C$ such that $V(C)\not\subseteq X_G$.
Let $m$ be the number of paths of $\PP$ contained in $C$; we additionally choose $C$ so that $m$ is smallest possible.

Let $C'=C\cup X_G$ and $D=V(G)\setminus (V(C)\setminus (A\cap B))$, and note that $(C',D)$ is a root separation of $G$ of order five.
Let $G'$ be an $\id$-rooted minor of $G$ obtained from $G$ by contracting each path of $\PP$ contained in $C$ to a single vertex,
then contracting the rest of the component $C$ to the resulting root vertices so that the subgraph $G'[X_G\cap V(C)]$ is connected.
Note that $G'$ is obtained from $R_{C',D}$ by possibly adding some edges between roots, and thus $G'$ is 4-light and $\rho_4(G')=\rho_4(R_{C',D})$.
Moreover, we have $n(G')=n(G)-|V(C)\setminus X_G|<n(G)$, and thus Observation~\ref{obs-lightminor} gives
\begin{equation}\label{eq-lcpd}
\rho_4(R_{C',D})=\rho_4(G')<\rho_4(G).
\end{equation}
Suppose now that $m\le 2$, and let $X$ be the set $(X_G\cup (A\cap B))\cap V(C)$ consisting exactly of the origins and terminators of the paths of $\PP$ contained in $C$.
Let $a=|X\setminus (A\cap B)|=|X\setminus X_G|$.
We have $|X|=m+a\le 2m\le 4$, and since $G$ is $4$-light, it follows that $\rho_4(\roots{C}{X})\le 0$.
Let $t=|E(C[X])|$, and note that since $X_G$ is an independent set and $C$ is a component of $L_{A,B}$,
$t$ is the number of edges of $G[X_G\cup X]$ and all these edges join vertices of $X\setminus X_G$ to other vertices of $X$.
Consequently $t\le ma+\binom{a}{2}\le (m+a)a\le 4a$.
Moreover, note that $\rho_4(R_{C',D})=\rho_4(G)-\rho_4(\roots{C}{X})-t+4a\ge \rho_4(G)$.  This contradicts (\ref{eq-lcpd}).

Therefore, we have $m\ge 3$.  Since $|\PP|=5$, each component of $L_{A,B}$ other than $C$ contains at most two paths of $\PP$,
and by the choice of $C$ so that $m$ is smallest possible, it follows that each such component $C'$ satisfies $V(C')\subseteq X_G$.
Therefore, we have $(C',D)=(A,B)$ and
$$\rho_4(R_{A,B})=\rho_4(G').$$
Since the graph $G'[C\cap X_G]$ is connected, it has at least $m-1$ edges, and thus the graph $K_5-E(G'[X_G])$ has at most $11-m\le 8$ edges.
By Observation~\ref{obs-nomg}, the graph $K_5-E(G'[X_G])$ does not belong to the target of $G'$, and
thus $\rho_4(R_{A,B})=\rho_4(G')\le 4$.
\end{proof}

\section{Connectivity}

From Corollary~\ref{cor-3conn}, we know that the underlying graph of every minimal counterexample $G$ has to be connected.
However, it might be the case that the graph $G-X_G$ has more than one component.  Hence, we can express $G$ as the union of two $5$-rooted
graphs $G'$ and $G''$ intersecting only in the roots and satisfying $n(G'),n(G'')\ge 1$.  Let $F$ be a flaw of $G$.  To address this case, we might
first contract $G''$ to obtain some edges of $F$ and let $F'$ be the proper subgraph of $F$ formed by the remaining edges.
We now need to show that $F'$ is an $\id$-rooted minor of $G'$, and by the minimality of the counterexample $G$, it suffices to argue that $F'$
is contained in the target of $G'$.  Since $\rho_4(G')=\rho_4(G)-\rho_4(G'')$ may be smaller than $\rho_4(G)$,
this is not automatic.  In the following observations, we characterize the problematic cases depending on the
difference between $\rho_4(G)$ and $\rho_4(G')$.
\begin{observation}\label{obs-trade}
Let $G$ be a counterexample and let $G'$ be a 5-rooted graph such that $\rho_4(G')=\rho_4(G)-1$.
Let $F$ be a flaw of $G$, let $F'$ be a proper spanning subgraph of $F$, and let $\delta=|E(F)|-|E(F')|$.
If $F'$ does not belong to the target of $G'$, then
\begin{itemize}
\item[(i)] $\rho_4(G)=5$, $|E(F)|=8$, and $\delta=1$, or
\item[(ii)] $\rho_4(G)=4$, $|E(F)|=6$, and $\delta=1$, or
\item[(iii)] $\rho_4(G)=4$, $|E(F)|=6$, $\delta=2$, and $F'$ is isomorphic to $K_2+K_3$, or
\item[(iv)] $\rho_4(G)=2$, $|E(F)|=3$, and $\delta=1$.
\end{itemize}
\end{observation}
\begin{proof}
Since $F$ belongs to the target $\SS$ of $G$ and $|E(F)|\ge\delta\ge 1$, we have $\rho_4(G)\ge 1$.
Moreover, since $F'$ does not belong to the target $\SS'$ of $G'$, we have
$\rho_4(G')\le 6$, and thus $\rho_4(G)=\rho_4(G')+1\le 7$.
We now perform a straightforward case analysis summarized in the following table
(the value of $|E(F)|$ given in the table is exact, since $F$ is a flaw of $G$,
and thus $F\in \SS$ but no proper supergraph of $F$ belongs to $\SS$).

\begin{center}
\begin{tabular}{>{\rowmac}c|>{\rowmac}c>{\rowmac}c>{\rowmac}c>{\rowmac}c>{\rowmac}c<{\clearrow}}
$\rho_4(G)$		&$\SS$		&$|E(F)|$	&$\SS'$		&$F'\not\in\SS'\Rightarrow |E(F')|\ge$	&$\delta\le$\\
\hline
$7$			&$\SS_{5,10}$	&10		&$\SS_{5,9}$	&10					&0\\
$6$			&$\SS_{5,9}$	&9		&$\SS_{5,8}$	&9					&0\\
\setrow{\color{red}}$5$	&$\SS_{5,8}$	&8		&$\SS_{5,6}$	&7					&1\\
\setrow{\color{red}}$4$	&$\SS_{5,6}$	&6		&$\SS_{5,4}^-$	&4					&2\\
$3$			&$\SS_{5,4}^-$	&4		&$\SS_{5,3}$	&4					&0\\
\setrow{\color{red}}$2$	&$\SS_{5,3}$	&3		&$\SS_{5,1}$	&2					&1\\
$1$			&$\SS_{5,1}$	&1		&$\SS_{5,0}$	&1					&0\\
\end{tabular}
\end{center}
Since $F'$ is a proper spanning subgraph of $F$, we have $\delta\ge 1$, and thus only the
rows shown in red are possible.  The outcomes of this Observation then correspond to the only four ways
the constraints expressed in the table can be satisfied.
\end{proof}

\begin{observation}\label{obs-trade2}
Let $G$ be a counterexample and let $G'$ be a 5-rooted graph such that $\rho_4(G')=\rho_4(G)-2$.
Let $F$ be a flaw of $G$, let $F'$ be a spanning subgraph of $F$,
and let $\delta=|E(F)|-|E(F')|$.
If $\delta\ge 3$ and $F'$ does not belong to the target of $G'$, then
\begin{itemize} 
\item[(i)] $\rho_4(G)=5$, $|E(F)|=8$, and $\delta=3$, or
\item[(ii)] $\rho_4(G)=5$, $|E(F)|=8$, $\delta=4$, and $F'$ is isomorphic to $K_2+K_3$.
\end{itemize}
\end{observation}
\begin{proof}
Since $F$ belongs to the target $\SS$ of $G$ and $|E(F)|\ge \delta\ge 3$, we have $\rho_4(G)\ge 2$.
Moreover, since $F'$ does not belong to the target $\SS'$ of $G'$ and $|E(F')|\le \binom{5}{2}-\delta\le 7$, we have
$\rho_4(G')\le 4$, and thus $\rho_4(G)=\rho_4(G')+2\le 6$.
We now perform a straightforward case analysis summarized in the following table:

\begin{center}
\begin{tabular}{>{\rowmac}c|>{\rowmac}c>{\rowmac}c>{\rowmac}c>{\rowmac}c>{\rowmac}c<{\clearrow}}
$\rho_4(G)$		&$\SS$		&$|E(F)|$	&$\SS'$		&$F'\not\in\SS'\Rightarrow |E(F')|\ge$	&$\delta\le$\\
\hline
$6$			&$\SS_{5,9}$	&9		&$\SS_{5,6}$	&7					&2\\
\setrow{\color{red}}$5$	&$\SS_{5,8}$	&8		&$\SS_{5,4}^-$	&4					&4\\
$4$			&$\SS_{5,6}$	&6		&$\SS_{5,3}$	&4					&2\\
$3$			&$\SS_{5,4}^-$	&4		&$\SS_{5,1}$	&2					&2\\
$2$			&$\SS_{5,3}$	&3		&$\SS_{5,0}$	&1					&2\\
\end{tabular}
\end{center}
Since $\delta\ge 3$, only the row shown in red is possible.  The outcomes of this Observation then correspond to the only two ways
the constraints expressed in the table can be satisfied.
\end{proof}

\begin{observation}\label{obs-trade3}
Let $G$ be a counterexample and let $G'$ be a 5-rooted graph such that $\rho_4(G')=\rho_4(G)-3$.
Let $F$ be a flaw of $G$, let $F'$ be a spanning subgraph of $F$,
and let $\delta=|E(F)|-|E(F')|$.
If $\delta\ge 5$ and $F'$ does not belong to the target of $G'$, then
$\rho_4(G)=6$, $|E(F)|=9$, $\delta=5$, and $F'$ is isomorphic to $K_2+K_3$.
\end{observation}
\begin{proof}
Since $F$ belongs to the target $\SS$ of $G$ and $|E(F)|\ge \delta\ge 5$, we have $\rho_4(G)\ge 4$.
Moreover, since $F'$ does not belong to the target $\SS'$ of $G'$ and $|E(F')|\le \binom{5}{2}-\delta\le 5$, we have
$\rho_4(G')\le 3$, and thus $\rho_4(G)=\rho_4(G')+3\le 6$.
We now perform a straightforward case analysis summarized in the following table:

\begin{center}
\begin{tabular}{>{\rowmac}c|>{\rowmac}c>{\rowmac}c>{\rowmac}c>{\rowmac}c>{\rowmac}c<{\clearrow}}
$\rho_4(G)$		&$\SS$		&$|E(F)|$	&$\SS'$		&$F'\not\in\SS'\Rightarrow |E(F')|\ge$	&$\delta\le$\\
\hline
\setrow{\color{red}}$6$	&$\SS_{5,9}$	&9		&$\SS_{5,4}^-$	&4					&5\\
$5$			&$\SS_{5,8}$	&8		&$\SS_{5,3}$	&4					&4\\
$4$			&$\SS_{5,6}$	&6		&$\SS_{5,1}$	&2					&4\\
\end{tabular}
\end{center}
Since $\delta\ge 5$, only the row shown in red is possible.
Consequently $\rho_4(G)=6$, $|E(F)|=9$, $\delta=5$, and $|E(F')|=4$ and $F'\not\in \SS_{5,4}^-$, i.e., $F'$ is isomorphic to $K_2+K_3$.
\end{proof}

Let us now proceed with the proof of the promised connectivity result.

\begin{lemma}\label{lemma-connwiro}
If $G$ is a minimal counterexample, then the graph $G-X_G$ is connected and each root of $G$ has at least one non-root neighbor.
\end{lemma}
\begin{proof}
Let $C_1$, \ldots, $C_k$ be the components of $G-X_G$.  For $i\in \{1,\ldots,k\}$, let $B_i$ be the set consisting of $V(C_i)$
and of the roots with at least one neighbor in $C_i$, and let $A_i=(V(G)\setminus B_i)\cup X_G$; then $(A_i, B_i)$ is
a root separation of $G$.  Let $G_i=R_{A_i,B_i}$, and let $c_i=\rho_4(G_i)=\rho_4(R_{A_i,B_i})$.
Clearly $\rho_4(G)=c_1+\cdots+c_k$.
If there existed $i$ such that $c_i\le 0$,
then the 5-rooted graph $G-V(C_i)$ would be 4-light and its 4-density would be at least $\rho_4(G)$, and
thus it would be a counterexample contradicting the minimality of $G$.
Therefore, we have $c_i\ge 1$ for every $i\in \{1,\ldots,k\}$.  Since $G$ is 4-light,
it follows that $|A_i\cap B_i|=5$ and $X_{G_i}=X_G$.
In other words, every root vertex has at least one neighbor in every component of $G-X_G$.

Suppose now for a contradiction that $k\ge 2$.  For any set $S\subsetneq\{1,\ldots,k\}$,
let $G_S=\bigcup_{i\in S} G_i$ and note that by Observation~\ref{obs-univsep}, the 5-rooted graph $G_S$ is universal.

Suppose first that $c_1,\ldots,c_k\ge 2$.  By Corollary~\ref{cor-nok4}, we can contract $G_2$ to a graph $H$ on $X_G$ with seven edges.  Moreover, 
since $c_1\ge 2$, the 5-rooted graph $G_1$ is $\SS_{5,3}$-universal, and thus $\overline{H}$ is an $\id$-rooted minor of $G_1$.
It follows that $K_5$ is a rooted minor of $G$, which is a contradiction.

Therefore, we can assume that $c_k=1$.  Let $F$ be a flaw of $G$, let $v$ be a vertex of $F$ of maximum degree,
and let $\delta=\deg_F v$; note that $|E(F)|\le \tfrac{\delta\cdot |V(F)|}{2}=\tfrac{5}{2}\delta$.  Let $G'$ be the union of $G_{\{1,\ldots,k-1\}}$ with the $v$-star
and let $F'=F-E(G'[X_G])$ be the graph obtained from $F$ by removing the edges incident with $v$.  By Corollary~\ref{cor-nok4}, the $5$-rooted graph $G_k$ is $K_{1,4}$-universal,
and thus $G'$ is an $\id$-rooted minor of $G$.  By Observation~\ref{obs-nomg}, it follows that $F'$ is not an $\id$-rooted minor of $G'$.
Since $G'$ is universal, this implies that $F'$ does not belong to the target of $G'$.
Moreover, $\rho_4(G')=\rho_4(G)-c_k=\rho_4(G)-1$, and thus we can apply Observation~\ref{obs-trade}.
However if $\delta=2$, then $|E(F)|\le \tfrac{5}{2}\delta=5$, and if $\delta=1$, then $|E(F)|\le \tfrac{5}{2}\delta<3$,
which excludes all the outcomes of this observation.

This contradiction shows that $k=1$, and thus the graph $G-X_G$ is connected.
\end{proof}

We now continue in a similar vein, excluding another kind of special root separations (depicted in Figure~\ref{fig-nobisep}).
In preparation for that, let us show the following technical lemma.

\begin{lemma}\label{lemma-canconto}
Let $D$ be a 5-rooted universal graph such that $0\le\rho_4(D)\le 4$, and moreover if $\rho_4(D)\ge 1$, then $D$ is $K_{1,4}$-universal.
Let $x$ and $v$ be distinct roots of $D$ such that all edges of $D[X_D]$ are incident with $v$.  Let $S$ be a 5-rooted graph obtained from $D$ by adding
a new vertex $y$, possibly adding the edge $vy$, and letting $X_S=(X_D\setminus\{v\})\cup \{y\}$.
Let $F$ be a graph with $V(F)=X_S$ and without isolated vertices.
If $\rho_4(S)\ge 1$, then there exists an $\id$-rooted model $\mu$ of a graph $H$ with $V(H)=X_S$ in $S$
such that the following claims hold. Let $\delta=|E(F\cap H)|$ and $F'=F-E(H)$.
\begin{itemize}
\item The graph $F'$ has no edge incident with $x$ except possibly the edge $xy$, and if $xy\in E(F')$, then $v\in\mu(x)$.
\item If $F'-y$ has at least one edge, then at least one of the following conditions holds:
\begin{itemize}
\item[(a)] $\rho_4(S)=1$ and
\begin{itemize}
\item $\delta=1$, $F'$ has at least two isolated vertices, and one of them is $x$; or
\item $\delta=2$ and either $x$ is an isolated vertex of $F'$ or $|E(F'-\{x,y\})|=1$; or 
\item $\delta\ge 3$.
\end{itemize}
Or,
\item[(b)] $\rho_4(S)=2$, $\delta\ge 3$, $x$ is an isolated vertex of $F'$, and $|E(F'-\{x,y\})|=1$. Or,
\item[(c)] $\rho_4(D)=3$, $F-\{x,y\}$ is a triangle, $|E(F'-\{x,y\})|=1$, $\deg^X_S v\le 4$,
and if $\deg^X_S v=4$, then $x$ is an isolated vertex of $F'$.
\end{itemize}
Let us remark that the last case intentionally refers to $\rho_4(D)$ rather than $\rho_4(S)$.
\end{itemize}
\end{lemma}
\begin{proof}
Let $s=\deg^X_S v$ and note that $\rho_4(S)=\rho_4(D)+s-4$.

Suppose first that $\rho_4(S)=1$ and $vy\in E(S)$.
If $s\le 4$, then $\rho_4(D)=\rho_4(S)+4-s\ge 1$, and by the assumptions $D$ is $K_{1,4}$-universal.
Hence, we can contract $D$ to the $v$-star, ensuring that the resulting $\id$-rooted minor of $S$
contains all edges between $v$ and $X_S$ (including the edge $vy$ present by the assumptions).  This is also automatically the case when $s=5$.
If $\deg_F x\ge 2$, then we obtain $H$ by further contracting the edge $vx$;
this ensures that $x$ is an isolated vertex of $F'$ and $\delta\ge 2$.
Otherwise, since $F$ does not have isolated vertices, the vertex $x$ has exactly one neighbor $z$ in $F$.
In this case we let $H$ be obtained by instead contracting the edge $vz$; this ensures that $\delta\ge 1$ and
$x$ and $z$ are isolated vertices of $F'$.
In both cases, the conclusion (a) is satisfied.
Therefore, we can assume that
\begin{equation}\label{eq-s1novy}
\text{if $\rho_4(S)=1$, then $vy\not\in E(S)$.}
\end{equation}
In particular either $\rho_4(S)\ge 2$ or $s\le 4$, and thus $\rho_4(D)=\rho_4(S)+4-s\ge 1$.  

Let $\SS'$ be the target of $D$ and let $\SS=\SS'\cup \{K_{1,4}\}$;
by the assumptions, the 5-rooted graph $D$ is $\SS$-universal.  Let $M$ be the set of all pairs $uv$ where $u\in X_S\setminus\{y\}$
is a non-neighbor of $v$ in $S$ such that $u=x$ or $ux\in E(F)$.  Let $Q_0$ be the graph with vertex set $X_D$ and with edge set $M\cup E(F-\{x,y\})$,
and let $Q$ be a spanning subgraph of $Q_0$ such that $M\subseteq E(Q)$ and $Q$ is an $\id$-rooted minor of $D$,
chosen so that $|E(Q)|$ is maximum possible (such a subgraph exists, since all edges of $M$ are incident with the same vertex and $D$ is $K_{1,4}$-universal).

Let $H$ be the $\id$-rooted minor of $S$ obtained by first contracting $D$ to $Q$ (retaining the already existing edges of $D$ between $v$ and $X_D\setminus\{v\}$), then contracting the edge $vx$
(if $vx\not\in E(S)$, then $vx\in M$, and thus this edge is created in the first contraction).
By the definition of $M$ and $Q$, this ensures that $F'$ has no edge incident with $x$ except possibly for the edge $xy$.
Moreover, we always have $v\in\mu(x)$.

If $E(F'-y)=\emptyset$, then there is nothing more to prove.  Hence, suppose that this is not the case.
Since $x$ is an isolated vertex of $F'-y$, this means that $E(F'-\{x,y\})\neq\emptyset$, and thus $Q\neq Q_0$.
By the maximality of $|E(Q)|$, this in particular implies that $Q_0\not\in \SS$.  Observe that
\begin{equation}\label{eq-delta}
\delta\ge |E(Q)\setminus\{vx\}|.
\end{equation}
Moreover, if $vy\in E(S)$, then $x$ is an isolated vertex of $F'$.  Since $x$ is not isolated in $F$, we furthermore have
\begin{equation}\label{eq-deltabetter}
\text{if $M\subseteq\{vx\}$ and $vy\in E(S)$, then $\delta\ge |E(Q)\setminus\{vx\}|+1$.}
\end{equation}
Let us now discuss cases based on the value of $\rho_4(D)$.
\begin{itemize}
\item If $\rho_4(D)=1$, then either $s=4$ and $\rho_4(S)=1$, or $s=5$ and $\rho_4(S)=2$.
By (\ref{eq-s1novy}), in the former case we have $vy\not\in E(S)$.  Consequently, in either case $S$ contains all edges between $v$ and $X_S\setminus\{y\}$,
and consequently $M=\emptyset$.   It follows that $E(Q_0)=E(F-\{x,y\})$.  Since $D$ is $K_{1,4}$-universal and $Q\neq Q_0$,
the maximality of $|E(Q)|$ implies that $F-\{x,y\}$ is a triangle, $|E(Q)|=2$, and $|E(F'-\{x,y\})|=1$.
If $vy\in E(S)$, then $x$ is an isolated vertex of $F'$ and $\delta\ge 3$ by (\ref{eq-deltabetter}).
If $vy\not\in E(S)$, then $s=4$, $\rho_4(S)=1$, and $\delta\ge 2$ by (\ref{eq-delta}).
In either case, the conclusions (a) or (b) are satisfied.

\item If $\rho_4(D)=2$, then we have $s\ge 3$, $\rho_4(S)=s-2$, and $D$ is $\SS_{5,3}$-universal by the assumptions.
Since $Q_0\not\in \SS_{5,3}$, we have $|E(Q_0)|\ge 4$, and the maximality of $|E(Q)|$ implies that $|E(Q)|\ge 3$.
Moreover, note that $|M|\le 1$, since either $s\ge 4$, or $s=3$ and (\ref{eq-s1novy}) implies $vy\not\in E(S)$.
Since $|E(Q_0)|\ge 4$, observe that $F-\{x,y\}$ is a triangle, $|M|=1$, and $|E(F'-\{x,y\})|=1$.  Moreover, $|M|=1$ implies $s\le 4$ and $\rho_4(S)\le 2$.
If $vy\in E(S)$, then $x$ is an isolated vertex of $F'$, and depending on whether $M=\{vx\}$ or not,
(\ref{eq-deltabetter}) or (\ref{eq-delta}) implies that $\delta\ge 3$.
If $vy\not\in E(S)$, then $s=3$, $\rho_4(S)=1$, and (\ref{eq-delta}) implies $\delta\ge 2$.
In either case, the conclusions (a) or (b) are satisfied.

\item If $\rho_4(D)=3$, then we have $s\ge 2$, $\rho_4(S)=s-1$, and $D$ is $\SS_{5,4}^-$-universal by the assumptions.

If $Q_0$ is isomorphic to $K_2+K_3$, then since $Q_0$ does not contain any edges between $x$ and $X_S\setminus\{x,y\}$,
it follows that $M=\{vx\}$ and $F-\{x,y\}$ is a triangle. By the maximality of $|E(Q)|$, we have $|E(Q)|=3$,
and thus $|E(F'-\{x,y\})|=1$.  Since $vx\not\in E(S)$, we have $s\le 4$.  Moreover, if $s=4$, then $vy\in E(S)$
and $xy\not\in E(F')$.  Hence, the conclusion (c) is satisfied.

On the other hand, if $Q_0$ is not isomorphic to $K_2\cup K_3$, then since $Q_0\not\in \SS_{5,4}^-$, we have $|E(Q_0)|\ge 5$.
Observe that $Q_0$ has a spanning subgraph $Q'$ such that $M\subseteq E(Q')$, $|E(Q')|=4$, and $Q'$ is not isomorphic to $K_2\cup K_3$.
The maximality of $|E(Q)|$ implies that $|E(Q)|\ge 4$.  By (\ref{eq-delta}), we have $\delta\ge 3$, and thus the conclusions of the Lemma hold
when $\rho_4(S)=1$, that is, $s=2$.  Hence, we can assume that $s\ge 3$, and thus $|M|\le 2$.  Since $5\le |E(Q_0)|=|E(F-\{x,y\})|+|M|$,
it follows that $|M|=2$ (and thus $s=3$) and $|E(F-\{x,y\})|=3$ (i.e., $F-\{x,y\}$ is a triangle).
Moreover, since $|E(Q)|\ge 4$ and $|M|=2$, we have $|E(F'-\{x,y\})|=1$.  Therefore, the conclusion (c) is again satisfied.

\item Finally, suppose that $\rho_4(D)=4$, and thus $s\ge 1$, $\rho_4(S)=s$, and $D$ is $\SS_{5,6}$-universal by the assumptions.  
Since $Q_0\not\in \SS_{5,6}$, we have $|E(Q_0)|\ge 7$; this implies that $|M|=4$, and thus $s=1$ and $\rho_4(S)=1$.
Moreover, we have $|E(Q)|=6$ by the maximality of $|E(Q)|$ and $\delta\ge 5$ by (\ref{eq-delta}), and thus the conclusion (a) holds.
\end{itemize}
\end{proof}

\begin{figure}
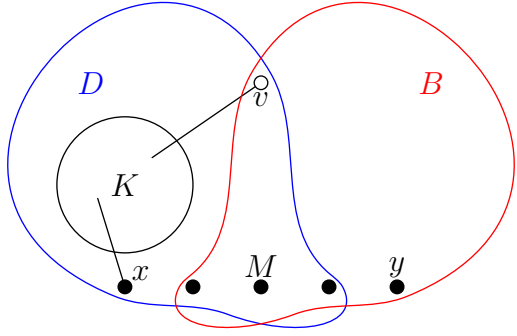

\begin{center}
\begin{asy}
v[0] = (0,0);
v[1] = (1,0);
v[2] = (2,0);
v[3] = (3,0);
v[4] = (4,0);
v[5] = (2,3);

v[6] = v[0] + (0,1.5);
draw (v[0] -- (v[6] + (-0.4,-0.2)));
draw (v[5] -- (v[6] + (0.4,0.4)));

draw ((v[0]+0.2SW) .. (interp(v[0],v[3],0.5)+(0,-0.4)) .. (v[3]+0.2E) .. (v[3]+0.2N) .. (v[5]+0.2NE){dir(120)} .. cycle, blue);
label ("$D$", (-0.5,3), blue);
draw ((v[1]+0.2N) .. (v[1]+0.2W) .. ((interp(v[1],v[4],0.5)+(0,-0.4)) .. v[4]+0.2SE) .. (v[5]+0.2NW){dir(240)} .. cycle, red);
label ("$B$", (4.5,3), red);

label("$x$", v[0], NE);
label("$y$", v[4], N);
label("$v$", v[5], S);
label("$M$", v[2], N);

draw (circle (v[6],1));
label ("$K$", v[6]);

vertex(v[0], black);
vertex(v[1], black);
vertex(v[2], black);
vertex(v[3], black);
vertex(v[4], black);
vertex(v[5]);
\end{asy}
\end{center}
\caption{The setting of Lemma~\ref{lemma-nobisep}.}\label{fig-nobisep}
\end{figure}

We are now ready to show that minimal counterexamples do not contain separations as depicted in Figure~\ref{fig-nobisep}.

\begin{lemma}\label{lemma-nobisep}
Let $G$ be a minimal counterexample and let $(A,B)$ be a root separation of $G$ of order five
such that $|B\cap X_G|=4$.  Then each vertex in $B\cap X_G$ has a neighbor in $V(G)\setminus (B\cup X_G)$.
\end{lemma}
\begin{proof}
Suppose for a contradiction that a vertex $y\in B\cap X_G$ does not have any neighbor in $V(G)\setminus (B\cup X_G)$.
Let $x$ be the unique root in $X_G\setminus B$, let $M=X_G\setminus\{x,y\}$, and let $v$ be the unique vertex in $A\cap B\setminus X_G$.
Let $C=B\cup X_G$ and $D=A\setminus\{y\}$ and observe that $(C,D)$ is another root separation of $G$ of order five.
Note that $G=R_{A,B}\cup R_{C,D}$ and $R_{A,B}\cap R_{C,D}=G[M\cup\{v\}]$.
Let $s=\deg^X v\le 5$ be the number of edges between $v$ and $X_G$, and note that
\begin{equation}\label{eq-deco}
\rho_4(G)=\rho_4(R_{A,B})+\rho_4(R_{C,D})+s-4.
\end{equation}
By Lemma~\ref{lemma-ubdel} and the symmetry between the root separations $(A,B)$ and $(C,D)$, we can also assume that
\begin{equation}\label{eq-abheavy}
\rho_4(R_{C,D})\le \rho_4(R_{A,B})\le 4.
\end{equation}
Let $F$ be a flaw of $G$.  Moreover, further exploiting the symmetry, we can assume that
\begin{equation}\label{eq-missleft}
\text{if $\rho_4(R_{C,D})=\rho_4(R_{A,B})$, then $\deg_F x\ge \deg_F y$.}
\end{equation}

Note that $R_{C,D}$ contains a path $P$ from $x$ to $v$ intersecting $C\cap D$ only in its ends:
By Lemma~\ref{lemma-connwiro}, $x$ has a non-root neighbor, necessarily belonging to $D\setminus X_G$.
Hence, either $xv\in E(G)$ and we can set $P=xv$, or $x$ has a neighbor in $D\setminus C$.
In the latter case, let $K$ be a component of $G[D\setminus C]$ containing a neighbor of $x$.
Since $G-X_G$ is connected by Lemma~\ref{lemma-connwiro}, $K$ also contains a neighbor of $v$,
and thus we can find the desired path $P$ through $K$.
We can now contract $P$ and obtain an $\id$-rooted minor of $G$ isomorphic to $R_{A,B}$.
Thus, Observation~\ref{obs-lightminor} implies that $\rho_4(R_{A,B})<\rho_4(G)$,
and by (\ref{eq-deco}) we have $\rho_4(R_{C,D})\ge 5-s\ge 0$.

Let $S=G-(B\setminus A)$.  Note that $\rho_4(S)=\rho_4(G)-\rho_4(R_{A,B})\ge 1$.
Moreover, note that $R_{C,D}$ is universal by Observation~\ref{obs-univsep},
and if $\rho_4(R_{C,D})\ge 1$, then $R_{C,D}$ is $K_{1,4}$-universal by Corollary~\ref{cor-nok4}.
By Corollary~\ref{cor-nearfour}, $F$ does not have isolated vertices.
Therefore, we can apply Lemma~\ref{lemma-canconto} (with $R_{C,D}$ playing the role of the rooted graph $D$ from the statement of this lemma);
let $\mu$, $H$, $\delta$, and $F'$ be as in the conclusions.

Let $F''$ be the graph with vertex set $A\cap B$ obtained from $F'$ by relabelling the vertex $x$ to $v$.
We claim that $F''$ is not an $\id$-rooted minor of $\widetilde{R}_{A,B}$.
Indeed, suppose for a contradiction that we could form an $\id$-rooted minor $S'$ of $G$
by contracting $\widetilde{R}_{A,B}$ to $F''$.
Note that $S'$ is a supergraph of $S$, and since $x$ has no neighbor in $F'$ other than possibly $y$,
the set $E(S')\setminus E(S)$ consists of the edges $E(F'-x)$, which join the vertices of $X_G=X_S$,
and possibly the edge $vy$.  Moreover, if $xy\in E(F')$, then $v\in\mu(x)$.
Hence, we can now further contract the connected subgraphs $S'[\mu(z)]$ for $z\in X_G$
and the resulting $\id$-rooted minor of $G$ is $H\cup F'\supseteq F$.  This is a contradiction, since $F$ is a flaw of $G$.
Let $G'$ be the 5-rooted graph obtained from $\widetilde{R}_{A,B}$ by relabelling the vertex $v$ to $x$,
so that $F'$ is not an $\id$-rooted minor of $G'$.
By Observation~\ref{obs-univsep}, $\widetilde{R}_{A,B}$ as well as $G'$ are universal, and thus $F'$ is not contained
in the target of $G'$.

Since $\rho_4(G)\ge 2$ by Observation~\ref{obs-42dense},
(\ref{eq-deco}) and (\ref{eq-abheavy}) give $$2\rho_4(G')\ge \rho_4(R_{A,B})+\rho_4(R_{C,D})=\rho_4(G)+4-s\ge 6-s,$$
and thus $\rho_4(G')\ge 1$.  By Corollary~\ref{cor-nok4}, this implies that $G'$ is $K_{1,4}$-universal.
Since $F'$ is not an $\id$-rooted minor of $G'$, not all edges of $F'$ are incident with $y$, and thus
$E(F'-y)\neq \emptyset$.
Consequently at least one of the conditions (a), (b), or (c) from the statement of Lemma~\ref{lemma-canconto} holds.
Let us discuss them separately.

\begin{itemize}
\item If (a) holds, then $\rho_4(G')=\rho_4(G)-\rho_4(S)=\rho_4(G)-1$ and $\delta\ge 1$.
Note that $F'$ is a spanning subgraph of $F$ with $|E(F')|=|E(F)|-\delta$, and thus
we can apply Observation~\ref{obs-trade}, which in particular implies that $\delta\le 2$.

If $\delta=2$, then by (a) either $x$ is an isolated vertex of $F'$, or $xy$ is the only
edge of $F'$ incident with $x$ and $|E(F'-\{x,y\})|=1$, and thus $F'$ is not isomorphic
to $K_2+K_3$.  Consequently the outcome (iii) of Observation~\ref{obs-trade} is excluded,
and thus $\delta=1$.

Hence, by (a) the graph $F'$ has two isolated vertices $x$ and $z$.  In particular $|E(F')|\le 3$
and $|E(F)|=|E(F')|+\delta\le 4$, excluding the outcomes (i) and (ii) of Observation~\ref{obs-trade}.
Consequently the outcome (iv) holds, and thus $|E(F)|=3$ and $|E(F')|=2$.  Since $x$ and $z$
are isolated vertices of $F'$, both edges of $F'$ are incident with the same vertex $u\in X_G\setminus\{x,z\}$.
However, since $G'$ is $K_{1,4}$-universal, this implies that $F'$ is an $\id$-rooted minor of $G'$,
which is a contradiction.

\item If (b) holds, then $\rho_4(G')=\rho_4(G)-\rho_4(S)=\rho_4(G)-2$ and $\delta\ge 3$.
Since $x$ is an isolated vertex of $F'$, the graph $F'$ is not isomorphic to $K_2+K_3$.
Moreover, we additionally have $|E(F'-\{x,y\})|=1$, and thus $|E(F')|\le 4$.
However, this contradicts Observation~\ref{obs-trade2}.

\item Finally, suppose that (c) holds, and in particular $\rho_4(R_{C,D})=3$,
and (\ref{eq-abheavy}) implies that $\rho_4(G')\in \{3,4\}$.
Since $F'$ cannot have any edge incident with $x$ other than $xy$ and since $|E(F'-\{x,y\})|=1$, we get
$|E(F')|\le 5$.  Since $F'$ does not belong to the target of $G'$, it follows that
$\rho_4(G')=3$ and the target of $G'$ is $\SS_{5,4}^-$.  Since $F'\not\in \SS_{5,4}^-$,
we have that either $|E(F')|=5$ or $F'$ is isomorphic to $K_2+K_3$, and thus $x$ is not an
isolated vertex of $F'$.  Consequently $xy\in E(F')$ and (c) implies that $s\le 3$.

By (\ref{eq-deco}), we have $\rho_4(G)=\rho_4(G')+\rho_4(R_{C,D})+s-4=s+2\le 5$.
Since $F$ is a flaw of $G$, it follows that $|E(F)|\le 8$, and thus $F$ has at least two non-edges.
Since $xy\in E(F)$ and $F-\{x,y\}$ is a triangle, these non-edges are between $\{x,y\}$ and $M=X_G\setminus\{x,y\}$.
By (\ref{eq-missleft}), at least one of these non-edges is incident with $y$.
Therefore, $E(F')$ consists of the edge $xy$, the single edge of $E(F'-\{x,y\})$, and at most two edges between $y$ and $M$,
and thus $|E(F')|\le 4$.  Moreover, if $|E(F')|=4$, then $\deg_{F'} y=3$, and thus $F'$ is not isomorphic to $K_2+K_3$.
It follows that $F'$ belongs to the target $\SS_{5,4}^-$ of $G'$, which is a contradiction.
\end{itemize}
In all cases, we obtained a contradiction, concluding the argument.
\end{proof}

In particular, we obtain an improvement to the bound on the degrees of roots
in a minimal counterexample.

\begin{corollary}\label{cor-twononroot}
If $G$ is a minimal counterexample, then each root of $G$ has at least two non-root neighbors.
\end{corollary}
\begin{proof}
By Lemma~\ref{lemma-connwiro}, each root of $G$ has at least one non-root neighbor.
Suppose now for a contradiction that a root $x$ of $G$ has exactly one non-root neighbor~$v$.
Let $A=X_G\cup\{v\}$ and $B=V(G)\setminus \{x\}$; then $(A,B)$ is a root separation of $G$ of order five
and $|B\cap X_G|=4$.  However, $V(G)\setminus (B\cup X_G)=\emptyset$, contradicting Lemma~\ref{lemma-nobisep}.
\end{proof}

\section{(Nearly) saturated separations}

For a rooted graph $G$ and a set $Y\subseteq V(G)$,
let $K^G_Y$ denote the graph obtained from the clique with vertex set $Y$ by deleting the edges between the vertices of $Y\cap X_G$.
That is, $K^G_Y$ contains all edges with at least one non-root end.
We say that a root separation $(A,B)$ of $G$ is \emph{saturated} if $K^G_{A\cap B}$ is an $\id$-rooted minor of $R_{A,B}$.
It is easy to see that contracting $R_{A,B}$ in $G$ to $K^G_{A\cap B}$ does not create any fundamentally new separations,
making it simple to argue about the 4-lightness of the resulting rooted graph $G'$.
Based on this, one might suspect that the argument from Lemma~\ref{lemma-noclicut} used to exclude $K_\star$-universal
separations should be easy to generalize to the saturated ones.  This is essentially the case, but
we run into complications arising from the fact that in this setting $(A,B)$ is not necessarily $4$-light
(since we cannot easily exclude the case that $(A,B)$ has order five), and thus we need a much more involved accounting
for the value of $\rho_4(G')$.

\begin{lemma}\label{lemma-nosatur}
A minimal counterexample $G$ does not have any saturated proper root separation of order at most five.
\end{lemma}
\begin{proof}
Suppose for a contradiction that there exists a saturated proper root separation of order at most five,
and let $(A,B)$ be one with $A$ minimal.  
Consider any isolator $(C,D)$ for $(A,B)$, and let $\PP$ be an isolator
linkage for $(C,D)$ and $(A,B)$.  Consider distinct vertices $u,v\in C\cap D$ such that $v\not\in X_G$, and let $u'$ and $v'$ be
the terminators of the paths $P_u,P_v\in \PP$ with origins $u$ and $v$, respectively.
Note that $v'\not\in X_G$; indeed, otherwise we would have $v'\in C\cap D$ and the path $P_v$ would consist only of $v'$,
i.e., we would have $v=v'\in X_G$.  In particular, we have $u'v'\in E(K^G_{A\cap B})$.  Since this holds for all distinct vertices $u,v\in C\cap D$ such that $v\not\in X_G$,
it follows that by first contracting $R_{A,B}$ to $K^G_{A\cap B}$, then contracting the paths of $\PP$,
we obtain a supergraph of $K^G_{C\cap D}$ as an $\id$-rooted minor of $R_{C,D}$.  Therefore, the root separation $(C,D)$ is saturated.
Since $|C\cap D|\le |A\cap B|$, the minimality of $A$ implies that if $(C,D)\neq (A,B)$, then the root separation $(C,D)$ is not proper.
Since $|C\cap D|\le |A\cap B|\le 5=|X_G|$, this can only be the case when $(A,B)$ has order five and $(C,D)=(X_G,V(G))$.  It follows that $(A,B)$ is strongly linked to the roots.

Let $G_0$ be the $\id$-rooted minor of $G$ obtained by contracting $R_{A,B}$ to $K^G_{A\cap B}$.
Consider any proper root separation $(C,D)$ of $G_0$ of order at most four.  Since the only isolators of $(A,B)$
are $(A,B)$ itself and $(X_G,V(G))$ when $(A,B)$ has order five, we cannot have $A\cap B\subseteq D$,
since then $(C,D\cup B)$ would be a different isolator.  Hence, there exists a vertex $v\in (A\cap B)\setminus D$.
The vertex $v$ is adjacent to all vertices of $A\cap B\setminus X_G$ in $G_0$, and thus $A\cap B\setminus X_G\subseteq C$.
We also have $X_G\subseteq C$, and thus $A\cap B\subseteq C$.
Note that $(C\cup B,D)$ is a root separation of $G$ of order at most four with the same strictly right-hand side as $(C,D)$,
and since $G$ is $4$-light, it follows that the root separation $(C,D)$ of $G_0$ is $4$-light.
Since this holds for all proper root separations of $G_0$ of order at most four, we conclude that $G_0$ is $4$-light.

Let $M=K^G_{A\cap B}-E(G[A\cap B])$ be the subgraph formed by the edges created by the contraction of $R_{A,B}$ to $K^G_{A\cap B}$,
and let $m=|E(M)|$. Note that
$$\rho_4(G_0)=\rho_4(L_{A,B})+m=\rho_4(G)-\rho_4(R_{A,B})+m.$$
Since the root separation $(A,B)$ is proper, we have $n(G_0)<n(G)$, and thus Observation~\ref{obs-lightminor}
gives $\rho_4(G_0)\le \rho_4(G)-1$. Therefore,
$$\rho_4(R_{A,B})\ge m+1\ge 1.$$
Since $G$ is $4$-light, this in particular implies that the root separation $(A,B)$ has order five.
Let $s=|A\cap B\cap X_G|$.  We cannot have $A\cap B=X_G$, since the root separation $(A,B)$ is proper and $G-X_G$ is connected by Lemma~\ref{lemma-connwiro}.
Moreover, the root separation $(A,B)$ is not $K_\star$-universal by Lemma~\ref{lemma-noclicut}, and thus $K^G_{A\cap B}$ is not a clique.
Therefore
$$2\le s\le 4.$$
Note also that $\rho_4(R_{A,B})\le 4$ by Lemma~\ref{lemma-ubdel}, and thus
$$m\le 3.$$
Let $y_1$, \ldots, $y_s$ be the vertices of $A\cap B\cap X_G$ and let $x_1$, \ldots, $x_{5-s}$ be the vertices of $X_G\setminus B$.
Let $\QQ=\{Q_1,\ldots,Q_5\}$ be a root linkage for $(A,B)$, where for $i\in\{1,\ldots,5-s\}$,
the origin of the path $Q_i$ is $x_i$; let $v_i$ denote the terminator of the path $Q_i$.
Let us remark that $A\cap B\setminus X_G=\{v_1,\ldots, v_{5-s}\}$.
Let $F$ be a flaw of $G$, let $F_X=F[A\cap B\cap X_G]$, and let $F_X^+$ be the graph obtained from $F_X$ by adding isolated vertices $v_1$, \ldots, $v_{5-s}$.  Let us now consider two special cases.
\begin{claim}\label{cl-s4}
If $s=4$, then $F_X$ is isomorphic to $K_4$, $\rho_4(R_{A,B})=m+1$, $\widetilde{R}_{A,B}$ is not $(K_4+K_1)$-universal, and $\rho_4(R_{A,B})\le 3$.
\end{claim}
\begin{subproof}
By Lemma~\ref{lemma-nobisep}, each vertex of $X_G\setminus \{x_1\}$ has a neighbor in $A\setminus (B\cup X_G)$.
Moreover, since the root separation $(A,B)$ is strongly linked to the roots, the vertex $x_1$ has a neighbor in each component $K$ of $G-(B\cup X_G)$,
as otherwise $(A\setminus V(K),B\cup V(K))$ would be an isolator of $(A,B)$.  Consequently, $G-B$ has a connected subgraph $S$ containing $x_1$ such that each vertex of $X_G\setminus \{x_1\}$
has a neighbor in $S$.

Let $G'$ be the 4-rooted graph with the underlying graph $G[B]$ and the roots $X_G\setminus\{x_1\}$.
If $F_X$ were an $\id$-rooted minor of $G'$, then we could contract $G'$ to $F_X$ and contract $S$ to $x_1$, obtaining a supergraph of $F$
as an $\id$-rooted minor of $G$.  Since $F$ is a flaw of $G$, this is not possible.  Hence, $F_X$ is not an $\id$-rooted minor of $G'$.
This in particular implies that $\widetilde{R}_{A,B}$ is not $(K_4+K_1)$-universal, since in that case $F_X$ would actually be an $\id$-rooted minor of $G'-v_1$.
By Corollary~\ref{cor-nok4}, this implies that $\rho_4(R_{A,B})\le 3$.

We now aim to apply Corollary~\ref{cor-k4} to $G'$.  First, we need to argue that $G'$ is $4$-light: Consider any root separation $(C,D)$ of $G'$ of order at most three,
and let $A'=A\cup C$.
If $v_1\in C$, then $(A',D)$ is a root separation of $G$ with the right-hand side equal to $R^{G'}_{C,D}$, and since $G$ is $4$-light, the separations $(A',D)$ and $(C,D)$
are $4$-light.  If $v_1\in D\setminus C$, then note that $(A',D)$ is a root separation of $G$ of order at most four (with $A'\cap D=(C\cap D)\cup\{v_1\}$), necessarily 4-light.
The vertex $v_1$ has at most $|C\cap D|\le 3$ neighbors in $C\cap D$, and thus $\rho_4(R_{C,D})\le \rho_4(R^G_{A',D})+3-4<0$ and the root separation $(C,D)$ is $4$-light in this case as well.
That is, $G'$ is indeed 4-light.

Moreover, $v_1$ has exactly $4-m$ neighbors in $X_{G'}=X_G\setminus\{x_1\}$,
and thus $\rho_4(G')=\rho_4(R_{A,B})+(4-m)-4=\rho_4(R_{A,B})-m\ge 1$.  Since $F_X$ is not an $\id$-rooted minor of $G'$, Corollary~\ref{cor-k4}
implies that $\rho_4(G')=\rho_4(R_{A,B})-m=1$ and $F_X$ is isomorphic to $K_4$.
\end{subproof}

\begin{claim}\label{cl-s3}
If $s=3$, then $E(M)\neq\{v_1v_2\}$.
\end{claim}
\begin{subproof}
Suppose for a contradiction that $E(M)=\{v_1v_2\}$.  Since $m=1$, we have $\rho_4(R_{A,B})\ge 2$, and $\widetilde{R}_{A,B}$ is $\SS_{5,3}$-universal
by Observation~\ref{obs-univsep}.  Consequently, we can contract $\widetilde{R}_{A,B}$ to its $\id$-rooted minor $F^+_X$,
then contract the paths $Q_1$ and $Q_2$ to obtain an $\id$-rooted minor of $G$ containing all edges between the vertices of $X_G$ except
for the edge $x_1x_2$.  Since $F$ is not an $\id$-rooted minor of $G$, we have $x_1x_2\in E(F)$.
Moreover, $F_X$ is a triangle and $\rho_4(R_{A,B})\le 3$, as otherwise we could have obtained a supergraph of $F$ as an $\id$-rooted minor of $G$ by contracting
$\widetilde{R}_{A,B}$ to $F^+_X+v_1v_2$ instead of just to $F^+_X$.

Similarly, the paths $Q_1$ and $Q_2$ are contained in different components of $G-(B\setminus \{v_1,v_2\})$.
Indeed, otherwise there would exist a path $P$ in $G-(B\setminus \{v_1,v_2\})$ between $Q_1$ and $Q_2$ intersecting these paths exactly in its ends,
and we could contract $\widetilde{R}_{A,B}$ to $F^+_X$, contract the paths $Q_1$ and $Q_2$, and contract the path $P$ to form the last edge $x_1x_2$ of $F$.
Therefore, we can divide $G-(B\setminus \{v_1,v_2\})$ into subgraphs $K_1$ and $K_2$ such that $Q_1\subseteq K_1$, $Q_2\subseteq K_2$,
and $G$ does not contain any edge between $K_1$ and $K_2$.

For $i\in \{1,2\}$, let $D_i=V(K_i)\cup \{y_1,y_2,y_3\}$ and let $C_i=(V(G)\setminus D_i)\cup \{x_i,v_i,y_1,y_2,y_3\}$; then $(C_i,D_i)$ is
a root separation of $G$ of order five.  Moreover, let $a_i=1$ if $x_iv_i\in E(G)$ and $a_i=0$ otherwise.
Suppose now that $\rho_4(R_{C_1,D_1})+a_1\ge 2$.  Since $\widetilde{R}_{C_1,D_1}$ is universal by Observation~\ref{obs-univsep},
it follows that we can contract $R_{C_1,D_1}$ to the $\id$-rooted minor with vertex set $C_1\cap D_1$ and edges $x_1v_1$ and $y_2y_3$,
contract $\widetilde{R}_{A,B}$ to its $\id$-rooted minor $(F^+_X-y_2y_3)+v_1v_2$, and contract the edge $x_1v_1$ and the path $Q_2$,
obtaining a supergraph of $F$ as an $\id$-rooted minor of $G$.  This is a contradiction, and thus $\rho_4(R_{C_1,D_1})+a_1\le 1$, and
similarly $\rho_4(R_{C_2,D_2})+a_2\le 1$.

Observe that
\begin{align*}
\rho_4(G)&=\rho_4(R_{C_1,D_1})+\rho_4(R_{C_2,D_2})+\rho_4(R_{A,B})+\deg^X v_1+\deg^X v_2-8\\
&=\rho_4(R_{C_1,D_1})+\rho_4(R_{C_2,D_2})+\rho_4(R_{A,B})+(3+a_1)+(3+a_2)-8\\
&=(\rho_4(R_{C_1,D_1})+a_1)+(\rho_4(R_{C_2,D_2})+a_2)+\rho_4(R_{A,B})-2\\
&\le 3.
\end{align*}
However, since $F_X$ is a triangle and $x_1x_2\in E(F)$, a subgraph of $F$ is isomorphic to $K_2+K_3$.
This is a contradiction, since $F$ belongs to the target of $G$.
\end{subproof}

Let $T\subseteq E(F_X)$ be a largest set such that $M+T$ is an $\id$-rooted minor of $\widetilde{R}_{A,B}$,
if possible chosen so that $F-T$ is not isomorphic to $K_2+K_3$, and let $\delta=|T|$.
Note that such a set exists, since $(A,B)$ is saturated, and thus $M$ is an $\id$-rooted minor of $\widetilde{R}_{A,B}$.
By first contracting $\widetilde{R}_{A,B}$ to $M+T$, then contracting the paths $Q_i$ for $i\in\{1,\ldots,5-s\}$,
we obtain a supergraph of $F-(E(F_X)\setminus T)$ as an $\id$-rooted minor of $G$.  Since $F$ is not an $\id$-rooted minor of $G$,
it follows that $T\neq E(F_X)$ and $\delta<|E(F_X)|$.  In combination with the preceding observations, this has the following consequences.
\begin{claim}\label{cl-narcliq}
$s\in\{3,4\}$, $\rho_4(R_{A,B})\le 3$, $m\le 2$, and $F_X$ is a clique.
\end{claim}
\begin{subproof}
Since $\rho_4(R_{A,B})\ge m+1$, Observations~\ref{obs-univsep} and \ref{obs-canaddm} imply that the 5-rooted graph $\widetilde{R}_{A,B}$ is $\SS_{5,m+1}$-universal,
and thus $\delta\ge 1$ and $|E(F_X)|\ge 2$.  Since $|V(F_X)|=s$, this in particular implies that $s\ge 3$, and thus $s\in\{3,4\}$.

Similarly, if $\rho_4(R_{A,B})=4$, then $\widetilde{R}_{A,B}$ is $\SS_{5,6}$-universal and $m\le 3$, and thus $\widetilde{R}_{A,B}$ is $\SS_{5,m+3}$-universal.
This implies that $\delta\ge 3$ and $|E(F_X)|\ge 4$, and since $|V(F_X)|=s$, we have $s=4$.  However, this contradicts Claim~\ref{cl-s4}, and thus $\rho_4(R_{A,B})\le 3$ and $m\le 2$.

Finally, suppose for a contradiction that $F_X$ is not a clique.  By Claim~\ref{cl-s4}, we have $s=3$.
We can by symmetry assume that all edges of $F_X$
are incident with $y_1$.  By Corollary~\ref{cor-nok4}, we can contract $\widetilde{R}_{A,B}$ to the $y_1$-star, and thus $M$ must contain an edge not incident with
$y_1$.  In particular $m\ge 1$, and thus $\rho_4(R_{A,B})\ge 2$ and $\widetilde{R}_{A,B}$ is $\SS_{5,3}$-universal.  Since $M\cup F_X$ is not an $\id$-rooted minor of $\widetilde{R}_{A,B}$ and
$m\le 2$, this implies that $m=2$, $|E(F_X)|=2$ and $\rho_4(R_{A,B})=3$.  Note that $M\cup F_X$ is not isomorphic to $K_2+K_3$,
since $F_X$ is an induced subgraph of $M\cup F_X$ with three vertices and two edges.  However $\rho_4(R_{A,B})=3$ implies that
$\widetilde{R}_{A,B}$ is $\SS_{5,4}^-$-universal and $M\cup F_X$ is an $\id$-rooted minor of $\widetilde{R}_{A,B}$, which is a contradiction.
\end{subproof}
Next, observe that by contracting $\widetilde{R}_{A,B}$ to $M+T$,
we obtain the graph $G_0+T$ as an $\id$-rooted minor of $G$.  Since $F$ is not an $\id$-rooted minor of $G$,
it follows that $F'=F-T$ is not an $\id$-rooted minor of $G_0$.  Recall that $G_0$ is 4-light and $n(G_0)<n(G)$,
and thus $G_0$ is universal by Observation~\ref{obs-lightminor}.  Consequently $F'$ does not belong to the target of $G_0$.

\begin{claim}
$\rho_4(R_{A,B})\ge m+2$, and consequently $m\in \{0,1\}$ and $\rho_4(R_{A,B})\in\{2,3\}$.
\end{claim}
\begin{subproof}
Suppose for a contradiction $\rho_4(R_{A,B})=m+1$, and thus $\rho_4(G_0)=\rho_4(G)-1$.  
We can by symmetry assume that $y_1$ is not the only vertex of $A\cap B\cap X_G$ adjacent to all vertices of $X_G\setminus B$ in $F$.
Let $T'=\{y_1y_2,y_1y_3\}$ and note that $M+T'$ is an $\id$-rooted minor of $\widetilde{R}_{A,B}$:
\begin{itemize}
\item If $m=0$, then this is the case since $\widetilde{R}_{A,B}$ is $K_{1,4}$-universal.
\item If $m=1$, then $\rho_4(R_{A,B})\ge 2$ and this is the case since $\widetilde{R}_{A,B}$ is $\SS_{5,3}$-universal.
\item If $m=2$, then $\rho_4(R_{A,B})=3$ and this is the case since $\widetilde{R}_{A,B}$ is $\SS_{5,4}^-$-universal
and the graph $M+T'$ cannot be isomorphic to $K_2+K_3$, since $\{y_1,y_2,y_3\}$ induces its subgraph with exactly two edges.
\end{itemize}
Moreover, note that $F-T'$ is not isomorphic to $K_2+K_3$, since $F_X$ is a clique and by the assumption on $y_1$ when $s=3$.
Therefore, we have $\delta\ge |T'|=2$ and if $\delta=2$, then $F'$ is not isomorphic to $K_2+K_3$.
However, this contradicts Observation~\ref{obs-trade} with $G_0$ playing the role of $G'$.
\end{subproof}
If $s=3$, then $|E(M\cup F_X)|=m+3\le \rho_4(R_{A,B})+1$.  Moreover, Claim~\ref{cl-s3} implies that $M\cup F_X$ is not isomorphic to $K_2+K_3$.
Since $\rho_4(R_{A,B})\in\{2,3\}$ and $\widetilde{R}_{A,B}$ is universal by Observation~\ref{obs-univsep}, it follows that $M\cup F_X$ is an $\id$-rooted minor of
$\widetilde{R}_{A,B}$, which is a contradiction.  Therefore, we have
$$s=4.$$
Recall that by Claim~\ref{cl-s4}, the 5-rooted graph $\widetilde{R}_{A,B}$ is not $(K_4+K_1)$-universal.
\begin{itemize}
\item If $\rho_4(R_{A,B})=2$, then by Corollary~\ref{cor-nok4}, the 5-rooted graph $\widetilde{R}_{A,B}$ is nearly $K_{2,3}^+$-universal, and we can by symmetry
assume that its exceptional root is $y_1$ or $v_1$; moreover, in this case we have $E(M)=\emptyset$.
\item If $\rho_4(R_{A,B})=3$, then by Corollary~\ref{cor-nok4}, the 5-rooted graph $\widetilde{R}_{A,B}$ is $K_{2,3}^+$-universal.
Moreover, we can by symmetry assume that $E(M)\subseteq \{v_1y_1\}$.
\end{itemize}
Let $T'=E(F_X)\setminus \{y_3y_4\}$.  In both cases, $M+T'$ is a subgraph of the graph obtained from the $\{y_1,y_2\}$-star
by deleting the edge $v_1y_2$, and thus $M+T'$ is an $\id$-rooted minor of $\widetilde{R}_{A,B}$.  Consequently $\delta\ge |T'|=5$.
By Observation~\ref{obs-trade2} applied with $G'=G_0$, we see that $\rho_4(G_0)\neq \rho_4(G)-2$.
Since $\rho_4(G_0)=\rho_4(G)-\rho_4(R_{A,B})+m$, it follows that $\rho_4(R_{A,B})=3$, $m=0$, and $\rho_4(G_0)=\rho_4(G)-3$.
By Observation~\ref{obs-trade3}, we conclude that $F'$ is isomorphic to $K_2+K_3$.
However, $F'-x_1$ has only one edge $y_3y_4$, which is a contradiction.
\end{proof}

Let us note the following simple consequence.
\begin{corollary}\label{cor-noplus1}
Let $(A,B)$ be a proper root separation of a minimal counterexample $G$.
If the order of $(A,B)$ is at most five, then $|A\cap B\setminus X_G|\ge 2$.
\end{corollary}
\begin{proof}
Suppose for a contradiction that $|A\cap B\setminus X_G|\le 1$.
By choosing such a root separation with $B$ minimal, we can assume that
the graph $H=G[B\setminus A]$ is connected and that each vertex of $A\cap B$ has a neighbor in $B\setminus A$.
By contracting $H$ to the vertex in $A\cap B\setminus X_G$ (if any), we obtain $K^G_{A\cap B}$ as an $\id$-rooted
minor of $R_{A,B}$.  However, this contradicts Lemma~\ref{lemma-nosatur}.
\end{proof}

We now aim to strengthen Lemma~\ref{lemma-nosatur} even more.  For a non-negative integer $k$, we say
that a root separation $(A,B)$ of a graph $G$ is \emph{$k$-nearly-saturated} if $K^G_{A\cap B}$ has a matching $M$ consisting of at most $k$ edges
such that $K^G_{A\cap B}-M$ is an $\id$-rooted minor of $R^G_{A,B}$.
Note that we do not require to be able to prescribe the matching $M$, and thus being $1$-nearly-saturated
is a weaker condition than being $K^-_\star$-universal even for root separations $(A,B)$ such that $K^G_{A\cap B}$ is a clique, i.e.,
such that $|B\cap X_G|\le 1$.  Moreover, being $0$-nearly-saturated is the same as being saturated.
Let us start with a technical observation.

\begin{observation}\label{obs-starsat}
Let $G$ be a rooted graph, let $(A,B)$ be a root separation of $G$,
and let $(C,D)$ be an isolator of $(A,B)$.
\begin{itemize}
\item If $(A,B)$ is $K_{1,\star}$-universal, then $(C,D)$ is also $K_{1,\star}$-universal.
\item If $|A\cap B|=|C\cap D|=5$ and $(A,B)$ is $K_{2,3}^+$-universal (or nearly $K_{2,3}^+$-universal), then $(C,D)$ is also $K_{2,3}^+$-universal (or nearly $K_{2,3}^+$-universal).
\item For a positive integer $k$, if $(A,B)$ is $k$-nearly-saturated, then $(C,D)$ is also $k$-nearly-saturated.
Moreover, if $(C,D)$ has order less than $|A\cap B|$, then $(C,D)$ is $(k-1)$-nearly-saturated.
\end{itemize}
In particular, if $G$ is a minimal counterexample and the root separation $(A,B)$ has order at most five and is $1$-nearly-saturated,
then $(A,B)$ is linked to the roots.
\end{observation}
\begin{proof}
Let $\PP$ be an isolator linkage for $(C,D)$ and $(A,B)$.  Let $T$ be the set of terminators of $\PP$, and for a vertex $v\in C\cap D$, let $f(v)\in T$
be the terminator of the path of $\PP$ with origin $v$.  Note that if a vertex $u$ is contained in $A\cap B\cap X_G$, then $u$
is also contained in $C\cap D\cap X_G$ and $f(u)=u$, and in particular $A\cap B\cap X_G\subseteq T$.

Suppose that $(A,B)$ is $K_{1,\star}$-universal, and let $v\in C\cap D$ be an arbitrary vertex.
Then $R_{A,B}$ contains the $f(v)$-star as an $\id$-rooted minor.  By further contracting the paths of $\PP$,
we obtain the $v$-star as an $\id$-rooted minor of $R_{C,D}$.
Hence, $(C,D)$ is $K_{1,\star}$-universal.  Analogously, if $|A\cap B|=|C\cap D|=5$ and $(A,B)$ is (nearly) $K_{2,3}^+$-universal, then $(C,D)$ is (nearly) $K_{2,3}^+$-universal.

Suppose next that $(A,B)$ is $k$-nearly-saturated, and let $M\subseteq E(K^G_{A\cap B})$ be a matching of size at most $k$
such that $H=K^G_{A\cap B}-M$ is an $\id$-rooted minor of $R_{A,B}$.  Let $H'$ be the graph obtained from $H$ by relabelling
each vertex $u\in T$ as $f^{-1}(u)$ (leaving the vertices in $V(H)\setminus T$ with the original label).
By first contracting $R_{A,B}$ to $H$ and then contracting the paths of $\PP$,
we see that $H'$ is an $\id$-rooted minor of $R_{C,D}$.  Since $f^{-1}$ maps vertices of $A\cap B\cap X_G$ to a subset of vertices of
$C\cap D\cap X_G$, the set $M'=E(K^G_{C\cap D})\setminus E(H')$ is a matching of size at most $k$.
Hence, the root separation $(C,D)$ is $k$-nearly-saturated.

Suppose now that $|C\cap D|<|A\cap B|$, and thus there exists a vertex $v\in V(H)\setminus T=V(H')\setminus (C\cap D)$; note that $v\not\in X_G$.
If $|M'|<k$, then $(C,D)$ is actually $(k-1)$-nearly-saturated; hence, suppose that $|M'|=k$.  This implies
that $v$ is not incident with any edge of $M$, and since $v\not\in X_G$ and $H+M=K^G_{A\cap B}$, the vertex $v$ is adjacent in $H$ to all other vertices.
Consequently $v$ is also adjacent in $H'$ to all other vertices.  Let $wz\in M'$ be an arbitrary edge.
By contracting the edge $vw$ of $H'$, we obtain $(H'-v)+wz\supseteq K^G_{C\cap D}-(M'\setminus\{wz\})$
as an $\id$-rooted minor of $R_{C,D}$, showing that $(C,D)$ is $(k-1)$-nearly-saturated.

If $G$ were a minimal counterexample, $k=1$, and $(A,B)$ had order at most five, then the saturated proper root separation $(C,D)$
would contradict Lemma~\ref{lemma-nosatur}.  Hence, in this case every isolator for $(A,B)$ has order $|A\cap B|$,
and thus $(A,B)$ is linked to the roots.
\end{proof}

The main challenge with excluding a $1$-nearly-saturated separation $(A,B)$ from a minimal counterexample $G$ is that after
contracting $R_{A,B}$ to $K^G_{A\cap B}-e$ for an edge $e=uv$, the resulting $\id$-rooted minor may have a root separation $(C,D)$
of order at most four separating $u$ from $v$, which does not directly correspond to a root separation of order at most four in $G$.
Thus, we need a separate argument to show that such root separations are 4-light.  Actually, we can as well do this argument
in the more general setting of $2$-nearly-saturated root separations.

\begin{lemma}\label{lemma-satulight}
Let $G$ be a minimal counterexample and let $(A,B)$ be a proper $2$-nearly-saturated
$K_{1,\star}$-universal root separation of $G$ of order five.  Let $M\subseteq E(K^G_{A\cap B})$ be a matching of size at most two
such that $H=K^G_{A\cap B}-M$ is an $\id$-rooted minor of $R_{A,B}$.
If the root separation $(A,B)$ is strongly linked to the roots, then the $\id$-rooted minor $L=L_{A,B}\cup H$ of $G$ is 4-light.
\end{lemma}
\begin{proof}
Consider any root separation $(C,D)$ of $L$ of order at most four.
Since $(A,B)$ is linked to the roots, we cannot have $A\cap B\subseteq D$, as in that case $(C,D\cup B)$ would be a root separation of $G$ of order at most four.

Suppose now that $A\cap B\not\subseteq C$ and $A\cap B\not\subseteq D$.  Let $u\in A\cap B\setminus D$
and $v\in A\cap B\setminus C$ be arbitrary vertices.  Then $uv\not\in E(L)$, and consequently $uv\not\in E(H)$.
Since $A\cap B\cap X_G\subset C$, we have $v\not\in X_G$, and thus $uv$ is an edge of $M$.
Since this holds for any pair of vertices in $A\cap B\setminus D$ and in $A\cap B\setminus C$
and since $M$ is a matching, it follows that $A\cap B\setminus D=\{u\}$, $A\cap B\setminus C=\{v\}$,
and $A\cap B\setminus\{u,v\}\subseteq C\cap D$.  Let $D'=D\cup \{u\}$, and observe that
$(C,D')$ is a root separation of $L$ of order $|C\cap D|+1\le 5$. Since $A\cap B\subseteq D'$, it follows that $(C,D'\cup B)$ is a root separation of $G$ of order at most five.
Since $(A,B)$ is strongly linked to the roots, the root separation $(C,D'\cup B)$ is equal to either $(A,B)$ or $(X_G,V(G))$.  The former is not possible, since $v\in A\setminus C$.
It follows that $C=X_G$, and thus $A\cap B\setminus\{v\}\subset X_G$; however, this contradicts Corollary~\ref{cor-noplus1}.

Therefore, for every root separation $(C,D)$ of $L$ of order at most four, we have $A\cap B\subseteq C$.
Consequently, $(C\cup B,D)$ is a root separation of $G$ with the strictly right-hand side equal to $\widetilde{R}^L_{C,D}$.
Since $G$ is 4-light, it follows that $L$ is $4$-light.
\end{proof}

With this, it is easy to mostly exclude $2$-nearly-saturated root separations of order five from minimal counterexamples.

\begin{lemma}\label{lemma-no2nearly}
Let $G$ be a minimal counterexample and let $(A,B)$ be a proper
$K_{1,\star}$-universal root separation of $G$ of order five.
If the root separation $(A,B)$ is strongly linked to the roots, then it is not $2$-nearly-saturated.
\end{lemma}
\begin{proof}
Suppose for a contradiction that the root separation $(A,B)$ is $2$-nearly-saturated.
Let $k\in\{0,1,2\}$ be minimum such that $(A,B)$ is $k$-nearly-saturated,
and let $M\subseteq E(K^G_{A\cap B})$ be a matching of size $k$ such that $H=K^G_{A\cap B}-M$ is an $\id$-rooted minor of $R_{A,B}$.
By Lemma~\ref{lemma-nosatur}, we have $k\ge 1$; let $e_0$ be any edge of $M$.  Let $L=L_{A,B}\cup H$.  By Lemma~\ref{lemma-satulight}, the $\id$-rooted minor $L$ of $G$ is 4-light.
Since $n(L)<n(G)$, Observation~\ref{obs-lightminor}
implies that $\rho_4(L)<\rho_4(G)$.
On the other hand,
$$\rho_4(L)=\rho_4(G)-\rho_4(R_{A,B})+|E(H)\setminus E(G)|,$$
and thus
$$\rho_4(R_{A,B})\ge |E(H)\setminus E(G)|+1=|E(H+e_0)\setminus E(G)|.$$
Let $t=\rho_4(R_{A,B})$ and note that $t\le 4$ by Lemma~\ref{lemma-ubdel}.
By Observations~\ref{obs-univsep} and \ref{obs-canaddm}, the 5-rooted graph $\widetilde{R}_{A,B}$ is $\SS_{5,t}$-universal.
This implies that the graph $(H+e_0)-E(G[A\cap B])$ is an $\id$-rooted minor of $\widetilde{R}_{A,B}$,
and thus $H+e_0=K^G_{A\cap B}-(M\setminus\{e_0\})$ is an $\id$-rooted minor of $R_{A,B}$.
It follows that the root separation $(A,B)$ is $(k-1)$-nearly-saturated, contradicting the minimality of $k$.
\end{proof}

We can simplify the assumptions for $1$-nearly-saturated separations.

\begin{corollary}\label{cor-noK5minus}
Let $G$ be a minimal counterexample and let $(A,B)$ be a proper $K_{1,\star}$-universal root separation of $G$ of order five.
Then $(A,B)$ is not $1$-nearly-saturated.
\end{corollary}
\begin{proof}
Suppose for a contradiction that there exists such a proper root separation $(A,B)$ which is $1$-nearly-saturated,
and let us choose one with $A$ minimal.  Observation~\ref{obs-starsat} implies that $(A,B)$ is linked to the roots,
and thus all isolators of $(A,B)$ have order five.
Moreover, Observation~\ref{obs-starsat} implies that isolators of $(A,B)$ are $K_{1,\star}$-universal and $1$-nearly-saturated.
By the minimality of $A$, the root separation $(A,B)$ does not have any proper isolators other than itself,
and thus it is strongly linked to the roots.  This contradicts Lemma~\ref{lemma-no2nearly}.
\end{proof}

We can also simplify the assumptions of Lemma~\ref{lemma-no2nearly} for $2$-nearly-saturated separations,
but only to a lesser extent: The issue is that Corollary~\ref{cor-noK5minus} does not exclude the existence
of $1$-nearly-saturated proper root separations of order at most four, and thus Observation~\ref{obs-starsat}
does not force $2$-nearly-saturated root separations of order five to be linked to the roots.
However, it turns out that the following additional assumption (forbidding ``degenerate'' root separations where some vertices
of the cut set can be removed from the left-hand side) suffices.
Let $(A, B)$ be a root separation of a rooted graph $G$.  We say that a vertex $u\in A\cap B$ is \emph{$(A,B)$-exposed}
if $G$ contains a path $P$ from $X_G$ to $u$ such that $V(P)\cap B=\{u\}$, or equivalently, that the component of $G-(A\cap B\setminus \{u\})$
containing $u$ intersects $X_G$.
We say that the root separation $(A,B)$ is \emph{exposed} if all vertices of $A\cap B$ are $(A,B)$-exposed.
In particular, every root separation linked to the roots is exposed.  We will need the following well-known observation (we include a proof for completeness).
\begin{observation}\label{obs-exposed}
Let $(A,B)$ be a root separation of a rooted graph $G$.  For every $(A,B)$-exposed vertex $v\in A\cap B$,
there exists a root linkage for $(A,B)$ containing a path ending in $v$.
\end{observation}
\begin{proof}
Let $k$ be the root connectivity of $(A,B)$ and let $\PP_0$ be an arbitrary root linkage for $(A,B)$ (of size $k$).
Since $v$ is $(A,B)$-exposed, there exists a path $P$ in $G$ from $X_G$ to $v$ such that $V(P)\cap B=\{v\}$.

Let $\vec{F}$ be the directed graph obtained as follows.  We start with $G[A]$ and replace each of its edges
by a pair of oppositely directed edges.  Next, we delete all edges directed away from vertices of $A\cap B$
(and in particular, all edges between vertices of $A\cap B$).  Finally, we add a set $T$ of $k-1$
new vertices and for each $x\in A\cap B$ and $y\in T$, we add a directed edge from $x$ to $y$.

Consider any set $X$ of $k-1$ vertices of $\vec{F}$.  Since $|\PP_0|=k$, there exists a path $Q\in \PP_0$
disjoint from $X$.  If $X\neq T$, then we can add a vertex of $T\setminus X$ to $Q$ and obtain a directed
path in $\vec{F}-X$ from $X_G$ to $T\cup\{v\}$.  If $X=T$, then $P$ corresponds to a directed path
in $\vec{F}-X$ from $X_G$ to $T\cup\{v\}$.

By Menger's theorem, this implies that $\vec{F}$ contains a system $\PP$ of $k$ pairwise vertex-disjoint directed paths from $X_G$ to $T\cup\{v\}$,
without loss of generality intersecting $X_G$ only in their first vertices.
Since $\vec{F}-T$ does not contain any edges directed away from vertices of $A\cap B$, each path of $\PP$ intersects $A\cap B$ in
exactly one vertex (the last one for the path ending in $v$ and the next-to-last one for the rest of the paths).
Hence, the subpaths of $\PP$ between $X_G$ and $A\cap B$ form a root linkage for $(A,B)$, necessarily containing a path ending in $v$.
\end{proof}

We can now show that adding the assumption of being exposed is enough to solve the aforementioned root linkedness issue.
\begin{lemma}\label{lemma-reachtolink}
Let $G$ be a minimal counterexample and let $(A,B)$ be a $2$-nearly-saturated root separation of $G$ of order five.  If the root separation $(A,B)$ is exposed,
then it is linked to the roots.
\end{lemma}
\begin{proof}
We can assume that $(A,B)$ is not $1$-nearly-saturated, as otherwise the conclusion holds by Observation~\ref{obs-starsat}.
Let $M\subseteq E(K^G_{A\cap B})$ be a matching of size two such that $H=K^G_{A\cap B}-M$ is an $\id$-rooted minor of $R_{A,B}$.
Let $V(H)=A\cap B=\{v,w,x,y,z\}$, where $M=\{wx,yz\}$.  Suppose for a contradiction that an isolator $(C,D)$ of $(A,B)$ has order at most four.
Corollary~\ref{cor-3conn} implies that $|C\cap D|=4$.

Since the vertex $v$ is $(A,B)$-exposed, Observation~\ref{obs-exposed} implies that there exists an isolator linkage $\PP$ for
$(C,D)$ and $(A,B)$
such that $v$ is a terminator of $\PP$.  We can assume that $z$ is the unique non-terminator of $\PP$ in $A\cap B$; then $z\not\in C$, and thus $z\not\in X_G$.
For each vertex $s\in\{v,w,x,y\}$, let $f(s)$ denote the origin of the path of $\PP$ with terminator $s$.  Let us also define $f(z)=z$.

Let $H'$ be the graph obtained from $H$ by relabelling each vertex $s\in V(H)$ to $f(s)$.
By first contracting $R_{A,B}$ to $H$, then contracting the paths of $\PP$,
we see that $H'$ is an $\id$-rooted minor of $R_{C,D}$.  Since $X_G\cap A\cap B\subseteq X_G\cap C\cap D$,
the vertex $v$ is not incident with any edge of $M$, and $z\not\in C\cap D$,
observe that $E(K^G_{C\cap D})\setminus E(H')\subseteq \{f(w)f(x)\}$.  Since $z\not\in X_G$ and $yz$ is the only edge of $M$ incident with $z$, we have $zw,zx\in E(H)$,
and thus $zf(w),zf(x)\in E(H')$.  Hence, we can further contract the edge $zf(x)$ in $H'$ to obtain a supergraph of $K^G_{C\cap D}$ as an $\id$-rooted minor of $R_{C,D}$.
This contradicts Lemma~\ref{lemma-nosatur}.
\end{proof}

Hence, with the assumption of being exposed, we get the desired result on $2$-nearly-saturated separations.

\begin{corollary}\label{cor-noK5minus2}
Let $G$ be a minimal counterexample and let $(A,B)$ be a proper $K_{1,\star}$-universal root separation of $G$ of order five.
If the root separation $(A,B)$ is exposed, then it is not $2$-nearly-saturated.
\end{corollary}
\begin{proof}
Suppose for a contradiction that there exists such a proper exposed $K_{1,\star}$-universal root separation $(A,B)$ of order five which is $2$-nearly-saturated,
and let us choose one with $A$ minimal.  By Lemma~\ref{lemma-reachtolink}, the root separation $(A,B)$ is linked to the roots,
and thus every isolator for $(A,B)$ has order five.
Moreover, Observation~\ref{obs-starsat} implies that each such isolator is $K_{1,\star}$-universal and $2$-nearly-saturated. By the minimality of $A$,
it follows that the root separation $(A,B)$ does not have any proper isolators other than itself, and thus it is strongly linked to the roots.
This contradicts Lemma~\ref{lemma-no2nearly}.
\end{proof}

For a rooted graph $G$ and a set $S\subseteq V(G)$, let $k^G(S)=|E(K^G_S)\setminus E(G)|$.
Corollary~\ref{cor-noK5minus2} can be used to obtain the following bounds.

\begin{corollary}\label{cor-neic2}
Let $G$ be a minimal counterexample and let $(A,B)$ be a proper exposed root separation of $G$ of order five.
If $(A,B)$ is $K_{1,\star}$-universal (and in particular if $\rho_4(R_{A,B})>0$), then $|A\cap B\cap X_G|\le 3$,
$k^G(A\cap B)\ge 4$, and if $k^G(A\cap B)=4$ then $K^G_{A\cap B}-E(G[A\cap B])$ is isomorphic to $C_4+K_1$.
Moreover,
\begin{itemize}
\item If $\rho_4(R_{A,B})=2$, then $k^G(A\cap B)\ge 6$.
\item If $\rho_4(R_{A,B})=3$, then $k^G(A\cap B)\ge 7$.
\item If $\rho_4(R_{A,B})=4$, then $k^G(A\cap B)\ge 9$ and $|A\cap B\cap X_G|\le 2$.
\end{itemize}
\end{corollary}
\begin{proof}
Recall that if $\rho_4(R_{A,B})>0$, then $R_{A,B}$ is $K_{1,\star}$-universal by Corollary~\ref{cor-nok4}.
Corollary~\ref{cor-noplus1} implies that $|A\cap B\cap X_G|\le 3$.
Let $H=K^G_{A\cap B}-E(G)$ and let $k=k^G(A\cap B)=|E(H)|$.
By Corollary~\ref{cor-noK5minus2} the root separation $(A,B)$ is not $2$-nearly-saturated,
and thus for every $\id$-rooted minor $H'$ of $R_{A,B}$ with vertex set $A\cap B$, the graph $H-E(H')$ has a vertex of degree at least two.

Let $v$ be a vertex of $H$ of the largest degree.  Since the $v$-star is an $\id$-rooted minor of $R_{A,B}$,
the graph $H-v$ has a vertex of degree at least two.  Hence, $k=\Delta(H)+|E(H-v)|\ge 4$.
Moreover, if $k=4$, then $\Delta(H)=2$ and $H-v$ has exactly two edges, both incident with the same vertex.
This implies that either $H$ is isomorphic to $C_4+K_1$ or $H$ is a path $uvz_1z_2z_3$;
and the latter is not possible, since in that case $\Delta(H-z_1)=1$ and $z_1$-star is an $\id$-rooted minor of $R_{A,B}$.

Let us now strengthen this bound for root separations $(A,B)$ such that $\rho_4(R_{A,B})\ge 2$.  To this end, recall that the 5-rooted graph $R_{A,B}$ is universal by Observation~\ref{obs-univsep}.
Moreover, since $k\ge 4$ and not all edges of $H$ are incident with the same vertex, there exists a matching $M\subseteq E(H)$ of size two.
\begin{itemize}
\item If $\rho_4(R_{A,B})=2$, then $R_{A,B}$ is $\SS_{5,3}$-universal.  Since $H-M$ is not
an $\id$-rooted minor of $R_{A,B}$, it follows that it has at least four edges, and thus
$k=|E(H)|\ge 6$.
\item Suppose now that $\rho_4(R_{A,B})=3$, and thus $R_{A,B}$ is $\SS_{5,4}^-$-universal.
Consequently $H-M\not\in \SS_{5,4}^-$, and thus either $k\ge 7$, or $k=6$ and $H-M$ is isomorphic to $K_2+K_3$.
In the latter case, $H$ consists of a $4$-cycle $v_1v_2v_3v_4$ and a vertex $v_5$ adjacent to $v_3$ and $v_4$,
and $M=\{v_1v_4,v_2v_3\}$.  However, then $H'=H-\{v_1v_2,v_3v_4\}\in \SS_{5,4}^-$ is an $\id$-rooted minor of $R_{A,B}$
and $\Delta(H-E(H'))=1$, which is a contradiction.  Therefore, we have $k\ge 7$.
\item If $\rho_4(R_{A,B})=4$, then $R_{A,B}$ is $\SS_{5,6}$-universal.  Consequently $H-M$ has at least $7$ edges,
and thus $k\ge 9$.
\end{itemize}
\end{proof}

Moreover, we can improve the bounds for nearly $K_{2,3}^+$-universal root separations.
\begin{observation}\label{obs-nearbc}
Let $G$ be a minimal counterexample and let $(A,B)$ be a proper exposed root separation of $G$ of order five.
\begin{itemize}
\item If $(A,B)$ is nearly $K_{2,3}^+$-universal, then $|A\cap B\cap X_G|\le 2$ and $k^G(A\cap B)\ge 7$.
\item If $(A,B)$ is $K_{2,3}^+$-universal, then $G[A\cap B]$ together with the clique on $A\cap B\cap X_G$ forms a matching and $k^G(A\cap B)\ge 8$.
\end{itemize}
\end{observation}
\begin{proof}
Let $K$ be the clique on $A\cap B\cap X_G$, let $H=G[A\cap B]\cup K$, and note that since $E(K)$ contains exactly the non-edges of $K^G_{A\cap B}$,
we have $k^G(A\cap B)=10-|E(H)|$.  If $(A,B)$ is nearly $K_{2,3}^+$-universal but not $K_{2,3}^+$-universal, then let $y\in A\cap B$ be the exceptional root of $R_{A,B}$ and
let $T=A\cap B\setminus\{y\}$; otherwise, let $T=A\cap B$.

Let us first consider the case that there exists a vertex $u\in A\cap B$ incident with at least two distinct edges $uv$ and $uw$ in $H$ such that
$v,w\in T$.  Let $A\cap B\setminus\{u,v,w\}=\{x_1,x_2\}$, let $S$ be the
graph obtained from the $\{x_1,x_2\}$-star by deleting the edge $ux_2$, and let $S'=S\cup G[A\cap B]$.  Let $M=\{ux_2,vw\}$.
Since $uv,uw\in E(H)$, the graph $H\cup S$ contains all edges between the vertices of $A\cap B$ except possibly for those in $M$,
and thus $E(K^G_{A\cap B})\setminus E(S')\subseteq M$.
Since $(A,B)$ is nearly $K_{2,3}^+$-universal and the roots $v$ and $w$ of $R_{A,B}$ are not exceptional, the graph $S$ is an $\id$-rooted minor of $R_{A,B}$,
and thus $S'$ is also an $\id$-rooted minor of $R_{A,B}$.  This implies that the root separation $(A,B)$ is $2$-nearly-saturated,
contradicting Corollary~\ref{cor-noK5minus2}.

Hence, each vertex of $H$ has at most one neighbor in $T$.  If $(A,B)$ is $K_{2,3}^+$-universal,
then this implies that $H$ is a matching and $k^G(A\cap B)=10-|E(H)|\ge 8$.

If $(A,B)$ is only nearly $K_{2,3}^+$-universal, then this implies that $\deg_H y\le 1$,
if $H$ contains an edge $uy$, then $\deg_H u\le 2$, and each vertex $v\in T$ non-adjacent to $y$
satisfies $\deg_H v\le 1$.  Thus, all components of $H$ have size at most two, except possibly the component containing $y$ which can be a 3-vertex path.
It follows that $|A\cap B\cap X_G|=|V(K)|\le 2$, $|E(H)|\le 3$, and $k^G(A\cap B)\ge 7$.
\end{proof}

\section{Novas}

Next, we aim to mostly exclude $(K_4+K_1)$-universal separations from a minimal counterexample.
Let us start with a technical lemma which we use in the proof to argue that a certain root separation is exposed,
and thus it satisfies the bounds from Corollary~\ref{cor-neic2}.
We say that a root separation $(A,B)$ of a rooted graph $G$ is \emph{right-linked}
if for all vertices $u\in A\cap B\setminus X_G$ and $v\in A\cap B$, the graph $R_{A,B}-(A\cap B\setminus\{u,v\})$ contains a path from $u$ to $v$.
Note that this in particular is the case whenever $R_{A,B}$ contains an $\id$-rooted minor with vertex set $A\cap B$ in which $u$ and $v$ are adjacent.
\begin{lemma}\label{lemma-showexposed}
Let $(A,B)$ be a root separation of a rooted graph $G$, let $H$ be an $\id$-rooted minor of $R_{A,B}$ with vertex set $A\cap B$,
let $L=L_{A,B}\cup H$ be the $\id$-rooted minor of $G$ obtained by contracting $R_{A,B}$ to $H$, and let $(C,D)$ be a proper exposed root separation of $L$ of order at most $|A\cap B|$ such that $C\cap D\neq X_G$.
If $(A,B)$ is right-linked and strongly linked to the roots, then the root separation $(C\cup B,D)$ of $G$ is exposed.
\end{lemma}
\begin{proof}
For distinct vertices $u,v\in A\cap B$ such that at least one of them is not a root, let $P_{uv}$ be a path from $u$ to $v$ in $R_{A,B}-(A\cap B\setminus\{u,v\})$,
which exists since $(A,B)$ is right-linked.
For each $uv\in E(G)$, let $P_{uv}$ be the path $uv$.   Note that these definitions are consistent in the overlapping case that $uv\in E(G[A\cap B])$.
Let us consider a vertex $z\in (C\cup B)\cap D$, and let us distinguish two cases.
\begin{itemize}
\item Suppose first that $z\in C$.  Since the root separation $(C,D)$ of $L$ is exposed, there exists a path $Q$
in $L[C]$ from the roots to $z$ disjoint from $(B\cup C)\cap D\setminus\{z\}$.
This path could contain edges of $L$ between vertices of $A\cap B$; however, the concatenation of the paths $P_{uv}$ for $uv\in E(Q)$ is a walk from the roots to $z$
in $G$ disjoint from $(B\cup C)\cap D\setminus\{z\}$.
\item Next, suppose that $z\in B\setminus C$.  Let $M$ be the set of vertices of all components of $G-(C\cap D)$ intersecting $X_G$.
If $M$ did not contain any vertex of $B$, then $(M',N')=(M\cup (C\cap D),V(G)\setminus M)$ would be a root separation of $G$ of order $|C\cap D|\le |A\cap B|$
such that $M'\subseteq C\subsetneq A$ and $B\subseteq N'$.  Moreover, $(M',N')\neq (X_G,V(G))$, since $C\cap D\neq X_G$. This would contradict the assumption that the root separation $(A,B)$ is strongly linked to the
roots.  Hence, $G-(C\cap D)$ contains a path $Q'$ from $X_G$ to $B$ such that the path $Q'$ intersects $B$ only in its last vertex $u$.  Then the concatenation of $Q'$
with $P_{uz}$ is a path from the roots to $z$ disjoint from $(B\cup C)\cap D\setminus\{z\}$.
\end{itemize}
In both cases, the vertex $z$ is $(B\cup C,D)$-exposed in $G$.  It follows that the root separation $(B\cup C,D)$ of $G$ is exposed.
\end{proof}

It turns out to be convenient to work with a notion which is less restrictive than $(K_4+K_1)$-universality.
For a rooted graph $G$, a set $S\subseteq V(G)$, and a vertex $v\in S\setminus X_G$, a
\emph{$(G,S,v)$-nova} is a graph obtained from $K^G_S$ by deleting all but at most one edge incident with $v$.
A graph $H$ is a \emph{$(G,S)$-nova} if there exists a vertex $v\in S\setminus X_G$ such that $H$ is a $(G,S,v)$-nova.
We say that a root separation $(A,B)$ of a rooted graph $G$ is \emph{nova-universal} if for every $v\in A\cap B\setminus X_G$,
the rooted graph $R_{A,B}$ contains a $(G,A\cap B,v)$-nova as an $\id$-rooted minor (since an edge between $v$ and $A\cap B\setminus \{v\}$
in a $(G,A\cap B,v)$-nova is optional, observe that this is equivalent to saying that the rooted graph $R_{A,B}-v$ contains $K^G_{A\cap B}-v$ as an $\id$-rooted minor).
Clearly, every $(K_4+K_1)$-universal separation of order five is nova-universal.
A $(G,A\cap B)$-nova $H$ is \emph{$(A,B)$-profitable} if $|E(H)\setminus E(G)|>\rho_4(R_{A,B})$.
As a first step towards excluding nova-universal separations, let us show that they give rise to profitable novas.
\begin{lemma}\label{lemma-profitable-nova}
Let $G$ be a minimal counterexample, and let $(A_0,B_0)$ be a proper $K_{1,\star}$-universal root separation in $G$ of order five.
If $(A_0,B_0)$ is nova-universal, then $G$ also has a proper $K_{1,\star}$-universal nova-universal root separation $(A,B)$ of order five
such that $(A,B)$ is strongly linked to the roots and $R_{A,B}$ contains an $(A,B)$-profitable $(G,A\cap B)$-nova as an $\id$-rooted minor.
\end{lemma}
\begin{proof}
Let $(A,B)$ be an isolator for $(A_0,B_0)$ such that $A\neq X_G$ and subject to that $A$ is minimal, let $\PP$ be an
isolator linkage for $(A,B)$ and $(A_0,B_0)$, and let $T$ be the set of terminators of this linkage.

If $|A\cap B|<5$, then choose a vertex $v\in A_0\cap B_0\setminus T$ arbitrarily.
Since $A_0\cap B_0\cap X_G\subseteq A\cap B\cap X_G$, we have $v\not\in X_G$.  Since $(A_0,B_0)$ is nova-universal,
$R_{A_0,B_0}-v$ contains $K^G_{A_0\cap B_0}-v$ as an $\id$-rooted minor.  By further contracting the paths of $\PP$,
we obtain a supergraph of $K^G_{A\cap B}$ as an $\id$-rooted minor of $R_{A,B}$.  Hence,
the root separation $(A,B)$ is saturated, contradicting Lemma~\ref{lemma-nosatur}.
Therefore, we have $|A\cap B|=5$.

By the minimality of $A$, the root separation $(A,B)$ cannot have an isolator other than itself and $(X_G,V(G))$,
and thus $(A,B)$ is strongly linked to the roots, and in particular exposed.
Moreover, the root separation $(A,B)$ is $K_{1,\star}$-universal by Observation~\ref{obs-starsat}, and nova-universal by an argument analogous to the proof
of Observation~\ref{obs-starsat}.

Let $H=K^G_{A\cap B}-E(G[A\cap B])$, so that $|E(H)|=k^G(A\cap B)$.  Let $s$ be the number of edges of $H$ with exactly one end in $X_G$,
and let $m=|E(H)|-s$ be the number of edges of $H$ not incident with the vertices of $X_G$.
Let $t=|A\cap B\setminus X_G|$; by Corollary~\ref{cor-noplus1}, we have $2\le t\le 5$.
For a vertex $u\in A\cap B\setminus X_G$, let $a_u=|E(H-u)|$.  Note that
$$\sum_{u\in A\cap B\setminus X_G} a_u=(t-1)s+(t-2)m=(t-1)k^G(A\cap B)-m\ge (t-1)k^G(A\cap B)-\binom{t}{2},$$
and thus there exists a vertex $v\in A\cap B\setminus X_G$ such that
$$a_v\ge \frac{t-1}{t}\cdot k^G(A\cap B) - \frac{t-1}{2}.$$
Relevant values of the expression $\tfrac{t-1}{t}\cdot k - \frac{t-1}{2}$ can be seen in the following table.
\begin{center}
\begin{tabular}{c|cccc}
$k$	&$t=2$	&$t=3$	&$t=4$	&$t=5$\\
\hline
9	&	&5	&5.25	&5.2\\
7	&3	&$3.\overline{6}$	&3.75	&3.6\\
6	&2.5	&3	&3	&2.8\\
4	&1.5	&$1.\overline{6}$	&1.5	&1.2
\end{tabular}
\end{center}
By Corollary~\ref{cor-neic2}, we obtain the following bounds.
\begin{itemize}
\item If $\rho_4(R_{A,B})=4$, then $k^G(A\cap B)\ge 9$ and $t\ge 3$, and thus $a_v\ge 5>\rho_4(R_{A,B})$.
\item If $\rho_4(R_{A,B})=3$, then $k^G(A\cap B)\ge 7$, and thus $a_v\ge 4>\rho_4(R_{A,B})$ unless $t=2$ and $k^G(A\cap B)=7$ (or equivalently, $E(K^G_{A\cap B})\cap E(G)=\emptyset$),
in which case $a_v=3=\rho_4(R_{A,B})$.
\item If $\rho_4(R_{A,B})=2$, then $k^G(A\cap B)\ge 6$, and thus $a_v\ge 3>\rho_4(R_{A,B})$.
\item If $\rho_4(R_{A,B})\le 1$, then $k^G(A\cap B)\ge 4$, and thus $a_v\ge 2>\rho_4(R_{A,B})$.
\end{itemize}
Since $(A,B)$ is nova-universal, $R_{A,B}$ contains the $(G,A\cap B,v)$-nova $Z$ with the edge set $E(K^G_{A\cap B}-v)$ as an $\id$-rooted minor.
Note that $|E(Z)\setminus E(G)|=|E(H-v)|=a_v$.  Hence, if $a_v>\rho_4(R_{A,B})$, then the $(G,A\cap B)$-nova $Z$ is $(A,B)$-profitable.
By the preceding analysis, this is the case unless $\rho_4(R_{A,B})=3$, $t=2$, and $E(K^G_{A\cap B})\cap E(G)=\emptyset$.
In this case, let $v'$ be the unique vertex in $A\cap B\setminus (X_G\cup \{v\})$.  Then the $(G,A\cap B,v)$-nova $Z+vv'$
is $(A,B)$-profitable, and $Z+vv'$ is an $\id$-rooted minor of $R_{A,B}$ since $Z+vv'$ is equal to the $v'$-star and $(A,B)$ is $K_{1,\star}$-universal.
\end{proof}

However, the separations postulated in the previous lemma actually cannot appear in a minimal counterexample.

\begin{lemma}\label{lemma-noK4plus}
Let $G$ be a minimal counterexample and let $(A,B)$ be a proper $K_{1,\star}$-universal root separation of $G$ of order five strongly linked to the roots.
Let $H$ be a $(G,A\cap B,v)$-nova for a vertex $v\in A\cap B\setminus X_G$.
If $H$ is an $\id$-rooted minor of $R_{A,B}$, then it is not $(A,B)$-profitable.
\end{lemma}
\begin{proof}
Suppose for a contradiction the $(G,A\cap B,v)$-nova $H$ is $(A,B)$-profitable and appears as an $\id$-rooted minor of $R_{A,B}$.
Without loss of generality, we can assume that if $v$ is incident with an edge $e_0$ in $H$, then $e_0\not\in E(G)$,
as otherwise we could consider the $(A,B)$-profitable $(G,A\cap B,v)$-nova $H-e_0$, instead.
Let $\beta=\deg_H v$.  Let $L=L_{A,B}\cup H$ and note that $L$ is an $\id$-rooted minor of $G$.
Moreover, since $H$ is $(A,B)$-profitable, we have
\begin{equation}\label{eq-novaprofit}
\rho_4(L)\ge \rho_4(G)+1.
\end{equation}
Let us now give a quick outline of the rest of the argument.  The 5-rooted graph $L$ does not necessarily
contradict Observation~\ref{obs-lightminor}, since it may not be 4-light.  However, in Claim~\ref{cl-separof},
we show that all non-4-light root separations of $L$ of order at most four separate $v$ from the roots.
By Corollary~\ref{cor-k4}, we then see that such a root separation with minimal right-hand side either is $K_\star$-universal,
or has order four and it is $K^-_\star$-universal (this does not contradict Lemma~\ref{lemma-noclicut}, since
these root separations do not directly correspond to root separations of $G$ of the same order).
We next contract a $K_\star$-universal root separation $(C_1,D_1)$ of $L$ with maximal right-hand side and minimal left-hand side to a clique,
forming an $\id$-rooted minor $L_1$ of $L$ with no proper $K_\star$-universal root separations (if $L$ does not have such
a root separation, we let $L_1=L$).  This does not decrease the $4$-density, as we argue in Claim~\ref{cl-L1heavy}.
Moreover, Theorem~\ref{thm-k4} together with Lemma~\ref{lemma-diamonds} now guarantees that $L_1$ can only be non-4-light because of $K^-_\star$-universal
root separations of order four and 4-density one.  We finish the argument by performing the $2$-twin reduction on such a root separation $(C_3,D_3)$
with minimal left-hand side, obtaining an $\id$-rooted minor $L_3$ of $G$ with the same census.  This decreases the 4-density by at most one (which still
implies that $\rho_4(L_3)\ge \rho_4(G)$ by (\ref{eq-novaprofit})), and using Lemma~\ref{lemma-2twin}, we can easily argue that $L_3$ is 4-light.
This implies that $L_3$ is a smaller counterexample than $G$, which is a contradiction.

Let us now perform the steps carefully.
For a root separation $(M,N)$ of $L$ of order at most four, let $v(M,N)$ be the number of neighbors
of $v$ in $M\cap N$ in $G$.  The 5-rooted graph $L$ is not necessarily 4-light,
but its non-4-light root separations of order at most four are quite constrained.

\begin{claim}\label{cl-separof}
Every root separation $(M,N)$ of $L$ of order at most four satisfies $A\cap B\setminus \{v\}\subseteq M$.
Let $M'=M\cup\{v\}\cup B$, so that $(M',N)$ is a root separation of $G$.
If $(M,N)$ is not $4$-light, then $v\in N\setminus M$, $|M'\cap N|=|M\cap N|+1$, and
$$\rho_4(R^L_{M,N})=\rho_4(R^G_{M',N})+v(M,N)+\beta-4\le \rho_4(R^G_{M',N}).$$
\end{claim}
\begin{subproof}
Suppose first for a contradiction that
there exists a vertex $u\in A\cap B\setminus \{v\}$ belonging to $N\setminus M$.
Since $X_G\subseteq M$, we have $u\not\in X_G$, and thus $u$ is adjacent in the $(G,A\cap B,v)$-nova $H$ to all vertices of $A\cap B\setminus\{u,v\}$.
Consequently all vertices of $A\cap B\setminus\{u,v\}$ are also adjacent to $u$ in $L$, and thus $A\cap B\setminus\{u,v\}\subset N$.
Let $N'=N\cup\{v\}\cup B$ and note that $(M,N')$ is a root separation of $G$.  Since $M\subseteq A$, $B\subseteq N'$,
and $(A,B)$ is linked to the roots, we have $|M\cap N'|\ge 5$, and thus $v\in M\setminus N$, $|M\cap N|=4$, and $|M\cap N'|=5$.
Note that $M\neq A$, since $u$ is a vertex of $A\setminus M$, and that $M\neq X_G$, since $v$ is a vertex of $M\setminus X_G$.
It follows that $(M,N')$ is an isolator of $(A,B)$ different from $(A,B)$ as well as $(X_G,V(G))$, contradicting the assumption that $(A,B)$
is strongly linked to the roots.  Therefore, we have $A\cap B\setminus \{v\}\subseteq M$.

Suppose now that the root separation $(M,N)$ of $L$ is not 4-light.  Since $G$ is 4-light, the root separations $(M,N)$ and $(M',N)$
have different strictly right-hand sides.  That is, we have $v\in N\setminus M$, the root separation $(M',N)$ of $G$ has order $|M\cap N|+1$,
and the rooted graph $R^L_{M,N}$ is obtained from $R^G_{M',N}+E(H[A\cap B\cap N])$ by making $v$ into a non-root vertex.
Since $A\cap B\setminus \{v\}\subseteq M$, we have $(A\cap B\cap N)\setminus (M\cap N)=\{v\}$,
and thus only the $v(M,N)+\beta$ edges of $L$ between $v$ and $M\cap N$ contribute to $\rho_4(R^L_{M,N})$ but not to $\rho_4(R^G_{M',N})$.
Therefore,
\begin{align*}
\rho_4(R^L_{M,N})&=\rho_4(R^G_{M',N})+v(M,N)+\beta-4\\
&\le \rho_4(R^G_{M',N})+|M\cap N|-4\le \rho_4(R^G_{M',N}).
\end{align*}
\end{subproof}

We now further modify the $\id$-rooted minor $L$ of $G$ to eliminate the root separations of order four with very dense right-hand sides.
Let $L_1$ be the $\id$-rooted minor of $L$ and let $K$ be a set of vertices of $L_1$ defined as follows:
\begin{itemize}
\item If $L$ has a $K_\star$-universal root separation $(C_1,D_1)$ such that $v\in D_1\setminus C_1$, then choose one with
$D_1$ maximal and subject to that with $C_1$ minimal, let $L_1$ be the torso of $(C_1,D_1)$, and let $K=C_1\cap D_1$.
\item If $L$ does not have any such root separation, then let $L_1=L$ and let $K$ consist of $v$ and all its neighbors in $L$.
\end{itemize}
The first case is illustrated in Figure~\ref{fig-noK4plus}.
Let us prove the key properties of $L_1$.

\begin{figure}
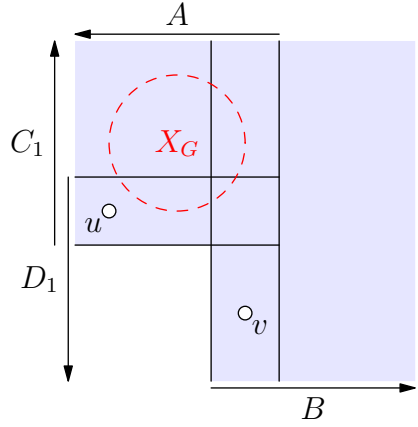

\begin{center}
\begin{asy}
v[0] = (0,0);
v[1] = (2,0);
v[2] = (3,0);
v[3] = (0,-2);
v[4] = (3,-2);
v[5] = (0,-3);
v[6] = (3,-3);
v[7] = (2,-5);
v[8] = (3,-5);
v[9] = (5,0);
v[10] = (5,-5);
v[11] = (2, -3);
v[12] = (0,-5);

fill(v[0]--v[5]--v[11]--v[7]--v[10]--v[9]--cycle, interp(blue, white, 0.9));

draw(L="$A$", g=(v[2]+(0,0.1)) -- (v[0]+(0,0.1)), arrow=Arrow, align=N);
draw(L="$B$", g=(v[7]+(0,-0.1)) -- (v[10]+(0,-0.1)), arrow=Arrow, align=S);
draw(L="$C_1$", g=(v[5]+(-0.3,0)) -- (v[0]+(-0.3,0)), arrow=Arrow, align=W);
draw(L="$D_1$", g=(v[3]+(-0.1,0)) -- (v[12]+(-0.1,0)), arrow=Arrow, align=W);

draw(v[1] -- v[7]);
draw(v[2] -- v[8]);
draw(v[3] -- v[4]);
draw(v[5] -- v[6]);

v[14] = (1.5,-1.5);
draw (circle (v[14],1), dashed+red);
label ("$X_G$", v[14], red);

v[13] = interp (v[8],v[11],0.5);
vertex (v[13]);
label ("$v$", v[13], SE);

v[15] = (0.5,-2.5);
vertex (v[15]);
label ("$u$", v[15], SW);
\end{asy}
\end{center}
\caption{The separations from the proof of Lemma~\ref{lemma-noK4plus}.  The set $C'_1$ is shown in blue.}\label{fig-noK4plus}
\end{figure}

\begin{claim}\label{cl-L1heavy}
The 5-rooted graph $L_1$ does not have any proper $K_\star$-universal root separation $(C,D)$ such that $K\subseteq D$,
and every non-4-light root separation $(C',D')$ of $L_1$ of order at most four satisfies $K\subseteq D'$ and $K\not\subseteq C'$.
Moreover, $\rho_4(L_1)\ge \rho_4(L)$.
\end{claim}
\begin{subproof}
Suppose first that $L_1$ has a proper $K_\star$-universal root separation $(C,D)$ such that $K\subseteq D$.
We have $L_1=L$, as otherwise $(C,D\cup D_1)$ would be a $K_\star$-universal root separation of $L$ contradicting the
maximality of $D_1$ or the minimality of $C_1$.  However then $K$ consists of $v$ and all its neighbors,
and thus $(C\setminus\{v\},D)$ is a $K_\star$-universal root separation of $L$ not containing $v$ in its left-hand side,
and we should have considered it as a candidate for the root separation $(C_1,D_1)$.
This is a contradiction.

Let us now consider a non-4-light root separation $(C',D')$ of $L_1$.
If $L_1=L$, then Claim~\ref{cl-separof} implies $v\in D'\setminus C'$, and since $K$ contains $v$ and its neighbors, we have $K\subseteq D'$.
If $L_1\neq L$, then $K$ is a clique in $L_1$, and thus $K\subseteq D'$ or $K\subseteq C'$.
However, the latter is not possible, since $(C'\cup D_1,D')$ would be a non-4-light root separation
of $L$ contradicting Claim~\ref{cl-separof} (recall that $v\in D_1\setminus C_1=D_1\setminus V(L_1)$).

Finally, let us argue that $\rho_4(L_1)\ge \rho_4(L)$.
This is trivial if $L_1=L$.  Hence, we can assume that $L_1$ is the torso of the $K_\star$-universal root separation $(C_1,D_1)$ of $L$.
Consequently,
\begin{equation}\label{eq-l1l}
\rho_4(L_1)=\rho_4(L)-\rho_4(R^L_{C_1,D_1})+k^L(C_1\cap D_1)\ge \rho_4(L)-\rho_4(R^L_{C_1,D_1}).
\end{equation}
If the root separation $(C_1,D_1)$ of $L$ is 4-light, this gives the inequality $\rho_4(L_1)\ge \rho_4(L)$.
Therefore, let us assume that $(C_1, D_1)$ is not 4-light.

By Observation~\ref{obs-cliques}, every isolator of $(C_1,D_1)$ in $L$ is $K_\star$-universal, and thus by the maximality of $D_1$ and the minimality of $C_1$,
the root separation $(C_1,D_1)$ of $L$ is strongly linked to the roots.  Since $K_5$ is not a rooted minor of $G$,
it follows that $|C_1\cap D_1|\le 4$.  Claim~\ref{cl-separof} then gives $A\cap B\setminus\{v\}\subseteq C_1$.
Let $C'_1=C_1\cup B$, so that $(C'_1,D_1)$ is a root separation of $G$ of order $|C_1\cap D_1|+1$.
Since the root separation $(C_1, D_1)$ of $L$ is not 4-light, Claim~\ref{cl-separof} gives
\begin{equation}\label{eq-c1d1heavy}
\rho_4(R^G_{C'_1,D_1})\ge \rho_4(R^L_{C_1,D_1})>0,
\end{equation}
and since $G$ is 4-light, it follows
that $|C_1\cap D_1|=4$ and $|C'_1\cap D_1|=5$.
The root separation $(C_1, D_1)$ of $L$ is exposed, since it is linked to the roots,
and the root separation $(A,B)$ of $G$ is right-linked, since it is $K_{1,\star}$-universal;
and thus Lemma~\ref{lemma-showexposed} implies that the root separation $(C'_1,D_1)$ of $G$ is exposed.

Note that there exists a vertex $u\in C_1\cap D_1\setminus B$.  Indeed, since the root separation $(A,B)$ is linked to the roots,
$G[(A\setminus B)\cup \{v\}]$ contains a path from $X_G$ to $v$, and since $(C_1,D_1)$ is a root separation of $L\supseteq G[A]$, this path must pass through $C_1\cap D_1\setminus B$.
Since $u\in A\setminus B$, the vertex $u$ has the same neighborhood in $L$ as in $G$, and thus
\begin{align*}
k^L(C_1\cap D_1)&=k^G(C_1\cap D_1)+k^L(C_1\cap D_1\setminus\{u\})-k^G(C_1\cap D_1\setminus\{u\})\\
&\ge k^G(C_1\cap D_1)-k^G(C_1\cap D_1\setminus\{u\}).
\end{align*}
Let us remark that $k^G(C_1\cap D_1\setminus\{u\})\le |E(K^G_{C_1\cap D_1\setminus\{u\}})|\le 3$.
Moreover, since $v\not\in X_G$, the vertex $v$ is incident with exactly four edges of $K^G_{C'_1\cap D_1}$, and thus
$$k^G(C'_1\cap D_1)=k^G(C_1\cap D_1)+4-v(C_1,D_1).$$
Using (\ref{eq-l1l}) and Claim~\ref{cl-separof}, we get
\begin{align*}
\rho_4(L_1)-\rho_4(L)&=k^L(C_1\cap D_1)-\rho_4(R^L_{C_1,D_1})\\
&=k^L(C_1\cap D_1)-\rho_4(R^G_{C'_1,D_1})-v(C_1,D_1)-\beta+4\\
&\ge k^G(C_1\cap D_1)-k^G(C_1\cap D_1\setminus\{u\})-\rho_4(R^G_{C'_1,D_1})-v(C_1,D_1)-\beta+4\\
&=k^G(C'_1\cap D_1)-k^G(C_1\cap D_1\setminus\{u\})-\rho_4(R^G_{C'_1,D_1})-\beta\\
&\ge k^G(C'_1\cap D_1)-3-\rho_4(R^G_{C'_1,D_1})-\beta.
\end{align*}

By (\ref{eq-c1d1heavy}) and Lemma~\ref{lemma-ubdel}, we have $1\le \rho_4(R^G_{C'_1,D_1})\le 4$.
By Corollary~\ref{cor-neic2}, we have either $k^G(C'_1\cap D_1)-\rho_4(R^G_{C'_1,D_1})\ge 4$,
or $k^G(C'_1\cap D_1)-\rho_4(R^G_{C'_1,D_1})=3$ and $K^G_{C'_1\cap D_1}-E(G[C'_1\cap D_1])$ is isomorphic to $C_4+K_1$.
In the latter case, the 3-vertex subgraph $K^G_{C_1\cap D_1\setminus\{u\}}-E(G[C_1\cap D_1]\setminus \{u\})$ of this graph has at most
two edges, and thus $k^G(C_1\cap D_1\setminus\{u\})\le 2$.  In either case, it follows that
$\rho_4(L_1)-\rho_4(L)\ge 1-\beta\ge 0$, and thus $\rho_4(L_1)\ge\rho_4(L)$.
\end{subproof}
By (\ref{eq-novaprofit}) and Claim~\ref{cl-L1heavy}, the $\id$-rooted minor $L_1$ of $G$ satisfies $\rho_4(L_1)>\rho_4(G)$.
Since $n(L_1)<n(G)$, the 5-rooted graph $L_1$ cannot be $4$-light, as otherwise this would contradict Observation~\ref{obs-lightminor}.
Let $(C_2,D_2)$ be a non-4-light root separation of $L_1$ of order at most four, chosen with $D_2$ minimal.
Claim~\ref{cl-L1heavy} implies that $K\subseteq D_2$ and that the root separation $(C_2,D_2)$ is not $K_\star$-universal.
By the minimality of $D_2$, observe that the rooted graph $R_{C_2,D_2}$ is $4$-light.
By Corollary~\ref{cor-k4}, it follows that $|C_2\cap D_2|=4$, $\rho_4(R^{L_1}_{C_2,D_2})=1$, and the root separation $(C_2,D_2)$ is $K^-_\star$-universal.
Moreover, Claim~\ref{cl-L1heavy} implies that no isolator of $(C_2,D_2)$ is $K_\star$-universal,
and thus by Lemma~\ref{lemma-diamonds}, the root separation $(C_2,D_2)$ is linked to the roots in $L_1$ and every isolator $(C,D)$
of $(C_2,D_2)$ in $L_1$ is $K^-_\star$-universal and satisfies $\rho_4(R^{L_1}_{C,D})\le 1$.

Let $(C_3,D_3)$ be an isolator of $(C_2,D_2)$ with $C_3$ minimal, so that $(C_3,D_3)$ is strongly linked to the roots.
Let $S=C_3\cap D_3$, let $L_3$ be the $2$-twin reduction of $(C_3,D_3)$, let $Y$ be the set of two vertices added in this
reduction, and let $S'=S\cup Y$.  Since the root separation $(C_3,D_3)$ of $L_1$ is strongly linked to the roots, 
the root separation $(C_3,S')$ of $L_3$ with the same left-hand side is also strongly linked to the roots.

We claim that the 5-rooted graph $L_3$ is 4-light.
Indeed, suppose for a contradiction that a root separation $(C,D)$ of $L_3$ of order at most four is not $4$-light.
By Lemma~\ref{lemma-2twin}, we can assume that $S'\subseteq C$ and $Y\cap C\cap D=\emptyset$.  Let $C'=(C\setminus Y)\cup D_3$;
then $(C',D)$ is a root separation of $L_1$ with the same strictly right-hand side as the root separation $(C,D)$ of $L_3$, and thus $(C',D)$ is not 4-light.
However, since $K\subseteq D_3\subseteq C'$, this contradicts Claim~\ref{cl-L1heavy}.

Let $F$ be a flaw of $G$.  By Claim~\ref{cl-L1heavy} and (\ref{eq-novaprofit}), we have
$$\rho_4(L_3)=\rho_4(L_1)-\rho_4(R^{L_1}_{C_3,D_3})\ge \rho_4(L_1)-1\ge \rho_4(L)-1\ge\rho_4(G),$$
and thus $F$ belongs to the target of $L_3$.
Since $L_1$ is an $\id$-rooted minor of $G$, the graph $F$ is not an $\id$-rooted minor of $L_1$,
and since $L_3$ has the same census as $L_1$ by Observation~\ref{obs-samecensus}, $F$ also is not an $\id$-rooted minor of $L_3$.
Consequently, $L_3$ is a counterexample.

Note that $n(R^{L_1}_{C_2,D_2})\ge 2$, since $(C_2,D_2)$ is a root separation of order at most four and $\rho_4(R^{L_1}_{C_2,D_2})=1$.
Since the root separation $(A,B)$ of $G$ is proper, we have
\begin{align*}
n(L_3)&=n(L_1)-n(R^{L_1}_{C_3,D_3})+2\le n(L_1)-n(R^{L_1}_{C_2,D_2})+2\\
&\le n(L_1)\le n(L)=n(L^G_{A,B})<n(G).
\end{align*}
This contradicts the minimality of the counterexample $G$.
\end{proof}

In particular, we get the following bound.

\begin{corollary}\label{cor-noK4plus}
If $G$ is a minimal counterexample, then every proper root separation $(A,B)$ of $G$ of order five satisfies
$\rho_4(R_{A,B})\le 3$, and moreover
\begin{itemize}
\item if $(A,B)$ is $K_{1,\star}$-universal (and in particular if $\rho_4(R_{A,B})\ge 1$), then $(A,B)$ is not nova-universal and $|A\cap B\cap X_G|\le 2$;
\item if $\rho_4(R_{A,B})=2$, then $\widetilde{R}_{A,B}$ is nearly $K_{2,3}^+$-universal, and if it is exposed, then $k^G(A\cap B)\ge 7$; and
\item if $\rho_4(R_{A,B})=3$, then $\widetilde{R}_{A,B}$ is $K_{2,3}^+$-universal, $G[A\cap B]$ together with the clique on $A\cap B\cap X_G$ forms a matching, and $k^G(A\cap B)\ge 8$.
\end{itemize}
\end{corollary}
\begin{proof}
If $\rho_4(R_{A,B})\ge 1$, then $(A,B)$ is $K_{1,\star}$-universal by Corollary~\ref{cor-nok4}.  Hence, we can in all cases assume that $(A,B)$ is $K_{1,\star}$-universal.
If $(A,B)$ were nova-universal, then by Lemma~\ref{lemma-profitable-nova}, $G$ would also have a proper $K_{1,\star}$-universal
root separation $(A',B')$ of order five strongly linked to the roots such that an $(A',B')$-profitable
$(G,A'\cap B')$-nova is an $\id$-rooted minor of $R_{A',B'}$.  Since this would contradict Lemma~\ref{lemma-noK4plus},
it follows that $(A,B)$ is not nova-universal.

We claim that $|A\cap B\cap X_G|\le 2$.  Indeed, suppose for a contradiction that $|A\cap B\cap X_G|\ge 3$.
Then Corollary~\ref{cor-noplus1} implies that $|A\cap B\cap X_G|=3$; let $\{v_1,v_2\}=A\cap B\setminus X_G$.
However, note that for each $i\in\{1,2\}$, the $v_{3-i}$-star is a $(G,A\cap B,v_i)$-nova.
Since the root separation $(A,B)$ is $K_{1,\star}$-universal, it follows that it is also nova-universal, which is a contradiction.

Similarly, the root separation $(A,B)$ cannot be $(K_4+K_1)$-universal, as otherwise it would be nova-universal.
By Corollary~\ref{cor-nok4}, it follows that $\rho_4(R_{A,B})\le 3$.
\begin{itemize}
\item If $\rho_4(R_{A,B})=2$, then Corollary~\ref{cor-nok4} implies that $\widetilde{R}_{A,B}$ is nearly $K_{2,3}^+$-universal.
If $(A,B)$ is also exposed, then we have $k^G(A\cap B)\ge 7$ by Observation~\ref{obs-nearbc}.
\item If $\rho_4(R_{A,B})=3$, then Corollary~\ref{cor-nok4} implies that $\widetilde{R}_{A,B}$ is $K_{2,3}^+$-universal.
By Corollary~\ref{cor-litor}, the root separation $(A,B)$ is linked to the roots, and thus also exposed.
Observation~\ref{obs-nearbc} then implies that $G[A\cap B]$ together with the clique on $A\cap B\cap X_G$ forms a matching and $k^G(A\cap B)\ge 8$.
\end{itemize}
\end{proof}

\section{Splits}

Next, we aim to exclude the root separations $(A,B)$ of order five such that $\rho_4(R_{A,B})\ge 2$, as well as many of those with $\rho_4(R_{A,B})=1$.
Given a rooted graph $G$ and sets $T\subseteq S\subseteq V(G)$, the \emph{$(G,S,T)$-split}
is the graph obtained from $K^G_S$ by removing all edges between the vertices of $T$.
The case that $S=A\cap B$ for a root separation $(A,B)$ of order five and that $|T|=3$ is the most interesting for us, but
we will also apply the notion when $(A,B)$ has order four and $|T|=2$.
For both these applications, we need to understand small root separations in the graph obtained by contracting $R_{A,B}$ to the $(G,A\cap B,T)$-split.
Let us start by an easy observation.

\begin{observation}\label{obs-splitcut}
Let $L$ be a 5-rooted graph and let $T\subset S$ be sets of vertices of $L$ such that $2\le |T|\le 3$,
$|S|=|T|+2$, $|S\cap X_L|\le 2$, the $(L,S,T)$-split is a subgraph of $L$, and the root separation $(V(L),S)$ of $L$ is strongly linked to the roots.
If $(C,D)$ is a root separation of $L$ of order at most four and $S\not\subseteq C$, 
then $|C\cap D|=4$, $S\setminus T\subset C\cap D$, $|T\setminus C|=1$ and $|T\setminus D|=|T|-1$, and the root separation $(C,D)$ is linked to the roots.
Moreover, each vertex of $C\cap D$ has a neighbor in the component of $L[D\setminus C]$ containing the unique vertex of $T\setminus C$.
\end{observation}
\begin{proof}
Since $S\not\subseteq C$, there exists a vertex $v\in S\setminus C$.
Since the root separation $(V(L),S)$ is strongly linked to the roots, there also exists a vertex $v'\in S\setminus D$.  Since $vv'\not\in E(L)$, $v\not\in X_L$,
and $L[S]$ is a supergraph of the $(L,S,T)$-split, it follows that $v,v'\in T$.  Moreover, all vertices in $S\setminus T$
are adjacent to $v$ and either belong to $X_L$ or are adjacent to $v'$, and thus $S\setminus T\subset C\cap D$.

Let $D_1=D\cup T$.  Then $(C,D_1)$ is a root separation of $L$ of order $|C\cap D|+|T\setminus D|$
and $S\subseteq D_1$.  Note that $(C,D_1)\neq(X_L,V(L))$, since
$|S\cap C|\ge |(S\setminus T)\cup \{v'\}|=3$ and $|S\cap X_L|\le 2$.
Moreover, $(C,D_1)\neq (V(L),S)$, since $v\in D_1\setminus C$.
Since the root separation $(V(L),S)$ is strongly linked to the roots,
it follows that $(C,D_1)$ has order at least $|S|+1$.
Since $|S|=|T|+2$ and $|T\setminus D|\le |T\setminus\{v\}|=|T|-1$, we have
$$|S|+1\le |C\cap D|+|T\setminus D|\le |C\cap D|+|T|-1=|C\cap D|+|S|-3.$$
It follows that $|C\cap D|=4$ and the inequality is tight,
implying that $|T\setminus D|=|T|-1$.

Since each isolator $(C',D')$ of $(C,D)$ satisfies $S\not\subseteq C'$,
this also shows that $|C'\cap D'|=4$, and thus $(C,D)$ is linked to the roots.
Moreover, each vertex $z\in C\cap D$ has a neighbor in the component $K$ of $L[D\setminus C]$ containing the unique vertex $v\in T\setminus C$,
as otherwise $(C'',D'')=(V(L)\setminus V(K), V(K)\cup (C\cap D\setminus\{z\}))$ would be a root separation
of $L$ of order three with $v\in D''\setminus C''$.
\end{proof}

\begin{figure}
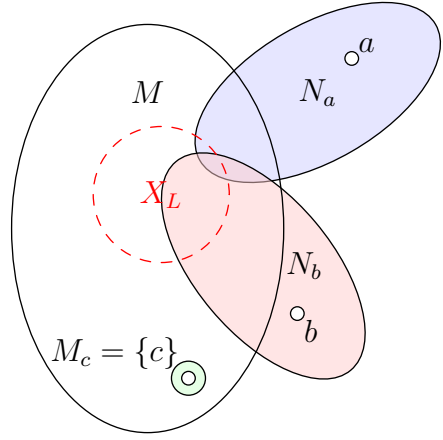

\begin{center}
\begin{asy}
v[0] = (0,0);
v[1] = (2.5,2);
v[2] = v[1] + (0.5,0.5);

v[3] = (1.7,-0.55);
v[4] = v[3] + (0.5,-0.7);

v[5] = (0.6,-2.2);

filldraw (shift (v[1]) * rotate (30) * ellipse ((0,0),2,1), fillpen=interp(blue, white, 0.9));
filldraw (shift (v[3]) * rotate (-50) * ellipse ((0,0),2,1), fillpen=interp(red, white, 0.8) + opacity(0.5));
filldraw (circle (v[5], 0.25), fillpen=interp(green, white, 0.9));
vertex (v[5]);
label ("$M_c=\{c\}$", v[5], 1.4NW);

draw (ellipse (v[0],2,3));
label ("$M$", v[0] + (0, 2));

label ("$N_a$", v[1]);
label ("$N_b$", v[3]+(0.6,0));

vertex (v[2]);
label ("$a$", v[2],NE);
vertex (v[4]);
label ("$b$", v[4],SE);

v[14] = (0.2,0.5);
draw (circle (v[14],1), dashed+red);
label ("$X_L$", v[14], red);
\end{asy}
\end{center}
\caption{An $\{a,b,c\}$-partition of a rooted graph.}\label{fig-tpartition}
\end{figure}
We now aim to show that in the situation of Observation~\ref{obs-splitcut},
we can actually decompose $L$ into a main part and $|T|$ almost disjoint parts
around the vertices of $T$ cut off by root separations of order four.
More precisely (see Figure~\ref{fig-tpartition} for an illustration), let $M$ be a subset of $V(L)$,
and for each $v\in T$, let $N_v$ be a subset of $V(L)$.  We say that $(M,\{N_v:v\in T\})$ is a \emph{$T$-partition of $L$}
if
\begin{itemize}
\item $\bigl(M,\bigcup_{v\in T} N_v\bigr)$ is a root separation of $L$,
\item for each $v\in T$, we have $T\cap N_v=\{v\}$, and
\item the sets $N_v\setminus M$ for distinct $v\in T$ are pairwise disjoint and $L$ does not contain any edges between them.
\end{itemize}
We say that the $T$-partition \emph{$S$-captures root $(\le\!4)$-separations} of $L$ if for every vertex $v\in T$,
\begin{itemize}
\item if every root separation $(C,D)$ of $L$ of order at most four satisfies $v\in C$, then $N_v=\{v\}\subset M$;
\item otherwise,
\begin{itemize}
\item $v\in N_v\setminus M$, $|M\cap N_v|=4$, $S\setminus T\subset M\cap N_v$, 
the root separation $(V(L)\setminus (N_v\setminus M), N_v)$ is linked to the roots, and the graph $L[N_v\setminus M]$ is connected, and
\item every root separation $(C,D)$ of $L$ of order at most four such that $v\in D\setminus C$
satisfies $C\cap D\subseteq N_v$, and if the graph $L[D\setminus C]$ is connected, then $D\setminus C\subseteq N_v\setminus M$.
\end{itemize}
\end{itemize}

\begin{lemma}\label{lemma-tricut}
Let $L$ be a 5-rooted graph and let $T\subset S$ be sets of vertices of $L$ such that $2\le |T|\le 3$,
$|S|=|T|+2$, $|S\cap X_L|\le 2$, and the root separation $(V(L),S)$ of $L$ is strongly linked to the roots.
If the $(L,S,T)$-split is a subgraph of $L$, then there exists a $T$-partition $(M,\{N_v:v\in T\})$ which $S$-captures root $(\le\!4)$-separations of $L$.
\end{lemma}
\begin{proof}
For each vertex $v\in T$:
\begin{itemize}
\item If $v\in C$ holds for every root separation $(C,D)$ of $L$ of order at most four, then let $N_v=\{v\}$ and $M_v=V(L)$.
\item Otherwise, let $(M'_v,N'_v)$ be a root separation of $L$ of order at most four such that $v\in N'_v\setminus M'_v$,
chosen with $N'_v$ maximal and subject to that with $M'_v$ minimal.  By Observation~\ref{obs-splitcut}, we have $|M'_v\cap N'_v|=4$,
$(M'_v,N'_v)$ is linked to the roots, $S\setminus T\subset M'_v\cap N'_v$, and $T\setminus \{v\}\subset M'_v\setminus N'_v$.
Let $K_v$ be the component of $L[N'_v\setminus M'_v]$ containing $v$, and let $(M_v,N_v)=(V(L)\setminus V(K_v),V(K_v)\cup (M'_v\cap N'_v))$.
Then $(M_v,N_v)$ is a root separation of $L$ such that $M_v\cap N_v=M'_v\cap N'_v$, $M'_v\subseteq M_v$, and $N_v\subseteq N'_v$,
and thus the root separation $(M_v,N_v)$ has order four, is linked to the roots, $S\setminus T\subset M_v\cap N_v$, and $T\setminus \{v\}\subset M_v\setminus N_v$.
\end{itemize}

\begin{figure}
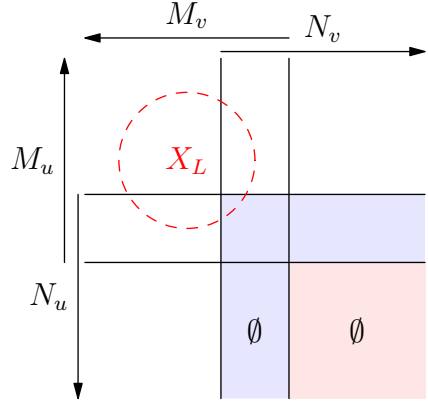

\begin{center}
\begin{asy}
v[0] = (0,0);
v[1] = (2,0);
v[2] = (3,0);
v[3] = (0,-2);
v[4] = (2,-2);
v[5] = (0,-3);
v[6] = (3,-3);
v[7] = (2,-5);
v[8] = (3,-5);
v[9] = (5,0);
v[10] = (5,-5);
v[11] = (2, -3);
v[12] = (0,-5);
v[13] = (5,-2);
v[14] = (5,-3);

fill(v[7]--v[4]--v[13]--v[14]--v[6]--v[8]--cycle, interp(blue, white, 0.9));
fill(v[6]--v[14]--v[10]--v[8]--cycle, interp(red, white, 0.9));

draw(L="$M_v$", g=(v[2]+(0,0.3)) -- (v[0]+(0,0.3)), arrow=Arrow, align=N);
draw(L="$N_v$", g=(v[1]+(0,0.1)) -- (v[9]+(0,0.1)), arrow=Arrow, align=N);
draw(L="$M_u$", g=(v[5]+(-0.3,0)) -- (v[0]+(-0.3,0)), arrow=Arrow, align=W);
draw(L="$N_u$", g=(v[3]+(-0.1,0)) -- (v[12]+(-0.1,0)), arrow=Arrow, align=W);

draw(v[1] -- v[7]);
draw(v[2] -- v[8]);
draw(v[3] -- v[13]);
draw(v[5] -- v[14]);

v[15] = (1.5,-1.5);
draw (circle (v[15],1), dashed+red);
label ("$X_L$", v[15], red);

label ("$\emptyset$", interp (v[6],v[7],0.5));
label ("$\emptyset$", interp (v[8],v[14],0.5));
\end{asy}
\end{center}
\caption{The separations from the proof of Lemma~\ref{lemma-tricut}.  The set $C_2\cap D_2$ is shown in blue, the set $D_2\setminus C_2$ in red.}\label{fig-splitcross}
\end{figure}

Let us now consider distinct vertices $u,v\in T$ such that $(M_u,N_u)\neq(V(L),\{u\})$ and $(M_v,N_v)\neq(V(L),\{v\})$,
see Figure~\ref{fig-splitcross}.
Then $(C_1,D_1)=(M_u\cap M_v,N_u\cup N_v)$ is a root separation of $L$ such that $u,v\in D_1\setminus C_1$,
and thus Observation~\ref{obs-splitcut} implies that the order of this separation is at least five.  By Observation~\ref{obs-submod},
the root separation $(C_2,D_2)=(M_u\cup M_v,N_u\cap N_v)$ has order at most three.  Since
$S\setminus T\subseteq M_u\cap N_u\cap M_v\cap N_v$ and $|S\setminus T|=2$,
we have
\begin{align*}
3&\ge |C_2\cap D_2|\\
&=|M_u\cap N_u\cap M_v\cap N_v|+|C_2\cap D_2\setminus M_u|+|C_2\cap D_2\setminus M_v|\\
&\ge 2+|C_2\cap D_2\setminus M_u|+|C_2\cap D_2\setminus M_v|.
\end{align*}
By symmetry, we can assume that $C_2\cap D_2\setminus M_u=\emptyset$.  Note that $C_2\cap D_2\setminus M_u=(M_v\cap N_v)\setminus M_u$,
and thus the graph $L[N_u\setminus M_u]$ does not have any edge with one end in $M_v\setminus N_v$ and the other end in $N_v$.
Since $u\in N_u\setminus (M_u\cup N_v)$ and the graph $L[N_u\setminus M_u]=K_u$ is connected, it follows
that $(N_u\setminus M_u)\cap N_v=\emptyset$.
Consequently, the sets $N_u\setminus M_u$ and $N_v\setminus M_v$ are disjoint and 
$L$ does not contain any edge between them.  This is also clearly the case if
$(M_u,N_u)=(V(L),\{u\})$ or $(M_v,N_v)=(V(L),\{v\})$.
This implies that $(M,\{N_v:v\in T\})$, where $M=\bigcap_{z\in T} M_z$,
is a $T$-partition of $L$.

Let us now consider any vertex $v\in T$ and any root separation $(C,D)$ of $L$ of
order at most four such that $v\in D\setminus C$.  Since such a root separation exists, we have $(M_v,N_v)\neq (V(L),\{v\})$.
Let $(C_3,D_3)=(M'_v\cap C,N'_v\cup D)$ and $(C_4,D_4)=(M'_v\cup C,N'_v\cap D)$.  Note that $(C_3,D_3)$ and $(C_4,D_4)$ are root separations of $L$
and that $v\in D_3\setminus C_3$ and $v\in D_4\setminus C_4$.  Hence, Observation~\ref{obs-splitcut} implies that $|C_3\cap D_3|\ge 4$ and $|C_4\cap D_4|\ge 4$.
On the other hand, by Observation~\ref{obs-submod}, we have $|C_3\cap D_3|+|C_4\cap D_4|=|M'_v\cap N'_v|+|C\cap D|\le 8$,
and thus $|C_3\cap D_3|=|C_4\cap D_4|=4$.

Since $C_3\subseteq M'_v$ and $N'_v\subseteq D_3$, the maximality of $N'_v$ and the minimality of $M'_v$ in the choice of the root separation $(M'_v,N'_v)$ implies that $D_3=N'_v$ and $C_3=M'_v$,
and thus $M'_v\subseteq C$ and $D\subseteq N'_v$.  Let $K$ be the component of $L[D\setminus C]$ containing $v$.
Then $K$ is a connected subgraph of $L[N'_v\setminus M'_v]$ containing $v$, and thus $K\subseteq K_v$.
By Observation~\ref{obs-splitcut}, each vertex of $C\cap D$ has a neighbor in $V(K)\subseteq V(K_v)=N_v\setminus M_v$,
and thus $C\cap D\subseteq N_v$.

Therefore, the $T$-partition $(M,\{N_v:v\in T\})$ $S$-captures root $(\le\!4)$-separations of $L$.
\end{proof}

For a root separation $(A,B)$ of a rooted graph $G$ and a set $T\subset A\cap B$ of size at most three, we say that a $(G,A\cap B,T)$-split $H$ is
\emph{$(A,B)$-profitable} when
$$|E(H)\setminus E(G)|\ge \begin{cases}
\rho_4(R_{A,B})+2&\text{if $|T\setminus X_G|=3$}\\
\rho_4(R_{A,B})&\text{if $|T\setminus X_G|\le 2$.}
\end{cases}$$
With Lemma~\ref{lemma-tricut}, it is relatively simple to exclude profitable splits.
\begin{lemma}\label{lemma-nocotri}
Let $G$ be a minimal counterexample, let $(A,B)$ be a proper right-linked root separation of $G$ of order four or five strongly linked
to the roots such that $|A\cap B\cap X_G|\le 2$, let $T\subset A\cap B$ be a set of size $|A\cap B|-2$,
and let $H$ be a $(G,A\cap B,T)$-split.  If $H$ is $(A,B)$-profitable, then $H$ is not an $\id$-rooted minor of $R_{A,B}$.
\end{lemma}
\begin{proof}
Suppose for a contradiction that $H$ is an $\id$-rooted minor of $R_{A,B}$, and let $L=L_{A,B}\cup H$
be the corresponding $\id$-rooted minor of $G$.  Since $H$ is $(A,B)$-profitable, we have
\begin{equation}\label{eq-prof}
\rho_4(L)\ge \begin{cases}
\rho_4(G)+2&\text{if $|T\setminus X_G|=3$}\\
\rho_4(G)&\text{if $|T\setminus X_G|\le 2$.}
\end{cases}
\end{equation}
Note that the 5-rooted graph $L$ and the sets $S=A\cap B$ and $T$
satisfy the assumptions of Lemma~\ref{lemma-tricut}; let $(M,\{N_v:v\in T\})$ be the $T$-partition which $S$-captures $(\le\!4)$-separations of $L$.
For every $v\in T$, let $M_v=V(L)\setminus (N_v\setminus M)$, so that $(M_v,N_v)$ is
a root separation of $L$ linked to the roots.
Let $T'=\{v\in T: N_v\neq \{v\}\}$.  For every $v\in T'$, we have $v\in N_v\setminus M$, and thus $T'\cap X_G=\emptyset$.
Let us now constrain the non-4-light root separations of $L$.

\begin{claim}\label{cl-coodi}
For every non-4-light root separation $(C,D)$ of $L$ of order at most four, there exists a unique vertex $v\in T'$ such that $v\in D\setminus C$,
and moreover in that case $K^L_{M_v\cap N_v}$ is an $\id$-rooted minor of $R^L_{M_v,N_v}$.
\end{claim}
\begin{subproof}
We cannot have $S\subseteq C$, since otherwise the root separation $(C,D)$ of $L$ would have the same strictly right-hand side
as the root separation $(C\cup B,D)$ of $G$, and the counterexample $G$ is $4$-light.
By Observation~\ref{obs-splitcut}, it follows that $|C\cap D|=4$, the root separation $(C,D)$ is linked to the roots, and there exists a vertex $v\in T$ such that 
$S\setminus C=\{v\}$.

Let $K$ be the component of $G[D\setminus C]$ containing $v$ and let $(C',D')=(V(L)\setminus V(K), V(K)\cup (C\cap D))$ and $(C'',D'')=(C\cup V(K),D\setminus V(K))$.
Then both $(C',D')$ and $(C'',D'')$ are root separations of $L$ of order four linked to the roots, since $C'\cap D'=C''\cap D''=C\cap D$. 
Since $S\subseteq C''$, the argument from the previous paragraph shows that the root
separation $(C'',D'')$ of $L$ is $4$-light, and thus
$$\rho_4(R^L_{C',D'})=\rho_4(R^L_{C,D})-\rho_4(R^L_{C'',D''})>0.$$
Moreover, note that since $v\in D'\setminus C'$, the graph $L[D'\setminus C']=K$ is connected, and the $T$-partition $(M,\{N_u:u\in T\})$ $S$-captures $(\le\!4)$-separations of $L$,
we have $v\in T'$ and $D'\subseteq N_v$.

Consider any root separation $(I,J)$ of the 4-rooted graph $R^L_{C',D'}$ of order at most three.
Since $(C'\cup I,J)$ is a root separation of $L$ of order $|I\cap J|\le 3$, Observation~\ref{obs-splitcut} implies that $S\subseteq C'\cup I$,
and thus $(C'\cup I\cup B,J)$ is a root separation of $G$ with strictly right-hand side $\widetilde{R}^L_{I,J}$.
Since $G$ is 4-light, this implies that the root separation $(I,J)$ of $R^L_{C',D'}$ is 4-light.
Consequently the 4-rooted graph $R^L_{C',D'}$ is 4-light.  Since $\rho_4(R^L_{C',D'})>0$, Corollary~\ref{cor-k4} implies that
the root separation $(C',D')$ of $L$ is $K^-_\star$-universal.

Recall that $(C',D')$ is linked to the roots in $L$.  Let $\PP$ be a root linkage for $(C',D')$.
Moreover, recall that $D'\subseteq N_v$, and thus the terminators of $\PP$ are contained in $N_v$.
For each $z\in C'\cap D'$, let us consider the path of $\PP$ with terminator $z$, and let $P_z$ be the longest suffix
of this path contained in $N_v$; clearly $P_z$ starts in a vertex of $M_v\cap N_v$.
Moreover, since $|M_v\cap N_v|=|C'\cap D'|$, one of these suffixes starts in each vertex of $M_v\cap N_v$.
Since the $T$-partition $(M,\{N_v:v\in T\})$ $S$-captures $(\le\!4)$-separations of $L$, the two vertices $x$ and $y$ of $S\setminus T$ belong to $M_v\cap N_v$;
let $x',y'\in C'\cap D'$ be the vertices such that $P_{x'}$ starts in $x$ and $P_{y'}$ starts in $y$.  Since the root separation $(C',D')$ of $L$ is $K^-_\star$-universal,
we can contract $R^L_{C',D'}$ to its $\id$-rooted minor containing all edges between vertices of $C'\cap D'$ except for the edge $x'y'$.
By further contracting the paths $P_z$ for $z\in C'\cap D'$, we conclude that $R^L_{M_v,N_v}$ has an $\id$-rooted minor containing all edges between the vertices of $M_v\cap N_v$
except possibly for the edge $xy$.  However, if $\{x,y\}\not\subseteq X_G$, then $xy$ is an edge of the $(G,A\cap B,T)$-split $H$, and thus $xy\in E(L)$.
It follows that $K^L_{M_v\cap N_v}$ is an $\id$-rooted minor of $R^L_{M_v,N_v}$.
\end{subproof}

Let us now consider how contracting $R^L_{M_v,N_v}$ to $K^L_{M_v\cap N_v}$ affects the $4$-density.
\begin{claim}\label{cl-contrdens}
For every vertex $v\in T'$, we have $\rho_4(R^L_{M_v,N_v})\le 2$ and $k^L(M_v\cap N_v)\ge 2\rho_4(R^L_{M_v,N_v})$.
\end{claim}
\begin{subproof}
Note that $v$ is adjacent in $H\subset L$ to both vertices of $S\setminus T$, which are contained in $M_v\cap N_v$.
Let $d$ be the number of neighbors of $v$ among the two vertices of $M_v\cap N_v\setminus (S\setminus T)$,
which is the same in $G$ and in $L$. Thus, $d+2$ neighbors of $v$ in $L$ are contained in $M_v\cap N_v$.
Let $M'_v=M_v\cup B$, so that $(M'_v,N_v)$ is a root separation of $G$ of
order five, and observe that
\begin{equation}\label{eq-rmvnv}
\rho_4(R^L_{M_v,N_v})=\rho_4(R^G_{M'_v,N_v})-4+(d+2)=\rho_4(R^G_{M'_v,N_v})+d-2.
\end{equation}
Moreover, out of the $k^G(M'_v\cap N_v)$ edges of $E(K^G_{M'_v\cap N_v})\setminus E(G)$,
up to $4-d$ could be incident with $v$, and one of the remaining
ones (the one between the vertices of $S\setminus T$) could belong to $E(L)\setminus E(G)$; hence,
\begin{equation}\label{eq-kmvnv}
k^L(M_v\cap N_v)\ge k^G(M'_v\cap N_v)+d-5.
\end{equation}
Since the root separation $(M_v,N_v)$ of $L$ is linked to the roots and since the root separation $(A,B)$ of $G$ is right-linked by the assumptions,
Lemma~\ref{lemma-showexposed} implies that the root separation $(M'_v,N_v)$ of $G$ is exposed.
By Corollary~\ref{cor-noK4plus}, we have $\rho_4(R^G_{M'_v,N_v})\le 3$; let us now discuss several cases depending on the value of $\rho_4(R^G_{M'_v,N_v})$.
\begin{itemize}
\item If $\rho_4(R^G_{M'_v,N_v})=3$, then Corollary~\ref{cor-noK4plus} implies that $d\le 1$ and $k^G(M'_v\cap N_v)\ge 8$.
Hence, (\ref{eq-rmvnv}) gives
$$\rho_4(R^L_{M_v,N_v})=d+1\le 2$$
and by (\ref{eq-kmvnv}),
$$k^L(M_v\cap N_v)\ge 3+d\ge 2(1+d)=2\rho_4(R^L_{M_v,N_v}).$$
\item If $\rho_4(R^G_{M'_v,N_v})=2$, then note that $k^G(M'_v\cap N_v)\ge 7$ by Corollary~\ref{cor-noK4plus},
(\ref{eq-rmvnv}) gives
$$\rho_4(R^L_{M_v,N_v})=d\le 2,$$
and by (\ref{eq-kmvnv}),
$$k^L(M_v\cap N_v)\ge 2+d\ge 2d=2\rho_4(R^L_{M_v,N_v}).$$
\item Finally, let us consider the case that $\rho_4(R^G_{M'_v,N_v})\le 1$.
In this case, (\ref{eq-rmvnv}) gives
$$\rho_4(R^L_{M_v,N_v})=\rho_4(R^G_{M'_v,N_v})+d-2\le d-1\le 1.$$
If $\rho_4(R^L_{M_v,N_v})\le 0$, then $k^L(M_v\cap N_v)\ge 2\rho_4(R^L_{M_v,N_v})$ trivially holds.
Hence, we can assume that $\rho_4(R^L_{M_v,N_v})=1$, $d=2$, and $\rho_4(R^G_{M'_v,N_v})=1$.

To finish the proof of Claim~\ref{cl-contrdens},
suppose for a contradiction that $k^L(M_v\cap N_v)\le 2\rho_4(R^L_{M_v,N_v})-1=1$.
Then (\ref{eq-kmvnv}) gives $k^G(M'_v\cap N_v)\le 4$.
By Corollary~\ref{cor-neic2}, the graph $K^G_{M'_v\cap N_v}-E(G[M'_v\cap N_v])$
is isomorphic to $C_4+K_1$.  Since $d=2$, all edges of this graph incident with $v$
have the other end in $S\setminus T$, and thus $E(K^L_{M_v\cap N_v})\setminus E(L[M_v\cap N_v])$
can differ from $E(K^G_{M'_v\cap N_v})\setminus E(G[M'_v\cap N_v])$ only on edges between the
three vertices of $\{v\}\cup (S\setminus T)$.  Since the graph $C_4+K_1$ is triangle-free,
it follows that $k^L(M_v\cap N_v)\ge k^G(M'_v\cap N_v)-2=2$, which is a contradiction.
\end{itemize}
\end{subproof}

Let $T''$ consist of the vertices $v\in T'$ such that there exists a non-4-light root separation $(C,D)$ of $L$ of order
at most four with $v\in D\setminus C$, and for each such vertex $v$, let $K_v=K^L_{M_v\cap N_v}$ and
recall that by Claim~\ref{cl-coodi}, $K_v$ is an $\id$-rooted minor of $R^L_{M_v,N_v}$.
Let $L_1$ be the rooted graph obtained from $L$ by, for each $v\in T''$, contracting $R^L_{M_v,N_v}$ to $K_v$.
Note that Claim~\ref{cl-contrdens} implies that taken in isolation, none of these contractions decreases the 4-density;
however, the graphs $K_u$ and $K_v$ for distinct $u,v\in T''$ can share edges, and thus
simply applying Claim~\ref{cl-contrdens} to $u$ and $v$ separately would lead to overcounting.
Hence, we need a more careful analysis.

If $T''=\emptyset$, then $L_1=L$ and $\rho_4(L_1)=\rho_4(L)\ge \rho_4(G)$ by (\ref{eq-prof}).  Otherwise, let
us fix a vertex $u\in T''$ such that $\rho_4(R^L_{M_u, N_u})$ is largest among all vertices of $T''$.
Note that $T''\subseteq T'\subseteq T\setminus X_G$.
By Claim~\ref{cl-contrdens} and (\ref{eq-prof}), it follows that
\begin{align*}
\rho_4(L_1)&\ge \rho_4(L)+k^L(M_u\cap N_u)-\sum_{v\in T''} \rho_4(R^L_{M_v,N_v})\\
&\ge \rho_4(L)+\rho_4(R_{M_u,N_u})-\sum_{v\in T''\setminus \{u\}} \rho_4(R^L_{M_v,N_v})\\
&\ge\begin{cases}
\rho_4(L)&\text{ if $|T''|\le 2$}\\
\rho_4(L)-2&\text{ if $|T''|=3$}
\end{cases}\\
&\ge \rho_4(G).
\end{align*}
Let us now consider any root separation $(C,D)$ of $L_1$ of order at most four.
For any vertex $v\in T''$, since $K_v$ has non-edges only among the roots, we have $V(K_v)\subseteq C$ or $V(K_v)\subseteq D$.
Let
\begin{align*}
C'&=C\cup \bigcup_{v\in T'':V(K_v)\subseteq C} N_v\text{ and}\\
D'&=D\cup \bigcup_{v\in T'':V(K_v)\not\subseteq C} N_v;
\end{align*}
then $(C',D')$ is a root separation of $L$ satisfying $C'\cap D'=C\cap D$.
We claim that $T''\subseteq C'$.  Indeed, suppose for a contradiction that there exists a vertex $v\in T''\setminus C'$.
Then Observation~\ref{obs-splitcut} implies that $|C'\cap D'|=|C\cap D|=4$.
Moreover, since the $T$-partition $(M,\{N_u:u\in T\})$ $S$-captures root $(\le\!4)$-separations of $L$, we have $C\cap D=C'\cap D'\subseteq N_v$.
In the construction of $L_1$, we have contracted $R^L_{M_v, N_v}$ to its $\id$-rooted minor $K_v$, and thus $V(L_1)\cap N_v=V(K_v)$ has size four.
It follows that $C\cap D=V(K_v)$.  However, in that case $V(K_v)\subseteq C$ and $v\in C'$ by the choice of $C'$, which is a contradiction.

Since $T''\subseteq C'$, we have $D=D'$, and thus the root separations $(C,D)$ of $L_1$ and $(C',D')$ of $L$ have the same strictly right-hand side.
Moreover, Claim~\ref{cl-coodi} and the choice of $T''$ imply that the root separation $(C',D')$ is 4-light, and thus so is $(C,D)$.
Therefore, the $\id$-rooted minor $L_1$ of $G$ is $4$-light.
Since $\rho_4(L_1)\ge\rho_4(G)$, this contradicts Observation~\ref{obs-lightminor}.
Therefore, the $(G,A\cap B,T)$-split $H$ cannot be an $\id$-rooted minor of $R^G_{A,B}$.
\end{proof}

Let us now give several applications of this lemma.
First, let us further restrict root separations of order four, getting rid of $1$-nearly-saturated ones.
\begin{corollary}\label{cor-no4satur}
Let $G$ be a minimal counterexample and let $(A,B)$ be a proper root separation of $G$.
If $|A\cap B|=4$, then $(A,B)$ is not $1$-nearly-saturated.
\end{corollary}
\begin{proof}
Suppose for a contradiction that there exists a $1$-nearly-saturated proper root separation of $G$ of order four,
and let $(A,B)$ be one with $A$ minimal.  Observation~\ref{obs-starsat} implies that $(A,B)$ is linked to the roots
and that all isolators of $(A,B)$ are $1$-nearly-saturated.  By the minimality of $A$, it follows that $(A,B)$ is strongly linked to the roots.
By Corollary~\ref{cor-noplus1}, we have $|A\cap B\cap X_G|\le 2$.
Since $(A,B)$ is $1$-nearly-saturated, there exists a set $T=\{u,v\}\subseteq A\cap B$ of size two
such that the $(G,A\cap B,T)$-split $H$ is an $\id$-rooted minor of $R_{A,B}$.
By Lemma~\ref{lemma-nosatur}, we have $T\not\subseteq X_G$.

We need to show that the root separation $(A,B)$ is right-linked.  Since $H=K^G_{A\cap B}-uv$ is an $\id$-rooted
minor of $R_{A,B}$, it suffices to show that $R_{A,B}-(A\cap B\setminus\{u,v\})$ contains a path from $u$ to $v$.
If not, then we can divide $B\setminus A$ into two disjoint (not necessarily non-empty) parts $B_u$ and $B_v$ such that $G$ does not have any edge between $B_u$ and $B_v$,
$u$ does not have any neighbors in $B_v$, and $v$ does not have any neighbors in $B_u$.  Since the root separation $(A,B)$ is proper, we can
by symmetry assume that $B_v\neq\emptyset$.  Then $(A'_v,B'_v)=(A\cup B_u, (A\cap B\setminus\{u\})\cup B_v)$ is a proper root separation of $G$ of order three,
and Lemma~\ref{lemma-nosatur} implies that $H-u=K^G_{A'_v\cap B'_v}$ is not an $\id$-rooted minor of $R_{A'_v,B'_v}$.  Since $H$ is an $\id$-rooted minor of $R_{A,B}$,
it follows that $B_u\neq\emptyset$.  By Corollary~\ref{cor-3conn}, the rooted graph $G$ is essentially $4$-connected, and thus $B_u$ and $B_v$
each consist of exactly one vertex of degree three whose neighborhood is an independent set.  Therefore, $A\cap B$ is an independent set
and every $\id$-rooted minor of $R_{A,B}$ with vertex set $A\cap B$ has at most four edges.  In particular $|E(H)|=4$, and thus $|A\cap B\cap X_G|=2$.
Since $T=\{u,v\}\not\subseteq X_G$, we can by symmetry assume that $u\not\in X_G$.  But then $|A'_v\cap B'_v\setminus X_G|=1$, which contradicts Corollary~\ref{cor-noplus1}.
It follows that the root separation $(A,B)$ is right-linked.

Finally, since $G$ is 4-light, we have $\rho_4(R_{A,B})\le 0$, and since $|T|=2$, it follows that the $(G,A\cap B,T)$-split $H$ is $(A,B)$-profitable.
However, this contradicts Lemma~\ref{lemma-nocotri}.
\end{proof}

This has the following consequence for linkedness of separations.
\begin{corollary}\label{cor-cotri-linked}
Let $G$ be a minimal counterexample and let $(A,B)$ be a proper root separation of $G$ of order five.
If there exists a set $T\subseteq A\cap B$ of size three such that the $(G,A\cap B, T)$-split $H$ is an $\id$-rooted minor of $R_{A,B}$,
then $(A,B)$ is linked to the roots.
\end{corollary}
\begin{proof}
Let $(C,D)$ be an isolator for $(A,B)$, let $\PP$ be an isolator linkage for $(C,D)$ and $(A,B)$, and let $Y$ be the set of its terminators.
Let $H'$ be the graph obtained from $H$ by relabeling the terminator of each path of $\PP$ to its origin.  By contracting the paths of $\PP$,
we see that $H'$ is an $\id$-rooted minor of $R_{C,D}$.

Since $G$ is essentially 4-connected by Corollary~\ref{cor-3conn} and $A\cap B\subseteq D$ has size five, the root separation $(C,D)$
has order at least four.  Suppose for a contradiction that $|C\cap D|=4$, and let $v$ be the unique vertex of $V(H')\setminus (C\cap D)$.
By Corollary~\ref{cor-no4satur}, the root separation $(C,D)$ is not $1$-nearly-saturated, and thus $|E(K^G_{C\cap D})\setminus E(H'-v)|>1$.
It follows that $v\not\in T$.

Let $T=\{w_1,w_2,w_3\}$, and for $i\in\{1,2,3\}$, let $w'_i\in C\cap D$ be the vertex such that $P_{w'_i}$ ends in $w_i$.
Since $H$ is a $(G,A\cap B, T)$-split and $v\not\in X_G$, note that $w_i$ is adjacent to $v$ in $H$, and thus $w'_i$ is adjacent to $v$ in $H'$.
By contracting the edge $vw'_1$, we conclude that the graph $(H'-v)+\{w'_1w'_2,w'_1w'_3\}\supseteq K^G_{C\cap D}-w'_2w'_3$ is an $\id$-rooted minor of $R_{C,D}$.
It follows that the root separation $(C,D)$ is $1$-nearly-saturated, contradicting Corollary~\ref{cor-no4satur}.

Therefore, all isolators of $(A,B)$ have order five, and thus the root separation $(A,B)$ is linked to the roots.
\end{proof}

Next, we can get rid of nearly $K_{2,3}^+$-universal separations (and thus also of $K_{2,3}^+$-universal separations).

\begin{corollary}\label{cor-nonear-bistar}
Let $G$ be a minimal counterexample and let $(A,B)$ be a proper root separation of $G$ of order five.
Then $(A,B)$ is not nearly $K_{2,3}^+$-universal, and in particular $\rho_4(R_{A,B})\le 1$.
\end{corollary}
\begin{proof}
Suppose for a contradiction $G$ has such a nearly $K_{2,3}^+$-universal root separation,
and choose $(A,B)$ among them so that $A$ is minimal.  Let $y\in A\cap B$ be the exceptional root of $R_{A,B}$.
Let us choose a vertex $x\in A\cap B\setminus \{y\}$ arbitrarily, let $T=A\cap B\setminus \{x,y\}$,
and let $H$ be the $(G,A\cap B,T)$-split.  Since $H$ is a subgraph of the $\{x,y\}$-star and $y$ is the exceptional root of $R_{A,B}$,
it follows that $H$ is an $\id$-rooted minor of $R_{A,B}$.

By Corollary~\ref{cor-cotri-linked}, the root separation $(A,B)$ is linked to the roots.
Observation~\ref{obs-starsat} implies that every isolator of $(A,B)$ is nearly $K_{2,3}^+$-universal,
and by the minimality of $A$, it follows that $(A,B)$ is strongly linked to the roots.
Moreover, since the root separation $(A,B)$ is nearly $K_{2,3}^+$-universal, it is right-linked.

By Observation~\ref{obs-nearbc}, we have $|A\cap B\cap X_G|\le 2$ and $k^G(A\cap B)\ge 7$.
Moreover, by Corollary~\ref{cor-noK4plus}, we have $\rho_4(R_{A,B})\le 3$, and if $\rho_4(R_{A,B})=3$, then $k^G(A\cap B)\ge 8$.
It follows that $k^G(A\cap B)\ge\rho_4(R_{A,B})+5$.
Since $H$ contains all edges of $K^G_{A\cap B}$ except for those between the vertices of $T$, we have
$$|E(H)\setminus E(G)|\ge |E(K^G_{A\cap B})\setminus E(G)|-3=k^G(A\cap B)-3\ge \rho_4(R_{A,B})+2,$$
and thus the $(G,A\cap B,T)$-split $H$ is $(A,B)$-profitable.  This contradicts Lemma~\ref{lemma-nocotri}.

Therefore, $G$ does not have any nearly $K_{2,3}^+$-universal proper root separation;
and by Corollary~\ref{cor-noK4plus}, every proper root separation of $G$ of order five has 4-density at most one.
\end{proof}

Let us now turn our attention to proper root separations of order five whose right-hand side has 4-density one.
This right-hand side must satisfy one of the outcomes of Corollary~\ref{cor-nok4}(i),
Corollary~\ref{cor-nonear-bistar} excludes one of them.  Lemma~\ref{lemma-nocotri} also enables us to exclude
the one concerning vampire minors.

\begin{lemma}\label{lemma-vampire}
Let $G$ be a minimal counterexample and let $(A,B)$ be a proper root separation of $G$ of order five.  Then
$R_{A,B}$ does not contain any vampire $W$ with $X_W=A\cap B$ as an $\id$-rooted minor.
\end{lemma}
\begin{proof}
Suppose for a contradiction that $R_{A,B}$ contains such a vampire $W$ as an $\id$-rooted minor;
among the proper root separations $(A,B)$ with this property, select one with $A$ minimal.
Let $v_1x_1$ and $v_2x_2$ with $x_1,x_2\in A\cap B$ be the fangs of the vampire $W$,
let $T=A\cap B\setminus \{x_1,x_2\}=\{z_1,z_2,z_3\}$, and let $H$ be the $(G,A\cap B,T)$-split.
By contracting the edges $v_1x_1$ and $v_2x_2$ of $W$, we obtain a supergraph of $H$,
and thus $H$ is an $\id$-rooted minor of $R_{A,B}$.

By Corollary~\ref{cor-cotri-linked}, the root separation $(A,B)$ is linked to the roots.
Moreover, observe that if $(C,D)$ is an isolator of $(A,B)$, then by contracting
$R_{A,B}$ to $W$ and then contracting the paths of an isolator linkage for $(C,D)$ and $(A,B)$,
we obtain a vampire with roots $C\cap D$ as an $\id$-rooted minor
of $R_{C,D}$.  By the minimality of $A$, it follows that the root separation $(A,B)$ is strongly linked to the roots.
Let us also note that the presence of $W$ as an $\id$-rooted minor in $R_{A,B}$ implies that
the root separation $(A,B)$ is $K_{1,\star}$-universal and right-linked.
By Corollary~\ref{cor-noK4plus}, we have $|A\cap B\cap X_G|\le 2$.

Let $R=K^G_T$, let $t=|E(R)|=3-\binom{|T\cap X_G|}{2}$, and let $H'=K^G_{A\cap B}-E(G[A\cap B])$.  The $(G,A\cap B,T)$-split $H$ is equal to $K^G_{A\cap B}-E(R)$,
and thus
\begin{align}
|E(H)\setminus E(G)|&=|(E(H')\setminus E(G))\setminus E(R)|\nonumber\\
&=k^G(A\cap B)-|E(H')\cap E(R)|\label{eq-profit}\\
&\ge k^G(A\cap B)-t.\nonumber
\end{align}
By Corollary~\ref{cor-nonear-bistar}, we have $\rho_4(R_{A,B})\le 1$.
By Lemma~\ref{lemma-nocotri}, the $(G,A\cap B,T)$-split $H$ is not profitable, and thus $k^G(A\cap B)\le t+2\le 5$.

By Corollary~\ref{cor-neic2}, we have $k^G(A\cap B)\ge 4$, and if $k^G(A\cap B)=4$ then $H'$ is isomorphic to $C_4+K_1$.
Let us consider the case that $k^G(A\cap B)=4$, and let $K$ be the 4-cycle in $H'$.
Since $H$ is not profitable, (\ref{eq-profit}) gives $$2\ge |E(H)\setminus E(G)|=4-|E(K)\cap E(R)|.$$
Since $|E(K)\cap E(R)|\ge 2$, we have $T\subset V(K)$, and by symmetry, we can assume that $K=z_1z_2z_3x_1$.
However, we can then contract the edges $v_1x_1$ and $v_2z_2$ of $W$ to obtain $K^G_{A\cap B}$ as an $\id$-rooted minor of $R_{A,B}$,
contradicting Lemma~\ref{lemma-nosatur}.

Therefore, we have $k^G(A\cap B)=5$ and (\ref{eq-profit}) gives
$$2\ge |E(H)\setminus E(G)|=5-|E(H')\cap E(R)|,$$ and thus $|E(H')\cap E(R)|=3$.
That is, the graph $H'$ consists of the triangle $R$ and either two edges between
$T$ and $\{x_1,x_2\}$, or one such edge and the edge $x_1x_2$.
By symmetry, we can assume that $z_1x_1\in E(H')$; let $e$ be the edge of $E(H')\setminus E(R)$ distinct from $z_1x_1$.
By contracting the edges $v_1z_1$ and $v_2z_2$ in $W$, we see that $R_{A,B}$ contains $K^G_{A\cap B}-e$ as an $\id$-rooted minor.
However, this contradicts Corollary~\ref{cor-noK5minus}.
\end{proof}

We can now essentially eliminate root separations of order five and positive 4-density, with a single easy to handle exception.
\begin{lemma}\label{lemma-kill5}
Let $G$ be a minimal counterexample and let $(A,B)$ be a proper root separation of $G$ of order five.
If $\rho_4(R_{A,B})=1$, then there exists a vertex in $B\setminus A$ adjacent to all vertices of $A\cap B$.
\end{lemma}
\begin{proof}
The root separation $(A,B)$ is $K_{1,\star}$-universal by Corollary~\ref{cor-nok4}, and thus it is not nova-universal
by Corollary~\ref{cor-noK4plus}.  Hence, $R_{A,B}$ is not $(K_4+K_1)$-universal.
The root separation $(A,B)$ is not nearly $K_{2,3}^+$-universal by Corollary~\ref{cor-nonear-bistar},
and $R_{A,B}$ does not contain any vampire $W$ with $X_W=A\cap B$ as an $\id$-rooted minor by Lemma~\ref{lemma-vampire}.
By Corollary~\ref{cor-nok4}, it follows that there exists a vertex in $B\setminus A$ adjacent to all vertices of $A\cap B$.
\end{proof}

\section{Degrees of vertices}

The results of the previous section have the following useful consequence.
\begin{corollary}\label{cor-cancontr}
For every edge $e$ of a minimal counterexample $G$, the rooted graphs $G-e$ and $G/e$ are $4$-light.
\end{corollary}
\begin{proof}
Let $e=uv$, where $v\not\in X_G$, let $G_1=G-e$ and $G_2=G/e$, and suppose for a contradiction that
$G_1$ or $G_2$ has a non-4-light root separation of order at most four.  Let $(C,D)$ be such a root separation with $D$ minimal.

Let us first consider the case that $(C,D)$ is a root separation of $G_1$.
Note that $\{u,v\}\not\subseteq C$ and $\{u,v\}\not\subseteq D$, since otherwise $(C,D)$ would also be a non-4-light root separation of $G$.
Hence, we can assume that $u\in C\setminus D$ and $v\in D\setminus C$.  Let $D_1=D\cup \{u\}$ and consider the root separation $(C,D_1)$ of $G$
of order at most five.  Note that $R^G_{C,D_1}$ is obtained from $R^{G_1}_{C,D}$ by adding the root vertex $u$, possibly some edges between $u$ and other roots,
and the edge $e$.  Consequently $\rho_4(R^G_{C,D_1})=\rho_4(R^{G_1}_{C,D})+1\ge 2$.

Since the root separation $(C,D_1)$ of $G$ is not 4-light, it has order
exactly five.  Moreover, since $\rho_4(R^G_{C,D_1})\ge 2$, Corollary~\ref{cor-nonear-bistar} implies that the root separation $(C,D_1)$ is not proper.
Since $v\in D_1\setminus C$, this is only possible if $C=X_G$.
However then $u\in C$ is a root of $G$ and $v$ is its only non-root neighbor, contradicting Corollary~\ref{cor-twononroot}.

Therefore, we can assume that $(C,D)$ is a root separation of $G_2$.  We use the label $u$ for the vertex of $G_2$ created by the contraction of the edge $e$.
Note that $u\in D$, as otherwise $(C\cup \{v\},D)$ would be a root separation of $G$ with the same strictly right-hand side as the root separation $(C,D)$ of $G_2$,
a contradiction since $G$ is 4-light.  Let $C_2=C\cup \{v\}$ and $D_2=D\cup \{v\}$.

If $u\in C\cap D$, then $(C_2,D_2)$ is a root separation of $G$ of order at most five.  Let $m$ be the number of common neighbors of $u$ and $v$ in $G$
that belong to $D\setminus C$.  We have $\rho_4(R^G_{C_2,D_2})=\rho_4(R^{G_2}_{C,D})+m$, and Corollary~\ref{cor-nonear-bistar} gives $\rho_4(R^G_{C_2,D_2})\le 1$.  Since $\rho_4(R^{G_2}_{C,D})>0$,
it follows that $\rho_4(R^G_{C_2,D_2})=1$ and $m=0$.  However, this contradicts Lemma~\ref{lemma-kill5}.

Therefore, we have $u\in D\setminus C$.  By the minimality of $D$, the graph $R^{G_2}_{C,D}$ is 4-light.
Since $\rho_4(R^{G_2}_{C,D})>0$, Corollary~\ref{cor-k4} implies that 
either $R^{G_2}_{C,D}$ is $K_\star$-universal or $|C\cap D|=4$ and $R^{G_2}_{C,D}$ is $K^-_\star$-universal.
Note that $(C,D_2)$ is a root separation of $G$ with $C\cap D_2=C\cap D$ and that
$R^{G_2}_{C,D}$ is an $\id$-rooted minor of $R^G_{C,D_2}$.
Therefore, either $R^G_{C,D_2}$ is $K_\star$-universal or $|C\cap D_2|=4$ and $R^G_{C,D_2}$ is $K^-_\star$-universal.
Lemma~\ref{lemma-noclicut} implies that the latter is the case and $|D_2\setminus C|\le 2$.
It follows that $D_2\setminus C=\{u,v\}$ and $D\setminus C=\{u\}$.
However, then $\rho_4(R^{G_2}_{C,D})\le 0$, which is a contradiction.
\end{proof}

It follows that minimal counterexamples have the smallest possible number of edges.
\begin{corollary}\label{cor-nume}
If $G$ is a minimal counterexample and $F$ is a flaw of $G$, then $F$ does not belong to the $(\rho_4(G)-1)$-target,
and in particular $\rho_4(G)\le 7$.
\end{corollary}
\begin{proof}
Let $e=uv$ be any edge of $G$.  The 5-rooted graph $G-e$ is 4-light by Corollary~\ref{cor-cancontr}.
Since $n(G-e)=n(G)$ and $|E(G-e)|<|E(G)|$, $G-e$ is not a counterexample, and thus it is universal.
Since $F$ is not an $\id$-rooted minor of $G$, it is not an $\id$-rooted minor of $G-e$, either.
Therefore, $F$ does not belong to the target of $G-e$, i.e., to the $(\rho_4(G)-1)$-target.
\end{proof}

On the other hand, a similar argument using contraction shows that all non-root vertices of a minimal counterexample
must have dense (and somewhat large) neighborhoods.
Let $G$ be a rooted graph and let $G'$ be the graph obtained from $G$ by adding the edges of the clique on $X_G$.
For an edge $e=uv$ of $G$ with $v\not\in X_G$, let $t_G(e)$ be the number of triangles in $G'$ containing $e$.  In other words,
if $u\not\in X_G$, then $t_G(e)$ is the number of triangles in $G$ containing $e$, and if $u\in X_G$, then $t_G(e)$
is the number of such triangles plus $\deg^X_G v-1$.  When the rooted graph $G$ is clear from the context, we drop the subscript.

\begin{corollary}\label{cor-fourtri}
Every minimal counterexample $G$ is internally 4-connected, all non-root vertices of $G$ have degree at least five, and
every edge $e\in E(G)$ satisfies $t_G(e)\ge 4$.
\end{corollary}
\begin{proof}
Let $G'$ be the $\id$-rooted minor of $G$ obtained by contracting the edge $e$,
and observe that $\rho(G')=\rho(G)-t_G(e)-1$ and $n(G')=n(G)-1$, and thus
$$\rho_4(G')=\rho_4(G)+3-t_G(e).$$
Since $G'$ is $4$-light by Corollary~\ref{cor-cancontr} and $n(G')<n(G)$, Observation~\ref{obs-lightminor}
implies that $\rho_4(G')<\rho_4(G)$, and thus $t_G(e)\ge 4$.

The graph $G-X_G$ is connected by Corollary~\ref{cor-3conn}, and
Observation~\ref{obs-42dense} implies that $n(G)\ge 2$.  It follows that every vertex $v\in V(G)\setminus X_G$ has a non-root neighbor $u$,
and thus $v$ has degree at least $t_G(uv)+1\ge 5$.  Since $G$ is essentially $4$-connected by Corollary~\ref{cor-3conn},
it follows that the minimal counterexample $G$ is internally 4-connected.
\end{proof}

We are now going to focus on neighborhoods of low degree vertices of a minimal counterexample in more detail; let us set up notation for that.
For a fixed rooted graph $G$ and a vertex $v\in V(G)\setminus X_G$, let $N(v)$ be the set of neighbors of $v$ in $G$, let $N[v]=N(v)\cup\{v\}$,
and let $N^+_v$ be the graph obtained from $G[N(v)]$ by adding the edges of the clique on $X_G\cap N(v)$.
Note that each vertex $u\in N(v)$ satisfies $\deg_{N^+_v} u=t(uv)$.  Thus, by Corollary~\ref{cor-fourtri},
we have the following.
\begin{observation}\label{obs-deg4}
If $G$ is a minimal counterexample, then for every vertex $v\in V(G)\setminus X_G$, the graph $N^+_v$ has minimum degree at least four.
\end{observation}
The lower bound on the degrees of vertices from Corollary~\ref{cor-fourtri} is easy to improve.
\begin{lemma}\label{lemma-no5}
If $G$ is a minimal counterexample, then every vertex $v\in V(G)\setminus X_G$ has degree at least six.
\end{lemma}
\begin{proof}
By Corollary~\ref{cor-fourtri}, we have $\deg v\ge 5$.
Suppose that $\deg v=5$.  By Observation~\ref{obs-deg4}, the graph $N^+_v$ is isomorphic to $K_5$.
However, this implies that the root separation $(V(G)\setminus\{v\},N[v])$ is saturated, contradicting Lemma~\ref{lemma-nosatur}.
Therefore, we have $\deg v\ge 6$.
\end{proof}

To constrain vertices $v$ of larger degree, we use the following observation (applied to the root separation $(V(G)\setminus\{v\},N[v])$).
For a rooted graph $G$, a set $Z\subseteq V(G)$, and a graph $F$ with $V(F)=X_G$,
we let $K^{G,F}_Z=K^G_Z\cup F[Z\cap X_G]$.

\begin{lemma}\label{lemma-cantcreateclique}
Let $G$ be a minimal counterexample and let $F$ be a flaw of $G$.
For every root separation $(A,B)$ of $G$ of order at least five,
there exists a set $Z\subseteq A\cap B$ of size five such that $A\cap B\cap X_G\subseteq Z$ and the rooted graph $\roots{G[B]}{Z}$
does not contain $K^{G,F}_Z$ as an $\id$-rooted minor.  Moreover, for every $(A,B)$-exposed vertex $z\in A\cap B$, there exists such a set $Z$ containing $z$.
\end{lemma}
\begin{proof}
Let $\PP$ be a root linkage for $(A,B)$ and let $Z'$ be the set of its terminators; note that $A\cap B\cap X_G\subseteq Z'$.
Let $Z\subseteq A\cap B$ be any superset of $Z'$ of size five.
We claim that the rooted graph $\roots{G[B]}{Z}$ does not contain $K^{G,F}_Z$ as an $\id$-rooted minor.  Suppose for a contradiction
that it does.  We cannot have $Z=Z'$, as otherwise we could obtain a supergraph of $F$ as an $\id$-rooted minor of $G$ by contracting $G[B]$ to $K^{G,F}_Z$
and then contracting the paths of $\PP$.  Hence, an isolator $(C,D)$ of $(A,B)$ in $G$ has order $|Z'|<|Z|=5$.  However, by contracting the subpaths of $\PP$ between $C\cap D$ and $Z'$,
we see that $R_{C,D}$ contains $K^G_{C\cap D}$ as an $\id$-rooted minor, and thus the proper root separation $(C,D)$ of $G$ is saturated.
This contradicts Lemma~\ref{lemma-nosatur}.

Moreover, Observation~\ref{obs-exposed} implies that for each $(A,B)$-exposed vertex $z\in A\cap B$, we can choose
the root linkage $\PP$ so that $z\in Z'\subseteq Z$.
\end{proof}

As a quick aside, this has the following consequence.

\begin{corollary}\label{cor-omega4}
Every clique in a minimal counterexample has size at most four.
\end{corollary}
\begin{proof}
Suppose for a contradiction that a minimal counterexample $G$ contains a clique of size at least five with vertex set $K$,
and consider the root separation $(V(G),K)$ of $G$.  Since $G[K]$ is a clique,
the rooted graph $\roots{G[K]}{Z}$ contains $K^{G,F}_Z$ as an $\id$-rooted minor for every graph $F$ with vertex set $X_G$ and for every set $Z\subseteq K$.
This contradicts Lemma~\ref{lemma-cantcreateclique}.
\end{proof}

We cannot quite exclude the existence of a vertex $v$ of degree six, seven, or eight from a minimal counterexample,
but we can show that many of its neighbors must be roots.  To do so, we apply Lemma~\ref{lemma-cantcreateclique}.
It will be convenient to work in the following setting.  For a minimal counterexample $G$, a flaw $F$, and a vertex $v\in V(G)\setminus X_G$,
let $N^{\overline{F}}_v$ denote the graph $G[N(v)]\cup \overline{F}[X_G\cap N(v)]$, that is, the graph obtained from $N^+_v$ by deleting the edges of
the flaw $F$ between the vertices of $X_G\cap N(v)$.  We say that a set $Z\subseteq N(v)$ is \emph{$(F,v)$-poor} if $|Z|=5$, $X_G\cap N(v)\subseteq Z$,
and for every $u\in Z$, the rooted graph $\rootsrem{N^{\overline{F}}_v}{Z}{u}$ does not contain $K_4$ as a rooted minor.
We say that a vertex $z\in N(v)$ is \emph{$v$-exposed} if it is $(V(G)\setminus\{v\},N[v])$-exposed.
\begin{corollary}\label{cor-cantcreateclique}
Let $G$ be a minimal counterexample and let $F$ be a flaw of $G$.
For every vertex $v\in V(G)\setminus X_G$, there exists an $(F,v)$-poor set $Z\subseteq N(v)$.
Moreover, for every $v$-exposed vertex $z\in N(v)$, there exists such a set $Z$ containing $z$.
\end{corollary}
\begin{proof}
By Lemma~\ref{lemma-cantcreateclique} applied with the root separation $(V(G)\setminus\{v\},N[v])$, there exists a set
$Z\subseteq N(v)$ of size five such that $N(v)\cap X_G\subseteq Z$ and the rooted graph $\roots{G[N[v]]}{Z}$
does not contain $K^{G,F}_Z$ as an $\id$-rooted minor; and moreover, for a $v$-exposed vertex $z\in N(v)$,
there exists such a set $Z$ containing $z$.  We claim that $Z$ is $(F,v)$-poor.

Indeed, consider any vertex $u\in Z$, and suppose for a contradiction
that the rooted graph $\rootsrem{N^{\overline{F}}_v}{Z}{u}$ contains $K_4$ as a rooted minor.
Equivalently, the rooted graph $\rootsrem{G[N(v)]}{Z}{u}$ contains $K^{G,F}_{Z\setminus\{u\}}$
as an $\id$-rooted minor.  By contracting $G[N(v)]-u$ to $K^{G,F}_{Z\setminus\{u\}}$ and then contracting the edge $uv$,
we see that the rooted graph $\roots{G[N[v]]}{Z}$ contains a supergraph of $K^{G,F}_Z$ as an $\id$-rooted minor, which is a contradiction.
\end{proof}

Recall that $\deg^X v$ denotes the number of roots adjacent to $v$ and that $\deg^+ v=\deg v+\deg^X v$.

\begin{lemma}\label{lemma-bo6}
Let $G$ be a minimal counterexample.  If $V(G)\setminus X_G$ contains a vertex of degree six, then $\rho_4(G)\le 5$.
Moreover, every vertex $v\in V(G)\setminus X_G$ of degree six satisfies $\deg^+ v\ge 10$.
\end{lemma}
\begin{proof}
Suppose for a contradiction that there exists a vertex $v\in V(G)\setminus X_G$ of degree six
such that $\deg^+ v\le 9$, or equivalently, $|X_G\cap N(v)|\le 3$.
Let $F$ be a flaw of $G$ and let $H=N^{\overline{F}}_v$.

\begin{claim}\label{cl-onlyone}
The graph $H$ has minimum degree at least two, and every vertex $u'$ of $H$ of degree at most three belongs to $X_G\cap N(v)$
and has a non-neighbor belonging to $X_G\cap N(v)$ (adjacent to $u'$ in $F$).  Moreover, $H$ has at most one vertex of degree two, and if it has such a vertex $u$, then
$|X_G\cap N(v)|=3$, $u\in X_G$, and $X_G\cap N(v)$ is an independent set in $H$ (a clique in $F$).
\end{claim}
\begin{subproof}
By Observation~\ref{obs-deg4}, the graph $N^+_v$ has minimum degree at least four, and thus each vertex of $H$ of degree at most three belongs to $X_G\cap N(v)$
and has a non-neighbor in $X_G\cap N(v)$.

Let us consider any vertex $u$ of $H$ of degree at most two.
Since $u$ has degree at least four in $N^+_v$, it follows that $u\in X_G$, $|X_G\cap N(v)|=3$, and $u$ has degree exactly four in $N^+_v$ and exactly two in $H$.
Hence, we have $t(uv)=4$ and $u$ is adjacent in $F$ to the two vertices $w_1$ and $w_2$ of $X_G\cap N(v)\setminus \{u\}$.

Let us now consider the $\id$-rooted minor $G'$ of $G$ obtained by contracting the edge $uv$, and let $F'=F-\{uw_1,uw_2\}$.
The rooted graph $G'$ is 4-light by Corollary~\ref{cor-cancontr}, and thus Observation~\ref{obs-nomg} implies that $F'$ does not belong to the target of $G'$.
Note that $\rho_4(G')=\rho_4(G)+3-t(uv)=\rho_4(G)-1$, and thus by Observation~\ref{obs-trade}, we see that $F'$ is isomorphic to $K_2+K_3$.  Consequently either
\begin{itemize}
\item $u$ belongs to the triangle of $F'$, $w_1$ and $w_2$ are the two vertices of the non-triangle component of $F'$, and $u$ is the unique vertex of $F$ of degree four, or
\item $u$ belongs to the non-triangle component of $F'$, $w_1$ and $w_2$ both belong to the triangle component of $F'$, $\Delta(F)=3$, and $u$ is the unique vertex of $F$
adjacent to a pendant vertex.
\end{itemize}
In either case, $X_G\cap N(v)$ induces the triangle $uw_1w_2$ in $F$.  Moreover, the vertex $u$ is uniquely determined
by the graph $F$, and consequently, only one vertex of $H$ can have degree two.
\end{subproof}

Claim~\ref{cl-onlyone} implies that $\Delta(\overline{H})\le 3$, $\overline{H}$ has at most one vertex of degree three (and if it has one, this vertex and two of its neighbors belong to $X_G\cap V(H)$),
and each vertex of $\overline{H}$ of degree two belongs to $X_G\cap V(H)$ and has a neighbor in $X_G\cap V(H)$.
Let us choose a vertex $z\in V(H)\setminus X_G$ as follows:
\begin{itemize}
\item If $\overline{H}$ has a (unique) vertex $u$ of degree three, then let $z$ be the neighbor of $u$ in $\overline{H}$ that does not belong to $X_G$.
\item If $\Delta(\overline{H})\le 2$ and $|X_G\cap V(H)|=3$, then let $z$ be an isolated vertex of $\overline{H}-X_G$ (there exists one, since $\Delta(\overline{H}-X_G)\le 1$).
\item Otherwise, we let $z$ be an arbitrary vertex in $V(H)\setminus X_G$.
\end{itemize}
We claim that the vertex $z$ is $v$-exposed.  Indeed, otherwise let $K$ be the vertex set of the component of $G-(N(v)\setminus\{z\})$
containing $z$ (and $v$) and consider the root separation $(C,D)=(V(G)\setminus K,N(v)\cup K)$ of order five, where $C\cap D=N(v)\setminus\{z\}$.
Since each vertex $w\in N(v)\setminus X_G$ has at most one non-neighbor in $N(v)$, we have $k^G(C\cap D)\le 2$.
Observe that since the vertex $v\in D\setminus C$ is adjacent to all vertices of $C\cap D$, it follows that the root separation $(C,D)$ is $1$-nearly-saturated
and $K_{1,\star}$-universal, contradicting Corollary~\ref{cor-noK5minus}.

By Corollary~\ref{cor-cantcreateclique}, there exists an $(F,v)$-poor set $Z\subseteq V(H)$ containing $z$;
that is, there is no vertex $u\in Z$ such that $K_4$ is a rooted minor of the rooted graph $\rootsrem{H}{Z}{u}$.
Let $w$ be the vertex in $V(H)\setminus Z$ and let us consider the possibilities for the graph $\overline{H}[Z]$.
\begin{itemize}
\item Suppose first that $\overline{H}[Z]$ contains a triangle, necessarily on the vertices of $X_G\cap V(H)$.
Let $y$ be the vertex in $Z\setminus (X_G\cup \{z\})$.  Since $|X_G\cap V(H)|=3$, we have chosen $z$ so that
all its neighbors in $\overline{H}$ belong to $X_G\cap V(H)$.
Moreover, if $\overline{H}$ has a vertex of degree three, it is the neighbor $u$ of $z$ in $X_G\cap V(H)$,
and $uz$ is the only edge of $\overline{H}$ between $X_G\cap V(H)$ and $\{w,y,z\}$.
If $\Delta(\overline{H})=2$, then the triangle $\overline{H}[X_G\cap V(H)]$ is a component of $\overline{H}$;
in this case, let $u$ be an arbitrary vertex in $X_G\cap V(H)$.
Let $u_1$ and $u_2$ be the two vertices of $X_G\cap V(H)\setminus\{u\}$.

Thus, $E(\overline{H})$ consists of the edges of the triangle on $X_G\cap V(H)$ and possibly of the edges $uz$ and $yw$.
However, then the only non-edges of the graph $H-u$ are $u_1u_2$ and possibly $yw$,
and thus by contracting the edge $wu_1$, we obtain $K_4$ as a rooted minor of $\rootsrem{H}{Z}{u}$.
This is a contradiction, since $Z$ is $(F,v)$-poor.

Therefore, we can assume that $\overline{H}$ does not contain a triangle.  By Claim~\ref{cl-onlyone}, this implies
that $\Delta(\overline{H})\le 2$.  Moreover, since all vertices in $V(H)\setminus X_G$ have degree at most one in $\overline{H}$,
it follows that $\overline{H}$ is a forest.
\item Suppose now that $\overline{H}[X_G\cap V(H)]$ is a path $u_1u_2u_3$.
Recall that $z$ was chosen as an isolated vertex of $\overline{H}-X_G$.
Let $y$ be the vertex in $Z\setminus (X_G\cup \{z\})$.  Since $y$ and $z$ have degree at most one in $\overline{H}$,
and since $u_1$ and $u_3$ have degree at most two in $\overline{H}$ and are adjacent to $u_2$,
we can choose the labels of $u_1$ and $u_3$ so that $yu_3, zu_1\not\in E(\overline{H})$.
Hence, the non-edges of the graph $H[Z]$ are $u_1u_2$, $u_2u_3$, and possibly $yu_1$ and $zu_3$.
\begin{itemize}
\item If $yu_1\not\in E(H)$, then since $\Delta(\overline{H})\le 2$ and $y$ has degree at most one in $\overline{H}$,
it follows that $w$ is adjacent in $H$ to $y$, $u_1$, and $u_2$.  In this case, we can contract the edge $u_1w$
and obtain $K_4$ as a rooted minor of $\rootsrem{H}{Z}{u_3}$.
\item If $zu_3\not\in E(H)$, then similarly $K_4$ is a rooted minor of $\rootsrem{H}{Z}{u_1}$.
\item If $yu_1,zu_3\in E(H)$, then $Z\setminus \{u_2\}$ is a clique of size four in $H[Z]-u_2$.
\end{itemize}
In either case, there exists a vertex $u\in Z$ such that $K_4$ is a rooted minor of $\rootsrem{H}{Z}{u}$,
which is a contradiction.

Therefore, $\overline{H}[X_G\cap V(H)]$ has at most one edge, and by Claim~\ref{cl-onlyone},
$\overline{H}$ has at most two vertices of degree two.
\item Let us next consider the possibility that $\overline{H}[Z]$ contains an isolated vertex $x$,
and thus $\overline{H}[Z]-x$ is a subgraph of a path $u_1u_2u_3u_4$.
Since $w$ has degree at most one in $\overline{H}$, we can assume that the labels of $u_1$, \ldots, $u_4$
are chosen so that $wu_1,wu_2\not\in E(\overline{H})$.  But then we can contract the edge $wu_1$ in $H$ and
obtain $K_4$ as a rooted minor of $\rootsrem{H}{Z}{u_3}$, which is a contradiction.

Therefore, $\delta(\overline{H}[Z])\ge 1$; moreover, since $|Z|=5$, it follows that $\overline{H}[Z]$ has a vertex of degree two,
and since $\overline{H}[X_G\cap V(H)]$ has at most one edge, Claim~\ref{cl-onlyone} implies that one of its neighbors belongs to $X_G$ and the other one does not.
\item Let $u_1u_2y$ be a path in $\overline{H}[Z]$, where $u_1,u_2\in X_G$ and $y\not\in X_G$, and let $Z\setminus\{u_1,u_2,y\}=\{y_1,y_2\}$.
Since $\delta(\overline{H}[Z])\ge 1$, $\Delta(\overline{H}[Z])=2$, and the vertices of $V(H)\setminus X_G$ have degree at most one in $\overline{H}$,
it follows that $\overline{H}[Z]$ is the disjoint union of the paths $u_1u_2y$ and $y_1y_2$.
Moreover, Claim~\ref{cl-onlyone} implies that $y_1$ and $y_2$ have degree one in $\overline{H}$ (even if one of them belongs to $X_G$, since
they are non-adjacent to $u_1,u_2\in X_G$).  Therefore, $y_1w,y_2w\not\in E(\overline{H})$.
By contracting the edge $wy_1$ of $H$, we obtain $K_4$ as a rooted minor of $\rootsrem{H}{Z}{u_2}$, which is a contradiction.
\end{itemize}
In all cases, we have obtained a contradiction.  Therefore, we have $\deg^+ v\ge 10$, and thus $|N(v)\cap X_G|\ge 4$.

Suppose now for a contradiction that there exists a vertex $v\in V(G)\setminus X_G$ of degree six
and that $\rho_4(G)\ge 6$.  Let $F$ be a flaw of $G$; by Corollary~\ref{cor-nume}, $F$ does not belong to the $5$-target $\SS_{5,8}$,
and thus $F$ has at most one non-edge.  Since $|N(v)\cap X_G|\ge 4$,
there exists a vertex $u\in N(v)\cap X_G$ adjacent in $F$ to all vertices of $N(v)\cap X_G\setminus\{u\}$.
Let $G'$ be the $\id$-rooted minor of $G$ obtained by contracting the edge $uv$ and let $F'$ be the subgraph of $F$
obtained by deleting all (at least three) edges from $u$ to $N(v)\cap X_G\setminus\{u\}$.
The rooted graph $G'$ is $4$-light by Corollary~\ref{cor-cancontr}, and thus Observation~\ref{obs-nomg} implies that $F'$ does not belong to the target of $G'$.
Moreover, $\rho_4(G')=\rho_4(G)+3-t(uv)\ge \rho_4(G)-2$.  Since $\rho_4(G)\ge 6$, this contradicts Observation~\ref{obs-trade} or \ref{obs-trade2}.
\end{proof}

Next, we consider the vertices of degree seven.

\begin{lemma}\label{lemma-bo7}
If $G$ is a minimal counterexample, then every vertex $v\in V(G)\setminus X_G$ of degree seven satisfies $\deg^+ v\ge 10$.
Moreover, if there exists a vertex $v\in V(G)\setminus X_G$ of degree seven such that $\deg^+ v=10$, then $\rho_4(G)\le 5$.
\end{lemma}
\begin{proof}
Let $v\in V(G)\setminus X_G$ be a vertex of degree seven, and suppose for a contradiction that $\deg^+ v\le 9$, or equivalently $|X_G\cap N(v)|\le 2$.
Let $F$ be a flaw of $G$ and let $H=N^{\overline{F}}_v$.  Since $\delta(N^+_v)\ge 4$, we have $\delta(H)\ge 3$, and if
$H$ has a vertex $x$ of degree three, then $|X_G\cap N(v)|=2$, $X_G\cap N(v)$ is an independent set in $H$, and $x\in X_G\cap N(v)$.
In particular, $H$ has at most two vertices of degree three.  Let us restate this observation from the perspective of the complement $\overline{H}$ of $H$.
\begin{claim}\label{cl-two3}
The graph $\overline{H}$ has maximum degree at most three and at most two of its vertices have degree exactly three.
Each vertex of $\overline{H}$ of degree three belongs to $X_G$ and one of its neighbors in $\overline{H}$ belongs to $X_G$ as well.
Moreover, if $\overline{H}$ has two vertices of degree three, then they are adjacent in $\overline{H}$.
\end{claim}

By Corollary~\ref{cor-cantcreateclique}, there exists an $(F,v)$-poor set $Z\subseteq V(H)$; let $s_1$ and $s_2$ be the two vertices in $V(H)\setminus Z$.
Note that $s_1,s_2\not\in X_G$, and thus $s_1$ and $s_2$ have degree at most two in $\overline{H}$.
Let us now consider the possibilities with regards to the graph $\overline{H}[Z]$.
\begin{itemize}
\item Suppose first that $\overline{H}[Z]$ has two vertices $u_1$ and $u_2$ of degree three.  
Let $z$ be a common neighbor of $u_1$ and $u_2$ in $\overline{H}[Z]$, and for $i\in\{1,2\}$, let $z_i$ be the neighbor of $u_i$ in $\overline{H}[Z]$ different from
$z$ and $u_{3-i}$ (it is possible that $z_1=z_2$).  By Claim~\ref{cl-two3},
there are no edges of $\overline{H}$ between $\{s_1,s_2\}$ and $\{u_1,u_2,z\}$.

If $z_1=z_2$, then the vertex in $Z\setminus \{u_1,u_2,z,z_1\}$ is isolated in $\overline{H}[Z]$ and $z_1z,s_1z_1,s_2z_1\not\in E(\overline{H})$.
Hence, by contracting the edges $s_1u_1$ and $s_2u_2$, we see that $K_5$ is a rooted minor of $\roots{H}{Z}$;
this is a contradiction, since $Z$ is $(F,v)$-poor.

Hence, suppose that $z_1\neq z_2$.  Since $z_1$ has degree at most two in $\overline{H}$, we can by symmetry between $s_1$ and $s_2$
assume that $s_1z_1\not\in E(\overline{H})$.  Hence, by contracting the edges $s_1u_1$ and $s_2u_2$, we see that
$K_4$ is a rooted minor of $\rootsrem{H}{Z}{z_2}$.  This is again a contradiction.

Therefore, $\overline{H}[Z]$ has at most one vertex of degree three.

\item Suppose next that $\overline{H}[Z]$ has exactly one vertex $u_1$ of degree three, and let $u_2$ be its neighbor in $X_G$
and $z_1$ and $z_2$ its neighbors not in $X_G$. Let $z$ be the vertex in $Z\setminus \{u_1,u_2,z_1,z_2\}$.

\begin{itemize}
\item If $z_1z_2\in E(\overline{H})$, then $z$ can only be adjacent to $u_2$ in $\overline{H}[Z]$, and $\overline{H}$ has no edges between $\{u_1,z_1,z_2\}$
and $\{s_1,s_2\}$.  By contracting the edges $s_1z_1$ and $s_2z_2$ in $H$, we obtain $K_4$ as a rooted minor of $\rootsrem{H}{Z}{u_2}$.
This is a contradiction, and thus $z_1z_2\not\in E(\overline{H})$.
\item If $z_1u_2\in E(\overline{H})$, then $z$ can only be adjacent to $z_2$ in $\overline{H}[Z]$.  Moreover, since $u_2$ has degree at most three in $\overline{H}$,
we can by symmetry between $s_1$ and $s_2$ assume that $s_2u_2\not\in E(\overline{H})$.
By contracting the edges $s_1u_1$ and $s_2u_2$ in $H$, we obtain $K_4$ as a rooted minor of $\rootsrem{H}{Z}{z_2}$.

This is a contradiction, and thus $z_1u_2\not\in E(\overline{H})$.  By symmetry between $z_1$ and $z_2$, we also have $z_2u_2\not\in E(\overline{H})$.
Moreover, note that neither $z_1$ nor $z_2$ can be adjacent to both $s_1$ and $s_2$ in $\overline{H}$.
\item If $u_2$ is not adjacent to both $s_1$ and $s_2$ in $\overline{H}$,
then we can contract the edges $s_1u_1$ and $s_2u_1$ in $H$ and obtain $K_4$ as a rooted minor of $\rootsrem{H}{Z}{z}$.
This is a contradiction, and thus $s_1u_2,s_2u_2\in E(\overline{H})$.  This implies that $z$ is not adjacent to $u_1$ and $u_2$ in $\overline{H}$.
\item If $s_1z,s_2z\not\in E(\overline{H})$, then by contracting the edges $s_1z$ and $s_2z$ in $H$, we obtain $K_4$ as a rooted minor of $\rootsrem{H}{Z}{u_1}$.
This is a contradiction, and thus by symmetry between $s_1$ and $s_2$, we can assume that $s_1z\in E(\overline{H})$.  Since $s_1$ has degree at most two in $\overline{H}$,
it follows that $s_1z_1,s_1z_2\not\in E(\overline{H})$.  Moreover, since $z$ has degree at most two in $\overline{H}$,
we can by symmetry between $z_1$ and $z_2$ assume that $z_1z\not\in E(\overline{H})$.
\item If $z_2z\not\in E(\overline{H})$, then $H[Z-u_1]$ is isomorphic to $K_4$.  If $z_2z\in E(\overline{H})$, then since $z$ and $z_2$ have degree at most two in $\overline{H}$,
it follows that $s_2$ is non-adjacent to $z$ and $z_2$ in $\overline{H}$, and thus by contracting the edge $s_2z$, we obtain $K_4$
as a rooted minor of $\rootsrem{H}{Z}{u_1}$.  In either case, this contradicts the $(F,v)$-poorness of $Z$.
\end{itemize}

In each of the subcases, we obtained a contradiction, and thus $\Delta(\overline{H}[Z])\le 2$.

\item Suppose next that $\overline{H}[Z]$ contains a cycle $C=u_1\ldots u_k$ for some $k\in \{3,4,5\}$, and let $Z\setminus V(C)=\{z_1,\ldots,z_{5-k}\}$.
Since $\Delta(\overline{H}[Z])\le 2$, the graph $\overline{H}$ does not have any edges between $V(C)$ and $Z\setminus V(C)$;
hence, $E(\overline{H}[Z])$ consists of $E(C)$ and possibly the edge $z_1z_2$ when $k=3$.

By Claim~\ref{cl-two3}, we can choose the labels of the vertices of $C$ so that $\overline{H}$ does not have any edges between $\{s_1,s_2\}$ and $V(C)\setminus\{u_1,u_2\}$.
Moreover, each of $u_1$ and $u_2$ has at most one neighbor in $\{s_1,s_2\}$, and thus by symmetry between $s_1$ and $s_2$, we can furthermore assume that the set $R$ of edges of $\overline{H}$ between $\{u_1,u_2\}$ and $\{s_1,s_2\}$
is either equal to $R_1=\{u_1s_1,u_2s_1\}$ or a subset of $R_2=\{u_1s_2,u_2s_1\}$.
If $k\ge 4$, then by contracting the edges $s_1u_k$ and $s_2u_3$, we see that $\rootsrem{H}{Z}{u_1}$ contains $K_4$ as a rooted minor.
This is a contradiction, and thus $k=3$.

If $s_1z_1,s_1z_2\not\in E(\overline{H})$, then we can contract the edges $s_1z_1$ and $s_2u_2$ in $H$
and obtain $K_4$ as a rooted minor of $\rootsrem{H}{Z}{u_1}$.  This is a contradiction,
and thus $s_1$ has a neighbor in $\{z_1,z_2\}$ in $\overline{H}$.
Since $s_1$ has degree at most two in $\overline{H}$, it follows that $R\neq R_1$,
and thus $R\subseteq R_2$; and moreover, $s_1$ has at most one neighbor in $\{u_2,s_2\}$ in $\overline{H}$.
By contracting the edges $s_1u_1$ and $s_2u_2$ in $H$, we see that $K_4$ is a rooted minor of $\rootsrem{H}{Z}{z_1}$.
This is a contradiction, since $Z$ is $(F,v)$-poor.

Therefore, the graph $\overline{H}[Z]$ is a forest.

\item Suppose now that there exists a vertex $u\in Z$ which is adjacent to both $s_1$ and $s_2$ in $\overline{H}$ and is not isolated in $\overline{H}[Z]$.
By Claim~\ref{cl-two3}, $u$ has degree exactly three in $\overline{H}$, belongs to $X_G$, and its unique neighbor $u'$ in $\overline{H}[Z]$
also belongs to $X_G$.  Since $\overline{H}[Z]$ is a forest of maximum degree at most two, $\overline{H}[Z\setminus \{u,u'\}]$ is a subgraph
of a path $z_1z_2z_3$.  If $z_1z_2,z_2z_3\not\in E(\overline{H})$, then $H[Z\setminus \{u'\}]$ is a clique,
in contradiction to $Z$ being $(F,v)$-poor.  Hence, we can choose the labels of the vertices of $\overline{H}[Z\setminus \{u,u'\}]$
so that $z_1z_2\in E(\overline{H})$.

Since $|X_G\cap N(v)|\le 2$, we have $z_1,z_2,z_3\not\in X_G$. Hence, 
Claim~\ref{cl-two3} implies that neither $z_1$ nor $z_2$ can be adjacent to both $s_1$ and $s_2$ in $\overline{H}$,
and if $z_3$ is adjacent to both $s_1$ and $s_2$ in $\overline{H}$, then $z_2z_3\not\in E(\overline{H})$.
If neither $s_1$ nor $s_2$ is adjacent to $z_2$ in $\overline{H}$, it follows that contracting the edges $s_1z_2$ and $s_2z_2$ of $H$
ensures that both $z_1$ and $z_3$ are adjacent to $z_2$ in the resulting minor.  This would imply that $K_4$ is a rooted minor of $\rootsrem{H}{Z}{u'}$,
a contradiction.  Therefore, we can assume that $s_1z_2\in E(\overline{H})$ and $s_2z_2\not\in E(\overline{H})$.

Since $s_1$ and $z_2$ have degree at most two in $\overline{H}$, this implies that $s_1z_1,z_2z_3\not\in E(\overline{H})$.
Moreover, since $s_2$ has degree at most two in $\overline{H}$, it cannot be adjacent to both $z_1$ and $s_1$ in $\overline{H}$.
Hence, by contracting the edges $s_1z_1$ and $s_2z_2$ of $H$, we see that $K_4$ is a rooted minor of $\rootsrem{H}{Z}{u'}$,
which is a contradiction.

Therefore, we can assume that every vertex adjacent to both $s_1$ and $s_2$ in $\overline{H}$ is isolated in $\overline{H}[Z]$.

\item Not all edges of $\overline{H}[Z]$ are incident with the same vertex, as otherwise $H[Z]$ would contain a clique of size
four, contradicting the $(F,v)$-poorness of $Z$.  Hence, the forest $\overline{H}[Z]$ has two non-adjacent leaves
without a common neighbor in $\overline{H}[Z]$.  Since $\overline{H}[Z]$ is a forest of maximum degree at most two and $|Z|=5$,
it follows that $\overline{H}[Z]$ is isomorphic to $P_5$ or $P_4+K_1$ or $P_3+K_2$ or $2K_2+K_1$.  Hence, we can label the vertices
of $Z$ as $z_1$, \ldots, $z_5$ so that for some $k\in\{4,5\}$, we have $E(\overline{H}[Z])\subseteq \{z_1z_2,z_2z_3,\ldots,z_{k-1}z_k\}$
and $z_1z_2,z_{k-1}z_k\in E(\overline{H}[Z])$.

Suppose first that $k=4$.  Since the vertices $z_1$, \ldots, $z_4$ are not isolated in $\overline{H}[Z]$, each of them is adjacent
to at most one of $s_1$ and $s_2$ in $\overline{H}$.  By symmetry, we can assume that $z_1s_1\not\in E(\overline{H})$.
If $z_2s_1\not\in E(\overline{H})$, then by contracting the edge $z_1s_1$ of $H$, we see that $K_4$ is a rooted minor of $\rootsrem{H}{Z}{z_3}$,
which is a contradiction.  Therefore, $z_2s_1\in E(\overline{H})$, and thus $z_2s_2\not\in E(\overline{H})$.

If $s_2$ is not adjacent to both $z_1$ and $s_1$ in $\overline{H}$, then by contracting the edges $z_1s_1$ and $z_2s_2$ of $H$,
we see that $K_4$ is a rooted minor of $\rootsrem{H}{Z}{z_3}$.  On the other hand, if $s_1s_2,z_1s_2\in E(\overline{H})$,
then since $s_2$ has degree at most two in $\overline{H}$, it is not adjacent to $z_3$ and $z_4$ in $\overline{H}$.
In this case, we can contract the edge $s_2z_3$ of $H$ and obtain $K_4$ as a rooted minor of $\rootsrem{H}{Z}{z_2}$.
This is again a contradiction.

\item It follows that it is not possible to relabel the vertices of $Z$ so that $k=4$, and thus $\overline{H}[Z]$ is isomorphic either to $P_5$ or to $P_3+K_2$.
Hence, we can assume that $k=5$ and $z_2z_3\in E(\overline{H})$.  In particular, no vertex of $\overline{H}[Z]$ is isolated, and thus each vertex of $Z$ has at most one neighbor in $\{s_1,s_2\}$
in $\overline{H}$.

If there exists $i\in\{2,4\}$ such that neither $s_1$ nor $s_2$ is adjacent to $z_i$ in $\overline{H}$, then contracting the edges $s_1z_i$ and $s_2z_i$ of $H$
makes $z_i$ adjacent to both $z_{i-1}$ and $z_{i+1}$, and thus $K_4$ is a rooted minor of $\rootsrem{H}{Z}{z_{6-i}}$.  This is a contradiction,
and thus both $z_2$ and $z_4$ have a neighbor in $\{s_1,s_2\}$ in $\overline{H}$.  In particular, $z_2$ has degree three in $\overline{H}$,
and by Claim~\ref{cl-two3} it belongs to $X_G$ and has a neighbor in $X_G$.  Since $z_2z_4\not\in E(\overline{H})$ and $|X_G\cap N(v)|\le 2$,
the vertex $z_4$ does not belong to $X_G$, and thus it has degree at most two in $\overline{H}$.  Therefore, $z_3z_4\not\in E(\overline{H})$ and $\overline{H}[Z]$ is isomorphic to $P_3+K_2$.

By symmetry between $s_1$ and $s_2$, we can assume that $z_2$ is adjacent to $s_1$ in $\overline{H}$.  If $s_2$ did not have a neighbor in $\{z_1,z_3\}$ in $\overline{H}$,
then by contracting the edge $s_2z_2$ in $H$, we would obtain $K_4$ as a rooted minor of $\rootsrem{H}{Z}{z_4}$, a contradiction.
Since $s_1$ and $s_2$ have degree at most two in $\overline{H}$, it follows that each of them has at most one neighbor in $\{z_4,z_5\}$ in $\overline{H}$,
and if at least one of them has one, then $s_1s_2\not\in E(\overline{H})$.  By symmetry between $z_4$ and $z_5$, we can assume that $s_1z_4,s_2z_5\not\in E(\overline{H})$.
However, then contracting the edges $s_1z_4$ and $s_2z_5$ in $H$ gives us $K_4$ as a rooted minor of $\rootsrem{H}{Z}{z_2}$, which is a contradiction.
\end{itemize}
In all cases, we obtained a contradiction, and thus $\deg^+ v\ge 10$.

Suppose now for a contradiction that there exists a vertex $v\in V(G)\setminus X_G$ of degree seven such that
$\deg^+v = 10$ (or equivalently, $|N(v)\cap X_G|=3$) and that $\rho_4(G)\ge 6$.
By Corollary~\ref{cor-nume}, a flaw $F$ of $G$ does not belong to the $5$-target $\SS_{5,8}$,
and thus $F$ has at most one non-edge.   Consider any vertex $u\in N(v)\cap X_G$;
then $u$ has at least one neighbor in $F$ belonging to $N(v)\cap X_G$.
Let $G'$ be the $\id$-rooted minor of $G$ obtained by contracting the edge $uv$ and let $F'$ be the subgraph of $F$
obtained by deleting an edge from $u$ to a vertex in $N(v)\cap X_G\setminus\{u\}$.
The rooted graph $G'$ is $4$-light by Corollary~\ref{cor-cancontr}, and thus Observation~\ref{obs-nomg} implies that $F'$ does not belong to the target of $G'$.
Note that $\rho_4(G')=\rho_4(G)+3-t(uv)\le \rho_4(G)-1$.
Since $\rho_4(G)\ge 6$, Observation~\ref{obs-trade} implies that $\rho_4(G')\neq \rho_4(G)-1$,
and thus $t(uv)\ge 5$.

As before, let $H=N^{\overline{F}}_v$, and note that for each vertex $u\in N(v)\cap X_G$, the inequality $t(uv)\ge 5$ means that $u$ has at least three neighbors not in $X_G$.
Equivalently, in $\overline{H}$, the vertex $u$ has at most one neighbor in $N(v)\setminus X_G$.
Moreover, since $\delta(N^+_v)\ge 4$, each vertex in $N(v)\setminus X_G$ has degree at least four in $H$ and at most two in $\overline{H}$.
Let $N(v)\cap X_G=\{u_1,u_2,u_3\}$.  

Note that at most one vertex of $N(v)$ is not $v$-exposed.  Indeed, suppose for a contradiction
that $q_1,q_2\in N(v)$ are distinct non-$v$-exposed vertices (clearly not belonging to $X_G$).  Let $K$ be the vertex set of the component of $G-(N(v)\setminus \{q_1,q_2\})$
containing $q_1$, $q_2$, and $v$.  Since neither $q_1$ nor $q_2$ is $v$-exposed, we have $K\cap X_G=\emptyset$, and thus
$(C,D)=(V(G)\setminus K,K\cup (N(v)\setminus \{q_1,q_2\}))$ is a proper root separation of $G$ of order five
with $C\cap D=N(v)\setminus \{q_1,q_2\}$.  Since $v$ is adjacent to all vertices of $C\cap D$, this separation is $K_{1,\star}$-universal.
Moreover, since $|C\cap D\cap X_G|=|N(v)\cap X_G|=3$ and each vertex of $N(v)\cap X_G$ has at most one non-neighbor in $N(v)\setminus X_G$,
the root separation $(C,D)$ is $1$-nearly-saturated.  This contradicts Corollary~\ref{cor-noK5minus}.

By Corollary~\ref{cor-cantcreateclique}, there exists an $(F,v)$-poor set $Z\subseteq V(H)$;
and moreover, if there exists a $v$-exposed vertex in $N(v)\setminus X_G$ with a non-neighbor in $N(v)\cap X_G$,
then we can assume that $Z$ contains one.  Let $N(v)\setminus X_G=\{s_1,s_2,z_1,z_2\}$, where $z_1,z_2\in Z$.
Let us now consider the possibilities for the graph $\overline{H}[Z]$.
\begin{itemize}
\item Suppose first that $z_1$ or $z_2$ has two neighbors in $\{u_1,u_2,u_3\}$ in $\overline{H}$.
By symmetry, we can assume that $z_1u_1,z_1u_2\in E(\overline{H})$.  Note that $\overline{H}$ has no edges between $\{s_1,s_2\}$
and $\{u_1, u_2, z_1\}$.

If $z_2u_3\in E(\overline{H})$, then we also have $s_1u_3,s_2u_3\not\in E(\overline{H})$, and by contracting the edges $s_1u_1$ and $s_2u_2$ in $H$, we obtain $K_4$ as a rooted minor of $\rootsrem{H}{Z}{z_2}$.
This is a contradiction, since $Z$ is $(F,v)$-poor.  Hence, $z_2$ has no neighbor in $\{u_1,u_2,u_3\}$ in $\overline{H}$.
By symmetry between $s_1$ and $s_2$, we can assume that $s_2u_3\not\in E(\overline{H})$.  However, then we can contract the edges $s_1u_1$ and $s_2u_3$ and obtain $K_4$ as a rooted minor
of $\rootsrem{H}{Z}{z_1}$, a contradiction.

Therefore, each of $z_1$ and $z_2$ has at most one neighbor in $\{u_1,u_2,u_3\}$ in $\overline{H}$.

\item Suppose next that $z_1$ and $z_2$ each have a neighbor in $\{u_1,u_2,u_3\}$ in $\overline{H}$.
By symmetry, we can assume that $z_1u_1,z_2u_2\in E(\overline{H})$.  Since $z_1$ and $z_2$
have degree at most two in $\overline{H}$ and $u_3$ has at most one neighbor in $\overline{H}$ not in $X_G$,
each of $z_1$, $z_2$, and $u_3$ has at most one neighbor in $\{s_1,s_2\}$ in $\overline{H}$.
By symmetry between $s_1$ and $s_2$, we can assume that $s_1$ is adjacent to at most one of $z_1$ and $z_2$ in $\overline{H}$
and that $s_2u_3\not\in E(\overline{H})$.  By symmetry between $z_1$ and $z_2$, we can furthermore assume that $s_1z_1\not\in E(\overline{H})$.
By contracting the edges $s_1u_1$ and $s_2u_3$ in $H$, we obtain $K_4$ as a rooted minor of $\rootsrem{H}{Z}{z_2}$.
This is a contradiction, since $Z$ is $(F,v)$-poor.

Therefore, $\overline{H}$ has at most one edge between $\{z_1,z_2\}$ and $\{u_1,u_2,u_3\}$.

\item Next, let us consider the case that $\overline{H}$ has exactly one edge between $\{z_1,z_2\}$ and $\{u_1,u_2,u_3\}$;
by symmetry, we can assume that $u_1z_1\in E(\overline{H})$.

Suppose first that $s_1$ has no neighbor in $\{u_2,u_3\}$ in $\overline{H}$.  If $s_2u_i\not\in E(\overline{H})$ for some $i\in \{2,3\}$,
then we can contract the edges $s_1u_{5-i}$ and $s_2u_i$ in $H$ and obtain $K_4$ as a rooted minor of $\rootsrem{H}{Z}{z_1}$, a contradiction.
Hence, $s_2u_2,s_2u_3\in E(\overline{H})$, and since $s_2$ has degree at most two in $\overline{H}$, it follows that $s_2$
is adjacent neither to $z_1$ nor to $z_2$ in $\overline{H}$.  Hence, by contracting the edges $s_1u_2$ and $s_2z_1$ of $H$,
we obtain $K_4$ as a rooted minor of $\rootsrem{H}{Z}{u_1}$, which is again a contradiction.

Therefore, $s_1$ has a neighbor in $\{u_2,u_3\}$ in $\overline{H}$, and by a symmetric argument, so does $s_2$.
Since each of $u_2$ and $u_3$ has at most one neighbor not in $X_G$ in $\overline{H}$, we can
assume that $s_1u_2, s_2u_3\in E(\overline{H})$.  Since $z_1$ has degree at most two in $\overline{H}$,
it is adjacent in $\overline{H}$ to at most one of $s_1$ and $s_2$; by symmetry, we can assume that $s_1z_1\not\in E(\overline{H})$.

If $z_1z_2\not\in E(\overline{H})$, then we can contract the edge $s_1u_1$ of $H$ and obtain $K_4$ as a rooted minor of
$\rootsrem{H}{Z}{u_2}$, a contradiction.  Therefore, we have $z_1z_2\in E(\overline{H})$.  Then we also have $s_2z_1\not\in E(\overline{H})$,
restoring the symmetry between $s_1$ and $s_2$.  Moreover, $z_2$ has at most one neighbor in $\{s_1,s_2\}$ in $\overline{H}$,
and by this restored symmetry, we can assume that $s_1z_2\not\in E(\overline{H})$.  But then we can contract in $H$
the edges $s_1z_1$ and $s_2u_1$ and obtain $K_4$ as a rooted minor of $\rootsrem{H}{Z}{u_3}$.  This is a contradiction, since $Z$ is $(F,v)$-poor.

\item Therefore, $\overline{H}$ has no edges between $\{u_1,u_2,u_3\}$ and $\{z_1,z_2\}$.
Recall that $Z$ was chosen so that if there exists a $v$-exposed vertex in $N(v)\setminus X_G$ with a non-neighbor in $N(v)\cap X_G$ in $G$ (or equivalently, in $H$),
then such a vertex is contained in $Z$.  Since both $z_1$ and $z_2$ are adjacent to all vertices of $\{u_1,u_2,u_3\}$ in $H$,
it follows that every $v$-exposed vertex in $N(v)\setminus X_G$ is also adjacent to all vertices of $\{u_1,u_2,u_3\}$ in $H$.
Moreover, recall that at most one vertex in $N(v)$ is not $v$-exposed, and thus by symmetry between $s_1$ and $s_2$,
we can assume that $s_1$ is $v$-exposed.  Consequently, $\overline{H}$ has no edges between $s_1$ and $\{u_1,u_2,u_3\}$.

If $z_1z_2\not\in E(\overline{H})$, then by contracting the edge $s_1u_1$ in $H$, we would obtain $K_4$ as a rooted minor of $\rootsrem{H}{Z}{u_3}$, a contradiction.
Hence, we have $z_1z_2\in E(\overline{H})$.  Similarly, $s_2$ is adjacent to at least one of $z_1$ and $z_2$ in $\overline{H}$, as otherwise we could
contract the edges $s_1u_1$ and $s_2z_1$ in $H$ and obtain $K_4$ as a rooted minor of $\rootsrem{H}{Z}{u_3}$.  Since $s_2$ has degree at most two in $\overline{H}$,
it has at most one neighbor in $\{u_1,u_2,u_3\}$ in $\overline{H}$.  By symmetry, we can assume that $s_2u_2,s_2u_3\not\in E(\overline{H})$.
However, then we can contract the edges $s_1u_1$ and $s_2u_2$ in $H$ and obtain $K_4$ as a rooted minor of $\rootsrem{H}{Z}{z_2}$.
This is a contradiction, since $Z$ is $(F,v)$-poor.
\end{itemize}
In all cases, we obtained a contradiction, and thus if $G$ has a non-root vertex $v$ of degree seven such that $\deg^+ v=10$, then $\rho_4(G)\le 5$.
\end{proof}

Finally, let us perform a similar case analysis at vertices of degree eight.

\begin{lemma}\label{lemma-bo8}
If $G$ is a minimal counterexample, then every vertex $v\in V(G)\setminus X_G$ of degree eight satisfies $\deg^+ v\ge 10$.
\end{lemma}
\begin{proof}
Let us for a contradiction consider such a vertex $v$ for which $\deg^+ v\le 9$, or equivalently $|X_G\cap N(v)|\le 1$.
Thus, $N^+_v=G[N(v)]$, and $N^{\overline{F}}_v=N^+_v$ for every flaw $F$ of $G$.  Let $H=N^+_v$.
By Observation~\ref{obs-deg4}, the graph $H=N^+_v$ has minimum degree at least four, and thus its complement $\overline{H}$ has maximum degree at most three.
By Corollary~\ref{cor-cantcreateclique}, there exists an $(F,v)$-poor set $Z\subseteq V(H)$; let $N(v)\setminus Z=\{s_1,s_2,s_3\}$.
Let us now consider the possibilities for the graph $H[Z]$.

\begin{itemize}
\item Suppose first that there exists a set $Q\subset Z$ of size $4$ such that $\overline{H}[Q]$ has at least five edges.  If $\overline{H}[Q]$ is
not a clique, then let $q_1q_2$ be the unique edge of $H[Q]$, otherwise let $q_1$ and $q_2$ be arbitrary distinct vertices of $Q$.  Let $\{q_3,q_4\}=Q\setminus\{q_1,q_2\}$ and let $z$ be
the vertex in $Z\setminus Q$.
Since $\overline{H}$ has maximum degree at most three, $\overline{H}$ has no edges between $\{q_3,q_4\}$ and $N(v)\setminus Q$,
and each of $q_1$ and $q_2$ has at most one neighbor in $N(v)\setminus Q$ in $\overline{H}$.  In particular, we can choose the labels of $s_1$, $s_2$, and $s_3$
so that $s_1q_1,s_1q_2, s_2q_2\not\in E(\overline{H})$.  However, then we can contract the edges $s_1q_1$, $s_2q_2$, and $s_3q_3$ of $H$ and obtain $K_4$ as a rooted minor of $\rootsrem{H}{Z}{z}$.
This is a contradiction, since $Z$ is $(F,v)$-poor.

Therefore, $\overline{H}[Z]$ has no proper induced subgraph with more than four edges.
\item Suppose next that $\overline{H}[Z]$ contains a 4-cycle $C=z_1z_2z_3z_4$.  Since $\overline{H}[Z]$ has no proper induced subgraph with more than four edges,
this cycle is induced and the vertex $z\in Z\setminus V(C)$ has at most two neighbors in $V(C)$ in $\overline{H}$.

Since $\overline{H}$ has maximum degree at most three, each vertex of $C$ has at most one neighbor in $\{s_1,s_2,s_3\}$ in $\overline{H}$.
Thus, there exist distinct $i,j\in\{1,2,3\}$ such that $s_iz_2,s_jz_2\not\in E(\overline{H})$, and each of $z_1$ and $z_3$ is adjacent to at most one of $s_i$ and $s_j$ in $\overline{H}$.
If $z$ did not have any neighbor in $\{z_1,z_2,z_3\}$ in $\overline{H}$, then we could contract the edges $s_iz_2$ and $s_jz_2$ in $H$ and obtain $K_4$ as a rooted minor of $\rootsrem{H}{Z}{z_4}$, a contradiction.
Hence, there exists $a\in\{1,2,3\}$ such that $z_az\in E(\overline{H})$.  By a symmetric argument, there also exists $b\in\{1,2,3,4\}\setminus\{a\}$ such that $z_bz\in E(\overline{H})$.

Since $\overline{H}$ has maximum degree at most three, the vertices $z_a$ and $z_b$ do not have any neighbors in $\{s_1,s_2,s_3\}$ in $\overline{H}$.
Let $\{c,d\}=\{1,2,3,4\}\setminus\{a,b\}$.  Since $z_c$ and $z_d$ each have at most one neighbor in $\{s_1,s_2,s_3\}$ in $\overline{H}$,
we can assume that $s_1z_c,s_1z_d,s_2z_d\not\in E(\overline{H})$.  But then by contracting the edges $s_1z_c$, $s_2z_d$, and $s_3z_a$ of $H$,
we obtain $K_4$ as a rooted minor of $\rootsrem{H}{Z}{z}$, which is a contradiction.

Therefore, the graph $\overline{H}[Z]$ does not have any 4-cycle.
\item Suppose that $\overline{H}[Z]$ contains a 5-cycle $C=z_1\ldots z_5$, necessarily an induced one.
Since $\overline{H}$ has maximum degree at most three, each vertex of $Z$ has at most one neighbor in $\{s_1,s_2,s_3\}$ in $\overline{H}$,
and in particular $\overline{H}$ has at most five edges between $Z$ and $\{s_1,s_2,s_3\}$.  By symmetry, we can assume that $s_1$ has at most one neighbor in $Z$,
and we can choose the labels of the vertices of $C$ so that $s_1$ does not have any neighbors in the path $P=z_5z_1z_2$ in $\overline{H}$.

If there existed $i\in\{2,3\}$ and $j\in\{2,4\}$ such that $s_i$ does not have any neighbor in $\{z_j,z_{j+1}\}$,
then we could contract the edges $s_1z_1$ and $s_iz_j$ of $H$ and obtain $K_4$ as a rooted minor of $\rootsrem{H}{Z}{z}$ for the unique vertex
$z\in Z\setminus (\{z_1,z_2,z_5\}\cup\{z_j,z_{j+1}\})$, which is a contradiction.  Therefore, $s_2$ and $s_3$ each have a neighbor both in $\{z_2,z_3\}$
and in $\{z_4,z_5\}$.  Since each vertex of $Z$ has at most one neighbor in $\{s_1,s_2,s_3\}$ in $\overline{H}$, it follows that
$s_1$ cannot have a neighbor in $\{z_2,\ldots,z_5\}$, and since $s_1z_1\not\in E(\overline{H})$, it follows that $s_1$ does not have any neighbor in $Z$ in $\overline{H}$.

However, then an analogous argument for other 3-vertex subpaths of $C$ shows that for every edge $xy\in E(C)$, both $s_2$ and $s_3$ have a neighbor in $\{x,y\}$ in $\overline{H}$.
This is a contradiction, since $C$ is not $2$-colorable.  Therefore, $\overline{H}[Z]$ does not contain any 5-cycle.

\item Suppose now that $\overline{H}[Z]$ contains a triangle $C=z_1z_2z_3$.  Since $\overline{H}[Z]$ does not contain $4$-cycles, each of the vertices $z_4,z_5\in Z\setminus V(C)$
has at most one neighbor in $V(C)$ in $\overline{H}$.  By symmetry, we can assume that $z_3z_4,z_3z_5\not\in E(\overline{H})$.

Let us first consider the case that there exists $i\in\{4,5\}$ such that $z_i$ has a neighbor in $V(C)$ in $\overline{H}$ and that $z_4z_5\not\in E(\overline{H})$.  By symmetry, we can assume that
$z_1z_i\in E(\overline{H})$.
Since $\overline{H}$ has maximum degree at most three, $z_i$ has at most two neighbors in $\{s_1,s_2,s_3\}$ in $\overline{H}$ and $z_3$ has at most one.
By symmetry, we can assume that $s_1z_i,s_2z_3\not\in E(\overline{H})$.  By contracting the edges $s_1z_1$ and $s_2z_1$ of $H$,
we obtain $K_4$ as a rooted minor of $\rootsrem{H}{Z}{z_2}$, which is a contradiction.

Therefore, if $z_4$ or $z_5$ has a neighbor in $V(C)$ in $\overline{H}$, then $z_4z_5\in E(\overline{H})$.
Since $\overline{H}[Z]$ does not contain a 4-cycle, it follows that $z_4$ and $z_5$ cannot both have a neighbor in $V(C)$ in $\overline{H}$.
By symmetry, we can assume that $z_4$ does not have any neighbor in $V(C)$.

Since $\overline{H}$ has maximum degree at most three, each vertex of $C$ has at most one neighbor in $\{s_1,s_2,s_3\}$ in $\overline{H}$.
If at least one of $s_1$, $s_2$, and $s_3$ has more than one neighbor in $C$ in $\overline{H}$, then we can by symmetry assume
that $\overline{H}$ does not have any edge except possibly for $s_2z_1$ between $\{s_1,s_2\}$ and $V(C)$.
Then we can contract the edges $s_1z_1$ and $s_2z_2$ of $H$ and obtain $K_4$ as a rooted minor of $\rootsrem{H}{Z}{z_5}$,
which is a contradiction.

Therefore, each of $s_1$, \ldots, $s_3$ has at most one neighbor in $C$, and by symmetry, we can assume that
$\overline{H}$ does not have any edges between $V(C)$ and $\{s_1,s_2,s_3\}$ other than (possibly) $s_1z_1$, $s_2z_2$, and $s_3z_3$.
In this case, we can contract the edges $s_1z_2$, $s_2z_3$, and $s_3z_1$ of $H$ and obtain $K_4$ as a rooted minor of $\rootsrem{H}{Z}{z_5}$.
This is a contradiction, and thus $\overline{H}[Z]$ is a forest.

\item Suppose that $\overline{H[Z]}$ contains a vertex $z$ of degree three, with neighbors $z_1$, $z_2$, and $z_3$, and let $z'$ be the vertex in $Z\setminus\{z,z_1,z_2,z_3\}$.
Since $\overline{H}$ has maximum degree at most three, $z$ has no neighbors in $\{s_1,s_2,s_3\}$ and each of $z_1$, $z_2$, and $z_3$ has a non-neighbor in
$\{s_1,s_2,s_3\}$ in $\overline{H}$.  Therefore, by contracting the edges $s_1z$, $s_2z$, and $s_3z$ of $H$, we obtain $K_4$ as a rooted minor of $\rootsrem{H}{Z}{z'}$.
This is a contradiction, and thus $\overline{H[Z]}$ has maximum degree at most two.

\item Let $P=z_1\ldots z_k$ be a component of $\overline{H}[Z]$ with the most vertices, and suppose that $k\ge 4$.
Since $\overline{H}$ has maximum degree at most three, we can assume that $z_2$ has no neighbor in $\{s_1,s_2,s_3\}$
other than possibly $s_3$.  Moreover, $z_3$ has a non-neighbor in $\{s_1,s_2\}$ in $\overline{H}$.
If $z_1$ had a non-neighbor in $\{s_1,s_2\}$, then by contracting the edges $s_1z_2$ and $s_2z_2$ in $H$,
we would obtain $K_4$ as a rooted minor of $\rootsrem{H}{Z}{z_4}$, a contradiction.
Therefore, $z_1$ is adjacent to both $s_1$ and $s_2$ in $\overline{H}$, and non-adjacent to $s_3$ since $\overline{H}$ has maximum degree at most three.
Moreover, $s_3z_2\in E(\overline{H})$, as otherwise we could obtain $K_4$ as a rooted minor of $\rootsrem{H}{Z}{z_4}$ by contracting the edges $s_1z_2$, $s_2z_2$, and $s_3z_2$ of $H$;
and $s_1s_3,s_2s_3\in E(\overline{H})$, as otherwise we could obtain $K_4$ as a rooted minor of $\rootsrem{H}{Z}{z_4}$ by contracting the edges $s_1z_2$, $s_2z_2$, and $s_3z_1$ of $H$.

Since $\overline{H}$ has maximum degree at most three, we have $s_3z_{k-1},s_3z_k\not\in E(\overline{H})$,
and if $k=5$, then $s_3z_{k-2}\not\in E(\overline{H})$.  Therefore, we obtain $K_4$ as a rooted minor of $\rootsrem{H}{Z}{z_2}$ by contracting the edge $s_3z_4$ of $H$.
This is a contradiction, and thus each component of $\overline{H}[Z]$ has at most three vertices.

\item Therefore, $\overline{H}[Z]$ is a subgraph of the disjoint union of paths $z_1z_2$ and $z_3z_4z_5$.
Note that $z_1z_2\in E(\overline{H}[Z])$, as otherwise $H[Z]-z_4$ is a clique, a contradiction since $Z$ is $(F,v)$-poor.
Since $\overline{H}$ has maximum degree at most three, each of $z_1$ and $z_2$ has a non-neighbor in $\{s_1,s_2,s_3\}$ in $\overline{H}$,
i.e., a neighbor in $\{s_1,s_2,s_3\}$ in $H$.  If $z_1$ and $z_2$ have neighbors in $H$ in the same component $K$ of $H[\{s_1,s_2,s_3\}]$, then we can contract this component $K$
to a single vertex $s$, then contract the edge $sz_1$ and obtain $K_4$ as a rooted minor of $\rootsrem{H}{Z}{z_4}$.
This is a contradiction, and thus neighbors of $z_1$ and $z_2$ in $H$ belong to different components of $H[\{s_1,s_2,s_3\}]$.

By symmetry, we can assume that $s_1z_1,s_2z_2\in E(H)$, $s_1z_2,s_2z_1\in E(\overline{H})$, and since $s_1$ and $s_2$ belong to different components of $H[\{s_1,s_2,s_3\}]$,
that $s_1s_2,s_2s_3\in E(\overline{H})$.  But since $\overline{H}$ has maximum degree at most three, it follows that $s_2$ has no neighbor in $\{z_3,z_4,z_5\}$ in $\overline{H}$.
Then we can contract the edge $s_2z_4$ of $H$ and obtain $K_4$ as a rooted minor of $\rootsrem{H}{Z}{z_2}$.
This is a contradiction, since $Z$ is $(F,v)$-poor.
\end{itemize}

In all cases, we have obtained a contradiction, and thus $\deg^+ v\ge 10$.
\end{proof}

For a rooted graph $G$, a vertex $v\in V(G)\setminus X_G$, and a positive integer $t$, let us define $\rho^G_t(v)=\tfrac{1}{2}\deg^+ v-t$,
so that
\begin{equation}\label{eq-sumwt}
\rho_t(G)=\sum_{v\in V(G)\setminus X_G} \rho^G_t(v)
\end{equation}
by (\ref{eq-rhot}).  When $G$ is clear from the context, then we write just $\rho_t(v)$ instead of $\rho^G_t(v)$.
We can summarize Lemmas~\ref{lemma-bo6}, \ref{lemma-bo7}, and \ref{lemma-bo8} as follows.
\begin{corollary}\label{cor-bo}
Every non-root vertex $v$ of a minimal counterexample $G$ satisfies $\rho_4(v)\ge \tfrac{1}{2}$.
Moreover, $\rho_4(v)\ge 1$ unless $\deg v=9$ and $v$ has no neighbors in $X_G$.
\end{corollary}

We can now bound the size of a minimal counterexample.
\begin{corollary}\label{cor-size}
If $G$ is a minimal counterexample, then $G$ is internally $6$-connected, $n(G)\le \rho_4(G)\le 7$ and every vertex $v\in V(G)\setminus X_G$
has a neighbor in $X_G$ and satisfies $\rho_4(v)\ge 1$.
\end{corollary}
\begin{proof}
Suppose first that $G$ has a proper root separation $(C,D)$ of order at most five, and let $R=R^G_{C,D}$.
Corollary~\ref{cor-nonear-bistar} and the fact that $G$ is $4$-light imply that $\rho_4(R)\le 1$.
Note that for each vertex $v\in D\setminus C$, we have $\rho^R_4(v)\ge \rho^G_4(v)\ge \tfrac{1}{2}$ by Corollary~\ref{cor-bo}, since all neighbors of $v$ in $X_G$ belong to $C\cap D$.
By (\ref{eq-sumwt}), it follows that $|D\setminus C|\le 2$, and in particular the vertices in $D\setminus C$ have degree at most six.
Corollary~\ref{cor-bo} then implies that $\rho^R_4(v)\ge \rho^G_4(v)\ge 1$ for every $v\in D\setminus C$.
By (\ref{eq-sumwt}), it follows that $D\setminus C$ consists of a single vertex $v$.  However, then $v$ has degree at most five,
contradicting Lemma~\ref{lemma-no5}.  Therefore the rooted graph $G$ is internally $6$-connected.

Let $S$ be the set of neighbors of vertices in $X_G$, and let $R=V(G)\setminus (X_G\cup S)$.
By Corollary~\ref{cor-bo}, we have $\rho_4(u)\ge 1$ for every $u\in S$ and $\rho_4(v)\ge\tfrac{1}{2}$ for every $v\in R$.
In particular, by Corollary~\ref{cor-nume} and (\ref{eq-sumwt}) we have
\begin{equation}\label{eq-rs}
|S|+\tfrac{1}{2}|R|\le \rho_4(G)\le 7.
\end{equation}
If $R\neq\emptyset$, then since $G$ is internally $6$-connected, we have $|S|\ge 6$,
and (\ref{eq-rs}) gives $|S|=6$ and $|R|\le 2$.  However, then each vertex $v\in R$ has degree at most $|S|+|R|-1\le 7$,
and by Corollary~\ref{cor-bo}, we have $\rho_4(v)\ge 1$.  By (\ref{eq-sumwt}), we then get $|R|=1$ and $\rho_4(G)=7$.
However, then the vertex $v\in R$ has degree six and $\rho_4(G)>5$, which contradicts Lemma~\ref{lemma-bo6}.

Therefore, we have $R=\emptyset$, and thus each vertex in $V(G)\setminus X_G$ has a neighbor in $X_G$.
Moreover, (\ref{eq-rs}) gives $n(G)=|S|\le \rho_4(G)\le 7$.
\end{proof}

At this point, proving Theorem~\ref{thm-mainplus} is in principle just a matter of inspecting the finitely many 5-rooted
graphs $G$ with $n(G)\le 7$.  However, there are quite many such graphs, and thus it will be convenient to further
constrain the counterexamples.  The bound on $n(G)$ gives rise to a bound on $\rho_4(v)$ for non-root vertices $v$,
by the following observation.
\begin{observation}\label{obs-fewtoheavy}
For every rooted graph $G$ and every positive integer $t$, every vertex $v\in V(G)\setminus X_G$ satisfies
$$\rho_t(v)\ge \deg v - \tfrac{1}{2}(n(G) - 1)-t,$$
and the equality holds if and only if $v$ is adjacent to all other vertices in $V(G)\setminus X_G$.
\end{observation}
\begin{proof}
Note that $\deg^X v\ge \deg v - (n(G)-1)$, with equality if and only if $v$ is adjacent to all other vertices in $V(G)\setminus X_G$.
Consequently
$$\rho_t(v)=\tfrac{1}{2}\deg^+v-t=\tfrac{1}{2}(\deg v+\deg^X v)-t\ge \deg v - \tfrac{1}{2}(n(G) - 1)-t,$$
with equality if and only if $v$ is adjacent to all other vertices in $V(G)\setminus X_G$.
\end{proof}

Let us now improve the bound from Corollary~\ref{cor-size} a bit.
\begin{lemma}\label{lemma-s5}
If $G$ is a minimal counterexample, then $n(G)\le 5$.
\end{lemma}
\begin{proof}
Suppose for a contradiction that $n(G)\in\{6,7\}$.  By Corollary~\ref{cor-size}, we have $\rho_4(G)\in\{6,7\}$.
By Corollary~\ref{cor-size}, we have $\rho_4(v)\ge 1$ for every $v\in V(G)\setminus X_G$.
Moreover, note that if $\rho_4(v)>1$, then $\rho_4(v)\ge \tfrac{3}{2}$.
Since $n(G)\ge 6$ and $\rho_4(G)\le 7$, (\ref{eq-sumwt}) implies that the set $K=\{u\in V(G)\setminus X_G:\rho_4(u)=1\}$
has size at least four.  

Consider any vertex $u\in K$; note that $\rho_4(u)=1$ is equivalent to $\deg^+ u=10$.
Since $\rho_4(G)>5$, by Lemmas~\ref{lemma-bo6} and \ref{lemma-bo7} this is only possible if $\deg u\ge 8$.
On the other hand, Observation~\ref{obs-fewtoheavy} gives
$$1\ge \deg u-\tfrac{1}{2}(n(G) - 1)-4\ge \deg u-7.$$
This is only possible if $\deg u=8$ and $n(G)=7$, in which case $\rho_4(u)=\deg u - \tfrac{1}{2}(n(G) - 1)-4$,
and by Observation~\ref{obs-fewtoheavy}, $u$ is adjacent to all other vertices in $V(G)\setminus X_G$.

Therefore, $K$ forms a clique in $G$ and the vertices of $K$ are adjacent to all vertices in $V(G)\setminus (K\cup X_G)$.
Consequently, we have $\omega(G)\ge 5$, which contradicts Corollary~\ref{cor-omega4}.
\end{proof}

For a rooted graph $G$ and a vertex $v\in V(G)\setminus X_G$, let $X^G_v$ denote the set of neighbors of $v$ in $X_G$.
When the rooted graph $G$ is clear from the context, we use just $X_v$ instead of $X^G_v$.
We are going to need the following observation on the neighborhoods of vertices of small degree
in a minimal counterexample.

\begin{observation}\label{obs-deg67}
Let $G$ be a minimal counterexample, let $u\in V(G)\setminus X_G$ be a vertex of degree six such that $\rho_4(u)=1$ (or equivalently, $\deg^X u=4$),
and let $U$ be the set of the two neighbors of $u$ in $V(G)\setminus X_G$.
Each vertex $v\in U$ satisfies one of the following conditions:
\begin{itemize}
\item[(i)] $\rho_4(v)\ge \tfrac{3}{2}$; or
\item[(ii)] $\rho_4(v)=1$, $\deg v=6$, and either $v$ is adjacent to the vertex in $U\setminus\{v\}$ or $X_v=X_u$; or
\item[(iii)] $\rho_4(v)=1$, $\deg v=7$, $v$ is adjacent to the vertex in $U\setminus\{v\}$, and $X_v\subsetneq X_u$.
\end{itemize}
\end{observation}
\begin{proof}
Let $U=\{v,w\}$.
Suppose that (i) is false; by Corollary~\ref{cor-size}, we have $\rho_4(v)=1$.
Note that $\deg v\ge 6$ by Lemma~\ref{lemma-no5}.  Moreover, Observation~\ref{obs-fewtoheavy} and Lemma~\ref{lemma-s5}
imply that
\begin{equation}\label{eq-deg67}
1\ge \deg v - \tfrac{1}{2}(n(G)-1)-4\ge \deg v-6,
\end{equation}
and thus $\deg v\le 7$.

Let $M$ be the set of common neighbors of $u$ and $v$ in $G$.
By Corollary~\ref{cor-fourtri}, we have $|M|=t(uv)\ge 4$.
Note that $M=X_u\cap X_v$ if $vw\not\in E(G)$ and $M=(X_u\cap X_v)\cup \{w\}$ if $vw\in E(G)$.

If $\deg v=6$, then $\rho_4(v)=1$ implies $|X_v|=4$.
If $vw\not\in E(G)$, then since $M=X_u\cap X_v$ has size at least four and $|X_u|=|X_v|=4$, we have $X_u=X_v$.
This gives the outcome (ii).

Hence, we can assume that $\deg v=7$. Then the inequalities in (\ref{eq-deg67}) hold
with equality, and Observation~\ref{obs-fewtoheavy} implies that $v$ is adjacent to all other vertices of $V(G)\setminus X_G$,
including the vertex $w$.  Note that $\rho_4(v)=1$ implies $|X_v|=3<|X_u|$.
Since $M=(X_u\cap X_v)\cup \{w\}$ has size at least four, it follows that $X_v\subsetneq X_u$.
Therefore, the outcome (iii) holds.
\end{proof}

We can now improve the bound from Lemma~\ref{lemma-s5}.

\begin{lemma}\label{lemma-s4}
If $G$ is a minimal counterexample, then $n(G)\le 4$.
\end{lemma}
\begin{proof}
By Lemma~\ref{lemma-s5}, we have $n(G)\le 5$.  Suppose for a contradiction that $n(G)=5$.

By Corollary~\ref{cor-nume}, we have $\rho_4(G)\le 7$, and since $5\cdot \tfrac{3}{2}>7$,
(\ref{eq-sumwt}) implies that there exists a vertex $v_0\in V(G)\setminus X_G$ such that
$\rho_4(v_0)<\tfrac{3}{2}$.  Hence, $\rho_4(v_0)\le 1$, and by Corollary~\ref{cor-size}, $\rho_4(v_0)=1$;
or equivalently, $\deg^+ v_0=10$.
Observation~\ref{obs-fewtoheavy} implies $\deg v_0\le 5+\tfrac{1}{2}(n(G)-1)=7$, and thus by Lemmas~\ref{lemma-bo6}
and \ref{lemma-bo7}, we have $\rho_4(G)\le 5$.

Corollary~\ref{cor-size} shows that every vertex $v\in V(G)\setminus X_G$ satisfies $\rho_4(v)\ge 1$.
By (\ref{eq-sumwt}), this implies that $\rho_4(G)=5$ and that $\rho_4(v)=1$ for every $v\in V(G)\setminus X_G$.
Observation~\ref{obs-fewtoheavy} then further implies that $\deg v\in \{6,7\}$ and that if $\deg v=7$,
then $v$ is adjacent to all other non-root vertices of $G$.

Let $V_6$ and $V_7$ be the sets of non-root vertices
of $G$ of degree six and seven, respectively; we have $|V_6|+|V_7|=n(G)=5$.
Since $\omega(G)\le 4$ by Corollary~\ref{cor-omega4}, we have $|V_7|\le 3$,
and thus there exists a vertex $u\in V_6$.  Since $\rho_4(u)=1$, we have $\deg^X u=4$, and
$u$ has exactly two neighbors in $V(G)\setminus X_G$.  All vertices in $V_7$ are adjacent
to $u$, and thus we actually have $|V_7|\le 2$ and $|V_6|\ge 3$.

Let us now distinguish several cases.
\begin{itemize}
\item If $|V_6|=5$ (and $V_7=\emptyset$), then since each vertex in $V_6$ has exactly two non-root
neighbors, the subgraph $G[V_6]$ is a $5$-cycle $u_1u_2u_3u_4u_5$.  By Observation~\ref{obs-deg67},
we have $X_{u_1}=\ldots=X_{u_5}$.  However, that means that one of the roots has no neighbor in $V(G)\setminus X_G$, which
contradicts Lemma~\ref{lemma-connwiro}.

\item If $|V_6|=4$, then $|V_7|=1$, the vertex $v\in V_7$ is adjacent to all vertices in $V_6$, and
Observation~\ref{obs-deg67} implies that $X_v\subset X_u$ for every $u\in V_6$.
Let $V_6=\{u_1,u_2,u_3,u_4\}$ and $X_G=\{x_1,x_2,x_3,x_4,x_5\}$, where $X_v=\{x_3,x_4,x_5\}$.
By Corollary~\ref{cor-twononroot}, $x_1$ and $x_2$ each have at least two neighbors in $V_6$,
and by symmetry, we can assume that $u_1x_1,u_2x_2\in E(G)$.
Therefore, we can contract the edges $u_1v$, $u_1x_1$, $u_2x_2$, $u_3x_3$, and $u_4x_4$
and obtain $K_5$ as a rooted minor of $G$ (indeed, the resulting minor contains all edges incident
with $x_1$, since $v$ is adjacent to all vertices of $V_6$ and since $x_5\in X_v$;
and all edges incident with $x_i$ for $i\in \{2,3,4\}$ since $\{x_3,x_4,x_5\}=X_v\subseteq X_{u_i}$).  This is a contradiction.

\item Finally, suppose that $|V_7|=2$ and $|V_6|=3$.  The vertices $v_1,v_2\in V_7$ are adjacent,
and they are also adjacent to all vertices in $V_6$.  Since each vertex in $V_6$ has exactly two non-root neighbors, $V_6$
is an independent set in $G$.  By Observation~\ref{obs-deg67}, we have $X_{v_1},X_{v_2}\subset X_u$ for every $u\in V_6$.
Since each root vertex has a non-root neighbor by Lemma~\ref{lemma-connwiro}, there exist vertices $u_1,u_2\in V_6$
such that $X_{u_1}\neq X_{u_2}$, and thus $X_{v_1}=X_{v_2}=X_{u_1}\cap X_{u_2}$.  Let $\{x_3,x_4,x_5\}=X_{v_1}$,
let $x_1$ and $x_2$ be the neighbors of $u_1$ and $u_2$, respectively, in $X_G\setminus X_{v_1}$, and
let $u_3$ be the vertex in $V_6\setminus \{u_1,u_2\}$.

We can now contract the edges $u_1v_1$, $u_1x_1$, $u_2x_2$, $u_3x_3$, and $v_2x_4$
and obtain $K_5$ as a rooted minor of $G$. This is again a contradiction.
\end{itemize}
In all cases we obtained a contradiction, and thus $n(G)\le 4$.
\end{proof}

We can even go one step further.
\begin{lemma}\label{lemma-s3}
If $G$ is a minimal counterexample, then $n(G)\le 3$.
\end{lemma}
\begin{proof}
By Lemma~\ref{lemma-s4}, we have $n(G)\le 4$.  Suppose for a contradiction that $n(G)=4$.

Consider any vertex $v\in V(G)\setminus X_G$.  By Corollary~\ref{cor-size}, we have 
$\rho_4(v)\ge 1$, or equivalently $\deg^+ v\ge 10$. 
Since $v$ has at most $n(G)-1=3$ non-root neighbors, this implies that $\deg^X v\ge 4$.
Hence, each non-root vertex has at most one non-neighbor in $X_G$.  It follows that there exists a root $x_5\in X_G$
adjacent to all vertices of $V(G)\setminus X_G$.

Consider any non-empty set $U\subseteq V(G)\setminus X_G$, and let $N_U$ be the set consisting of all vertices in $X_G\setminus\{x_5\}$
with a neighbor in $U$.  Since $\deg^X v\ge 4$ for any vertex $v\in U$, we have $|N_U|\ge 3$.  Moreover, if $U=V(G)\setminus X_G$,
then $|N_U|=4$ by Lemma~\ref{lemma-connwiro}.  Therefore, every such set $U$ satisfies $|N_U|\ge |U|$.
By Hall's theorem, $G$ contains a matching $M=\{x_1u_1, \ldots, x_4u_4\}$ between $X_G\setminus\{x_5\}$ and $V(G)\setminus X_G$.

Let $F$ be a flaw of $G$.  By contracting the edges of $M$, we do not obtain $F$ as an $\id$-rooted minor of $G$.
Hence, we can by symmetry assume that $x_1x_2\in E(F)$ and that $G$ does not contain any edges between
$\{u_1,x_1\}$ and $\{u_2,x_2\}$.  In particular, we have $X_{u_1}=X_G\setminus\{x_2\}$ and $X_{u_2}=X_G\setminus \{x_1\}$.
Moreover, since $u_1u_2\not\in E(G)$, it follows that for $i\in\{1,2\}$, the vertex $u_i$ has degree six and is
adjacent to $u_3$ and $u_4$.

Because $G$ contains non-root vertices of degree six, Lemma~\ref{lemma-bo6} implies that
$\rho_4(G)\le 5$.  By (\ref{eq-sumwt}), this gives $\rho_4(u_3)+\rho_4(u_4)\le 5-(\rho_4(u_1)+\rho_4(u_2))=3$,
and thus we can by symmetry assume that $\rho_4(u_3)\le \tfrac{3}{2}$, i.e., $\deg^+ u_3\le 11$.
Since $u_3$ has (at least) two non-root neighbors $u_1$ and $u_2$, this implies that $\deg^X u_3<5$, and thus $\deg^X u_3=4$.
If $u_3u_4\not\in E(G)$, then $\deg u_3=6$ and $\rho_4(u_3)=1$, and Observation~\ref{obs-deg67}
implies that $X_{u_1}=X_{u_3}=X_{u_2}$.  This is a contradiction, and thus $u_3u_4\in E(G)$.

Since $\deg^X u_3=4$, the vertex $u_3$ has a neighbor in $\{x_1,x_2\}$; by symmetry, we can assume that $x_1u_3\in E(G)$.
However, then we can consider the matching $M'=\{x_1u_3,x_2u_2,x_3u_1,x_4u_4\}$.  Contracting $M'$ gives $K_5$
as a rooted minor of $G$ (indeed, $x_5$ is adjacent to all non-root vertices, $u_1u_2$ is the only non-edge of $G-X_G$,
and $G$ has the edge $u_2x_3$ between the pairs $\{u_1,x_3\}$ and $\{u_2,x_2\}$).

This is a contradiction, and thus $n(G)\le 3$.
\end{proof}

The following lemma will be useful in dealing with the case that $n(G)=3$.

\begin{lemma}\label{lemma-k5from554}
Let $G$ be a 5-rooted graph which contains distinct vertices $u,v,w\in V(G)\setminus X_G$ 
such that $\deg^X u=\deg^X v=5$ and $\deg^X w\ge 4$.  Let $F$ be a graph with $V(F)=X_G$.
If $|E(F)|\le \deg^X w+4$, then $F$ is an $\id$-rooted minor of $G$.
\end{lemma}
\begin{proof}
Let $X_G=\{x_1,\ldots,x_5\}$.  Since $|E(F)|\le 9$, we can assume that $x_1x_2\not\in E(F)$;
and since $\deg^X w\ge 4$, we can assume that $x_5\in X_w$.

If $x_1,x_2\in X_w$, then (a supergraph of) $F$ can be obtained as an $\id$-rooted minor of $G$
by contracting the edges $ux_3$, $vx_4$, and $wx_5$ (the first two contractions ensure the presence
of all edges incident with $x_3$ and $x_4$, and the last one the presence of the edges $x_1x_5$ and $x_2x_5$).

Hence, we can by symmetry assume that $x_1\not\in X_w$.
Moreover, the same argument shows that all non-edges of $F$ are incident with $x_1$.
Since $x_1\not\in X_w$, we have $\deg^X w=4$, and thus $|E(F)|\le 8$.  Hence, we can
assume that $x_1x_3\not\in E(F)$.  But then we can contract the edges
$ux_4$, $vx_5$, and $wx_3$ and obtain (a supergraph of) $F$ as an $\id$-rooted minor of $G$.
\end{proof}

Moreover, note the following property of $\SS_{5,4}^-$.
\begin{observation}\label{obs-54minus}
The set $\SS_{5,4}^-$ consists exactly of all graphs in $\SS_{5,4}$ with a vertex cover of size at most two.
\end{observation}

Our main result now follows by a straightforward analysis of the remaining few cases.
\begin{proof}[Proof of Theorem~\ref{thm-mainplus}]
Suppose for a contradiction that Theorem~\ref{thm-mainplus} is false, and thus there exists a minimal counterexample $G$.
Let $V(G)\setminus X_G=\{v_1,\ldots,v_s\}$, where $s=n(G)\le 3$ by Lemma~\ref{lemma-s3}, and let $F$ be a flaw of $G$.

Lemma~\ref{lemma-no5} implies that $s\ge 2$.  Suppose that $s=2$; then Lemma~\ref{lemma-no5} moreover implies that
$v_1v_2\in E(G)$ and both $v_1$ and $v_2$ are adjacent to all vertices of $X_G$.  
Note that $\rho_4(G)=3$, and thus $F$ belongs to the $3$-target $\SS_{5,4}^-$. By Observation~\ref{obs-54minus},
there exist distinct vertices $x_1,x_2\in V(F)$ such that every edge of $F$ is incident with $x_1$ or $x_2$.
It follows that we can obtain a supergraph of $F$ as an $\id$-rooted minor of $G$ by contracting the edges $x_1v_1$ and $x_2v_2$.
This is a contradiction.

Therefore, we have $s=3$.  We can assume that $\deg^X v_1\le \deg^X v_2\le \deg^X v_3$.  By Lemma~\ref{lemma-no5},
we have $\deg v_1\ge 6$, and thus $\deg^X v_1\ge \deg v_1-(s-1)\ge 4$.

Suppose now that $\deg^X v_2=5$.  Note that $\rho(G)\le 3+\sum_{i=1}^3 \deg^X v_i=13+\deg^X v_1$ and $\rho_4(G)\le 1+\deg^X v_1\in \{5,6\}$.
Since the 5-target is $\SS_{5,8}$ and the 6-target is $\SS_{5,9}$, it follows that the target of $G$ is a subset of $\SS_{5,4+\deg^X v_1}$.
Therefore, we have $|E(F)|\le 4+\deg^X v_1$, and $F$ is an $\id$-rooted minor of $G$ by Lemma~\ref{lemma-k5from554}.
This is a contradiction, and thus $\deg^X v_2\le 4$.

Since $4\le \deg^X v_1\le \deg^X v_2$, it follows that $\deg^X v_1=\deg^X v_2=4$.
For $i\in \{1,2\}$, let $x_i$ be the unique non-neighbor of $v_i$ in $X_G$.
Since $\deg v_i\ge 6$ by Lemma~\ref{lemma-no5}, the vertex $v_i$ is adjacent to all vertices in $V(G)\setminus\{v_i,x_i\}$.
In particular, $G-X_G$ is a triangle, and thus
\begin{align*}
\rho(G)&=3+\sum_{i=1}^3 \deg^X v_i=11+\deg^X v_3\\
\rho_4(G)&=\deg^X v_3-1.
\end{align*}
Note also that by Corollary~\ref{cor-twononroot}, each vertex of $X_G$ has at most one non-neighbor in $\{v_1,v_2,v_3\}$,
and thus $x_1\neq x_2$.  

If $\deg^X v_3=5$, then $\rho_4(G)=4$, and $F$ belongs to the 4-target $\SS_{5,6}$, i.e., $|E(F)|\le 6$.
Let $\{y_1,y_2,y_3\}=V(F)\setminus \{x_1,x_2\}$.  If the graph $F-\{x_1,x_2\}$ is not the triangle, then
we can choose the labels of the vertices so that $y_1y_2\not\in E(F)$.  However, then we can obtain
a supergraph of $F$ as an $\id$-rooted minor of $G$ by contracting the edges $v_1x_2$, $v_2x_1$, and $v_3y_3$.
This is a contradiction, and thus $F-\{x_1,x_2\}$ is the triangle.  Since $|E(F)|\le 6$, there are at most three
edges in $F$ between $\{x_1,x_2\}$ and $\{y_1,y_2,y_3\}$, and by symmetry we can assume that $x_1y_2,x_1y_3\not\in E(F)$.
However, in this case we can obtain a supergraph of $F$ as an $\id$-rooted minor of $G$ by contracting the edges
$v_1y_2$, $v_2y_1$, and $v_3x_2$, again a contradiction.

Therefore, $\deg^X v_3=4$.  Let $x_3$ be the unique non-neighbor of $v_3$ in $X_G$.
Since each vertex of $X_G$ has at most one non-neighbor in $\{v_1,v_2,v_3\}$ by Corollary~\ref{cor-twononroot},
note that $x_3\not\in \{x_1,x_2\}$.  
Let $\{x_4,x_5\}=X_G\setminus\{x_1,x_2,x_3\}$.  If $x_1x_4,x_1x_5\not\in E(F)$, then we can obtain a supergraph of $F$ as an $\id$-rooted
minor of $G$ by contracting the edges $v_1x_5$, $v_2x_3$, and $v_3x_2$, a contradiction.  Therefore, $x_1$ has a neighbor in $\{x_4,x_5\}$ in $F$,
and by a symmetric argument, so do the vertices $x_2$ and $x_3$.
If $x_4x_5\not\in E(F)$, then we can obtain a supergraph of $F$ as an $\id$-rooted
minor of $G$ by contracting the edges $v_1x_2$, $v_2x_3$, and $v_3x_1$.
This is again a contradiction, and thus $x_4x_5\in E(F)$.

Since $\rho_4(G)=\deg^X v_3-1=3$, the flaw $F$ belongs to the 3-target $\SS_{5,4}^-$,
and thus $|E(F)|\le 4$.  Therefore, $E(F)$ consists of the edge $x_4x_5$ and of exactly one edge
from each of $x_1$, $x_2$, and $x_3$ to $\{x_4,x_5\}$.  By symmetry, we can assume that $x_1x_4\not\in E(F)$.
But then we can contract the edges $v_1x_4$, $v_2x_5$, and $v_3x_5$ and obtain a supergraph of $F$ as an $\id$-rooted minor of $G$.

In all cases, we have reached a contradiction, and thus Theorem~\ref{thm-mainplus} holds.
\end{proof}

\section*{Acknowledgement}

I would like to thank Sergey Norin for insightful discussions that guided
the basic approach taken in the paper.

\bibliographystyle{alpha}
\bibliography{main}

\end{document}